\documentclass[10pt]{amsart}
\makeatletter
\def\@tocline#1#2#3#4#5#6#7{\relax
  \ifnum #1>\c@tocdepth 
  \else
    \par \addpenalty\@secpenalty\addvspace{#2}%
    \begingroup \hyphenpenalty\@M
    \@ifempty{#4}{%
      \@tempdima\csname r@tocindent\number#1\endcsname\relax
    }{%
      \@tempdima#4\relax
    }%
    \parindent\z@ \leftskip#3\relax \advance\leftskip\@tempdima\relax
    \rightskip\@pnumwidth plus4em \parfillskip-\@pnumwidth
    #5\leavevmode\hskip-\@tempdima
      \ifcase #1
       \or\or \hskip 2em \or \hskip 3em \else \hskip 4em \fi%
      #6\nobreak\relax
    \dotfill\hbox to\@pnumwidth{\@tocpagenum{#7}}\par
    \nobreak
    \endgroup
  \fi}
\makeatother
\usepackage{graphicx}
\usepackage{amsthm}
\usepackage{float}
\usepackage{xcolor}
\usepackage{url}
\theoremstyle{plain}
\newtheorem{theorem}{Theorem}[section]
\newtheorem{corollary}[theorem]{Corollary}
\newtheorem{proposition}[theorem]{Proposition}
\newtheorem{lemma}[theorem]{Lemma}
\theoremstyle{definition}
\newtheorem{definition}[theorem]{Definition}

\newtheorem{remark}[theorem]{Remark}
\newcommand{\ctext}[1]{\raise0.2ex\hbox{\textcircled{\scriptsize{#1}}}}
\definecolor{refkey}{rgb}{0.9451,0.2706,0.4941}
\definecolor{labelkey}{rgb}{0.9451,0.2706,0.4941}
\usepackage[colorlinks=true,linkcolor=blue,citecolor=purple,urlcolor=magenta]{hyperref}
\usepackage{xurl}

\usepackage[
  backend=biber,
  style=alphabetic,   
  url = false,
  backref=true,       
  backrefstyle=two    
]{biblatex}
\DefineBibliographyStrings{english}{
  backrefpage  = {cited on p.},
  backrefpages = {cited on pp.},
}
\begin{document}
\title[]{A lower bound for the radius of Weinstein's Lagrangian tubular neighborhood}
\author[]{Hikaru Yamamoto}
\address{Department of Mathematics, Faculty of Pure and Applied Science, 
University of Tsukuba, 1-1-1 Tennodai, Tsukuba, Ibaraki 305-8571, Japan}
\email{\textcolor{magenta}{hyamamoto@math.tsukuba.ac.jp}}
\thanks{The first author was supported in part by 
JSPS KAKENHI Grant Number JP22K13909.}
\begin{abstract}
For an immersed Lagrangian submanifold $L$ in a K\"ahler manifold $(M,\omega)$, there exists a symplectic local diffeomorphism from a tubular neighborhood of the image of the zero section in the normal bundle $T^{\bot}L$ of $L$, equipped with a canonical symplectic form $\tilde{\omega}$, to $(M,\omega)$ whose restriction to $L$ is the identity map by Weinstein's Lagrangian tubular neighborhood theorem, where the image of the zero section in $T^{\bot}L$ is identified with $L$. 
In this paper, we give a lower bound for the supremum of the radii of tubular neighborhoods that have such a symplectic diffeomorphism into $(M,\omega)$ from below by a constant explicitly given in terms of up to second derivatives of the Riemannian curvature tensor of $M$ and the second fundamental form of $L$. 
We also give a similar lower bound in the case where $L$ is compact and embedded. 
\end{abstract}
\maketitle

\tableofcontents
\section{Introduction}\label{7gbvcscvwwf}
\subsection{Motivation}
In 1971, A. Weinstein proved in \cite{MR286137} a result concerning Lagrangian submanifolds in symplectic manifolds, which is now commonly referred to as the \emph{Weinstein Lagrangian tubular neighborhood theorem}. Although Weinstein's original theorem is formulated in a setting that also includes Banach manifolds, we restrict our attention to the finite-dimensional case and recall the following version. 

\begin{theorem}
Let $L$ be a compact Lagrangian submanifold in a symplectic manifold $(M,\omega)$. Then, there exist an open neighborhood $\mathcal{U}$ of $L$ in $T^{\ast}L$, where $L$ is identified with the image of the zero-section of $T^{\ast}L$, 
an open neighborhood $\mathcal{V}$ of $L$ in $M$ and a diffeomorphism 
$\Theta : \mathcal{U} \to \mathcal{V}$ such that $\Theta(p)=p$ for all $p\in L$ and $\Theta^{\ast}\omega=\tilde{\omega}$, where $\tilde{\omega}$ is the canonical symplectic form on $T^{\ast}L$.  
\end{theorem}

The neighborhood $\mathcal{U}$, identified with $\mathcal{V}$ through $\Theta$, is called the \emph{Weinstein Lagrangian tubular neighborhood}. 
One important consequence of the theorem is that every closed $1$-form $\eta$ on $L$ whose graph is contained in $\mathcal{U}$ determines a Lagrangian submanifold of $M$ via the embedding $\Theta\circ \eta: L\to M$. 
When we take a compatible Riemannian metric $g$, that is, a metric such that there exists an $\omega$-compatible almost complex structure $J$ on $M$ with $g=\omega(\,\cdot\,,J\,\cdot\,)$, 
we can say that the radius of $\mathcal{U}$ is uniformly bigger than $r$ if 
\[
U_{r}(T^{\bot}L)\subset \mathcal{U}, 
\]
where we identify the normal bundle $T^{\bot}L$ and the cotangent bundle $T^{\ast}L$ by 
$T^{\bot}L\ni \xi\mapsto -\omega(\xi,\,\cdot\,)\in T^{\ast}L$ and $U_{r}(T^{\bot}L)$ is the disk bundle defined by the union of $\{\,v\in T^{\bot}_{p}L\mid |v|<r\,\}$ over $p\in L$ for $r>0$. 
In other words, we can say that at least on $U_{r}(T^{\bot}L)$ there exists a diffeomorphism 
$\Theta$ from $U_{r}(T^{\bot}L)$ to its image in $M$ so that $\Theta(p)=p$ for all $p\in L$ and $\Theta^{\ast}\omega=\tilde{\omega}$. 

Weinstein's theorem guarantees the existence of such a neighborhood, but does not provide a quantitative lower bound for its radius. 
Consequently, applications of the theorem often involve expressions such as ``a Lagrangian submanifold sufficiently $C^1$-close to $L$'' or ``a sufficiently small closed $1$-form,'' without specifying how small is sufficient. The purpose of this paper is to obtain an explicit geometric control on this size. 

To be precise, we introduce the following quantity:  
\[
r_{W}(L):=\sup\left\{\,r>0\,\bigg|\, \begin{aligned}&\mbox{there exists a diffeomorphism }\Theta\mbox{ from }U_{r}(T^{\bot}L)\mbox{ to its}\\
&\mbox{image in }M\mbox{ so that }\Theta(p)=p\mbox{ for all }p\in L\mbox{ and }\Theta^{\ast}\omega=\tilde{\omega}\end{aligned}\,\right\}. 
\]
Our aim is to estimate $r_W(L)$ from below in terms of explicit geometric quantities associated with $M$ and $L$, such as the Riemannian curvature of $M$ and the second fundamental form of $L$.

There is a global obstruction to obtaining such an estimate for embedded Lagrangian submanifolds using only local geometric quantities: two points of $L$ that are far apart with respect to the intrinsic distance on $L$ may be close to each other in $M$. It is therefore natural first to consider the corresponding immersed problem. Let $i:L\to M$ be a Lagrangian immersion. 
We define 
\[
r_{W}^{imm}(L):=\sup\left\{\,r>0\,\bigg|\, \begin{aligned}&\mbox{there exists an immersion }\Theta:U_{r}(T^{\bot}L)\to M\\
&\mbox{so that }\Theta(p)=i(p)\mbox{ for all }p\in L\mbox{ and }\Theta^{\ast}\omega=\tilde{\omega}\end{aligned}\,\right\}. 
\]
Since $U_r(T^{\bot}L)$ and $M$ have the same dimension, the condition that $\Theta$ be an immersion is equivalent to its being a local diffeomorphism. In the embedded case, we clearly have $r_{W}(L)\leq r_{W}^{imm}(L)$. 
We first obtain an explicit lower bound for $r_W^{imm}(L)$ using only local geometric quantities. We then show that, after introducing an additional global quantity measuring the embedding of $L$ in $M$, one can also obtain an explicit lower bound for $r_W(L)$. 
\subsection{Setting and main results}
For several technical reasons, we work with a K\"ahler ambient manifold. In particular, the identity $\nabla J=0$ is used repeatedly in the estimates. Throughout the paper, $(M,g,J,\omega)$ denotes a connected complete K\"ahler manifold. We write $R_M$ for the Riemannian curvature tensor of $(M,g)$ and $\nabla^kR_M$ for its covariant derivatives.

Let $L$ be a connected manifold with $\dim L=\dim M/2$ and $i:L\to M$ be a Lagrangian immersion, that is, $i^{\ast}\omega=0$. 
We usually denote $i^{\ast}\omega$ and $i^{\ast}g$ by $\omega|_{L}$ and $g|_{L}$, respectively. 
We always assume that $(L,g|_{L})$ is complete. 
We denote the second fundamental form of $i:L\to M$ and its higher derivatives by $\mathrm{II}$ and $\nabla^{k}\mathrm{II}$ ($k\geq 1$), respectively. 

To state one of our main theorems, set 
\[
\begin{aligned}
B:=\max\bigg\{\,&\sup_{M}|R_{M}|^{\frac{1}{2}},\sup_{M}|\nabla R_{M}|^{\frac{1}{3}},\sup_{M}|\nabla^2 R_{M}|^{\frac{1}{4}},\\
&\sup_{L}|\mathrm{II}|,\sup_{L}|\nabla\mathrm{II}|^{\frac{1}{2}},\sup_{L}|\nabla^2\mathrm{II}|^{\frac{1}{3}}\,\bigg\}. 
\end{aligned}
\]
Our first main result is the following. 
\begin{theorem}\label{maintheorem1}
If $B<\infty$, then 
\[
r_{W}^{imm}(L)\geq 10^{-7729} \frac{1}{B}. 
\]
\end{theorem}

This is proved as Theorem \ref{97ygy3rowsd} in Section \ref{9vbcsanvgva3}. To the author's knowledge, an explicit lower bound of this form, depending only on the Riemannian curvature of the ambient manifold, the second fundamental form of the Lagrangian immersion, and their covariant derivatives up to second order, has not previously appeared in the literature. 

No attempt is made here to optimize the numerical constant $10^{-7729}$. The point is rather that the bound is completely explicit, depends only on specified geometric quantities, and has the correct scaling behavior. Indeed, as explained in Remark \ref{89bsjndava}, $B^{-1}$ and $r_W^{imm}(L)$ are both homogeneous of degree $1/2$ under scaling of the ambient metric. 

Theorem \ref{maintheorem1} applies in particular to noncompact complete Lagrangian immersions. For every $0<r<10^{-7729}B^{-1}$,  it provides an immersion $\Theta:U_r(T^{\bot}L)\to M$ such that $\Theta(p)=i(p)$ for all $p\in L$ and $\Theta^{\ast}\omega=\tilde{\omega}$. 

We next turn to the embedded case. Suppose that $L$ is compact and that $i:L\to M$ is an embedding. To measure the global behavior of the embedding, define the \emph{embedding constant} 
\[\mathop{\mathrm{emb}}(L):=\sup\left\{\frac{d_{L}(p,q)}{d_{M}(i(p),i(q))}\,\bigg|\,p,q\in L\mbox{ with }p\neq q\right\}, \]
where $d_{M}$ is the distance function on $M$ and $d_{L}$ 
is the one on $L$ with respect to the induced Riemannian metric on $L$. 
One has $1\leq \mathop{\mathrm{emb}}(L)<\infty$. 
Assume in addition that the injectivity radius of $(M,g)$, denoted by $\mathrm{inj}(M,g)$, is positive. 
Set 
\[
\begin{aligned}
B_{*}:=3\mathop{\mathrm{emb}}(L)\times\max\bigg\{\,&\frac{1}{\mathrm{inj}(M,g)},\sup_{M}|R_{M}|^{\frac{1}{2}},\sup_{M}|\nabla R_{M}|^{\frac{1}{3}},\sup_{M}|\nabla^2 R_{M}|^{\frac{1}{4}},\\
&\sup_{L}|\mathrm{II}|,\sup_{L}|\nabla\mathrm{II}|^{\frac{1}{2}},\sup_{L}|\nabla^2\mathrm{II}|^{\frac{1}{3}}\,\bigg\}. 
\end{aligned}
\]

Then we obtain the following embedding version.
\begin{theorem}\label{maintheorem2}
Assume that $L$ is compact and $i:L\to M$ is an embedding. 
If $B_{\ast}<\infty$, then 
\[
r_{W}(L)\geq 10^{-7729} \frac{1}{B_{\ast}}. 
\]
\end{theorem}

This is proved as Theorem \ref{3h2HuVLYkq} in Section \ref{b63lbjscahq4}. 
Thus, whenever $0<r<10^{-7729}B_{\ast}^{-1}$, there exists a diffeomorphism from $U_r(T^{\bot}L)$ onto its image in $M$ such that $\Theta(p)=p$ for all $p\in L$ and $\Theta^{\ast}\omega=\tilde{\omega}$. 

\begin{remark}\label{89bsjndava}
The quantities $B$ and $B_{\ast}$ are homogeneous of degree $-1/2$ under scaling of the ambient Riemannian metric. More precisely, for every $\lambda>0$, \[ B(\lambda g)=\lambda^{-1/2}B(g), \quad B_{\ast}(\lambda g)=\lambda^{-1/2}B_{\ast}(g). \]
\end{remark}

\subsection{Relation to previous work}
Quantitative control of local symplectic models has appeared in several forms in the literature. For example, Taubes \cite{MR1362874} uses symplectic coordinates of uniform size together with estimates comparing the ambient metric with the Euclidean metric. Joyce, Lee and Schoen \cite{MR2823871} construct families of Darboux coordinates with uniform metric estimates. These constructions are related in spirit to the present work, but concern local Darboux neighborhoods of points rather than Weinstein tubular neighborhoods along a Lagrangian submanifold. In particular, these works do not aim to give an explicit lower bound for the Weinstein radius in terms of the curvature of the ambient manifold, the second fundamental form of the Lagrangian submanifold, and their derivatives.

A related type of geometric control appears in the work of Groman and Solomon \cite{MR3548485} on $J$-holomorphic curves with Lagrangian boundary. In order to establish key analytic estimates, they obtain uniform control of a normal neighborhood of $L$ and of derivatives of the normal exponential map. Some of the local computations in Section 3 of \cite{MR3548485} are reminiscent of estimates used in the present paper. The purpose, however, is different: our argument requires quantitative estimates for the time-dependent vector field generating the Moser deformation and for its first derivative, which are needed to control the existence of the resulting flow up to time $1$.

There are also other ways to measure the size of symplectic embeddings. For two symplectic manifolds $(X,\Omega)$ and $(M,\omega)$, Cieliebak, Hofer, Latschev and Schlenk \cite{MR2369441} define the \emph{embedding capacity} $c_{(X,\Omega)}(M,\omega)$ as the supremum of all $\alpha>0$ for which there exists a symplectic embedding from $(X,\alpha\Omega)$ to $(M,\omega)$. 
For a compact embedded Lagrangian $L$, taking $X=U_1(T^{\bot}L)$ and $\Omega=\tilde{\omega}$ gives $c_{(X,\Omega)}(M,\omega)\geq r_W(L)$. 
Consequently, Theorem \ref{maintheorem2} yields \[ c_{(X,\Omega)}(M,\omega) \geq 10^{-7729}B_{\ast}^{-1}. \] Lower bounds for related embedding capacities have also been studied in specific situations, for example for $M=\mathbb{CP}^n$ and $L=T^n$ in \cite{MR3773793}. 
\subsection{Features of the proof and possible applications}
The proof follows the classical Moser construction underlying Weinstein's theorem, but every step is accompanied by explicit estimates. The desired map $\Theta$ is obtained from the time-$1$ map $\Phi(\,\cdot\,,1)$ of the generated flow $\Phi(\,\cdot\,,t)$ of a time-dependent vector field $\mathcal{X}_t$. A quantitative issue arises here that is invisible in the usual qualitative proof: local existence from the standard Picard--Lindel\"of theorem for ODEs does not by itself guarantee that the flow $\Phi(\,\cdot\,,t)$ starting in a prescribed neighborhood remains defined up to time $t=1$. We therefore use a quantitative ODE existence-time estimate, in the form of Lindel\"of's lemma, to control the flow. This requires estimates not only for $|\mathcal{X}_t|$ but also for the norm of its first derivative. See Section \ref{7t08ggna:kv} for details. 

The scaling behavior of the estimate is also worth emphasizing. As noted in Remark \ref{89bsjndava}, $B^{-1}$ and $B_{\ast}^{-1}$ are homogeneous of degree $1/2$ under scaling of $g$. On the other hand, $r_W^{imm}(L)$ and $r_W(L)$ have the same homogeneity. Hence the inequalities in Theorems \ref{maintheorem1} and \ref{maintheorem2} are invariant under rescaling of the ambient metric.

Although we do not pursue concrete applications in this paper, the explicit radius gives a precise domain in which nearby Lagrangian graphs can be constructed. For example, let $f\in C^{\infty}(L)$ satisfy
\[ \sup_L|\nabla f| < 10^{-7729}B^{-1}. \]
Then, $L_f:=\Theta\bigl(J(\nabla f)\bigr)$ defines a Lagrangian immersion in $M$. In the compact embedded setting, the corresponding condition with $B_{\ast}$ gives an embedded Lagrangian submanifold. Thus the results provide an explicit $C^1$-size condition for Hamiltonian deformations represented as Lagrangian graphs. Such a quantitative framework may be useful when geometric problems for nearby Lagrangians are reformulated as PDEs for the potential $f$. One possible direction is the study of special Lagrangian submanifolds within a Hamiltonian deformation class, in connection with the Thomas--Yau conjecture \cite{MR1957663}. We do not pursue this direction here, but hope that the present estimates provide a useful starting point for further work on quantitative problems involving Lagrangian submanifolds. 

\vspace{1ex}
\noindent
\textbf{Acknowledgments.}
This work was supported by JSPS KAKENHI Grant-in-Aid for Early-Career Scientists JP22K13909.
\section{Strategy of the proof and outline of the paper}\label{21usb783pav}
\subsection{Strategy of the proof}
Our proof is constructive. When $r>0$ satisfies $r< 10^{-7729}(1/B)$, we construct an immersion (or a diffeomorphism) from $U_{r}(T^{\bot}L)$ to $M$ that satisfies $\Theta(p)=p$ for all $p\in L$ and $\Theta^{\ast}\omega=\tilde{\omega}$. 
The construction is based on a well-known method: using the \emph{Moser trick} (appeared in \cite{MR182927}) to construct a one-parameter family of diffeomorphisms.
During the process, we need to shrink tubular neighborhoods repeatedly. 
Our proof is completed by quantitatively estimating how small these neighborhoods should be at each step.
To clarify when we shrink tubular neighborhoods during the construction, we outline the standard procedure explained in many papers and books (see, for example, \cite{MR598470}, \cite{MR516965}, or \cite{MR1853077}). 

Let $(M,\omega)$ be a symplectic manifold and $i:L\to M$ be a Lagrangian immersion. 
We usually denote the point $i(p)\in M$ simply by $p\in M$ for $p\in L$. 
We fix a Riemannian metric $g$ on $M$ and take a compatible almost complex structure $J$ on $M$ so that $g=\omega(\,\cdot\,,J\,\cdot\,)$. We remark that such an almost complex structure always exists (see \cite{MR1853077} for instance). 
We identify the normal bundle $T^{\bot}L$ (with respect to $g$) with $T^{\ast}L$ by $T^{\bot}L\ni \xi \mapsto -\omega(\xi,\,\cdot\,)\in T^{\ast}L$. 
On $T^{\ast}L$, there is the so-called tautological 1-form $\alpha$ defined by $\alpha(X):=\eta(\pi_{\ast}(X))$ for $X\in T_{\eta}(T^{\ast}L)$ ($\eta\in T^{\ast}_{p}L$)), where $\pi:T^{\ast}L\to L$ is the projection, and we define a canonical symplectic form $\omega$ on $T^{\ast}L$ by $\tilde{\omega}:=-d\alpha$. 
By the identification of $T^{\ast}L$ by $T^{\bot}L$, we consider $\tilde{\omega}$ as a symplectic form on $T^{\ast}L$. 

\vspace{2ex}
\noindent
\textbf{Step 1.}
Let $\exp_{p}:T_{p}M\to M$ be the exponential map at $p$ with respect to $g$. 
We define a smooth map $F:T^{\bot}L\to M$ by 
\[F(v):=\exp_{\pi(v)}v. \]
Then, there exists an open neighborhood $\mathcal{U}_{1}$ of $L$ in $T^{\bot}L$ so that $F$ becomes an immersion. This is the first time we should give a restriction to the radius of the tubular neighborhood of $L$. 

\vspace{2ex}
\noindent
\textbf{Step 2.}
Let $\tilde{\omega}$ be the canonical symplectic form on $T^{\bot}L$, and we restrict it to $\mathcal{U}_{1}\subset T^{\bot}L$. 
We also have a symplectic form $F^{\ast}\omega$ on $\mathcal{U}_{1}$. 
In the second step, we need to find a one-parameter family of symplectic forms, denoted by $\{\,\omega_{t}\,\}_{t\in [0,1]}$, on some smaller tubular neighborhood in $\mathcal{U}_{1}$ so that $\omega_{0}=\tilde{\omega}$ and $\omega_{1}=F^{\ast}\omega$. 
Usually, such a path is constructed by 
\[\omega_{t}:=(1-t)\tilde{\omega}+tF^{\ast}\omega. \]
The problem is that $\omega_{t}$ could lack the non-degeneracy at some point in $\mathcal{U}_{1}$. 
Since $\tilde{\omega}$ and $F^{\ast}\omega$ coincide on $L$, there is a sufficiently small tubular neighborhood $\mathcal{U}_{2}\subset \mathcal{U}_{1}$ so that $\omega_{t}$ is non-degenerate on $\mathcal{U}_{2}$ for all $t\in [0,1]$. 
This is the second time we should give a further restriction to the radius of the tubular neighborhood of $L$. 

\vspace{2ex}
\noindent
\textbf{Step 3.}
On the normal bundle $T^{\bot}L$, we have a scaling map $\rho_{t}:T^{\bot}L\to T^{\bot}L$ defined by $\rho_{t}(v):=tv$ for $t\in [0,1]$. 
Then, this one-parameter family of scaling $\{\,\rho_{t}\,\}_{t\in[0,1]}$ becomes a homotopy from the projection $\pi:T^{\bot}L\to L$ to the identity map $\mathrm{id}:T^{\bot}L\to T^{\bot}L$. 
Then, by the homotopy formula on the tubular neighborhood $\mathcal{U}_{2}$ and the property $(F^{\ast}\omega-\tilde{\omega})|_{L}=0$, there exists a 1-form $\mu$ on $\mathcal{U}_{2}$ such that $F^{\ast}\omega-\tilde{\omega}=d\mu$ and $\mu|_{L}=0$. 
By the non-degeneracy of $\omega_{t}$ on $\mathcal{U}_{2}$, we can define a time-dependent vector field $\mathcal{X}_{t}$ ($t\in[0,1]$) on $\mathcal{U}_{2}$ by 
\[\omega_{t}(\mathcal{X}_{t},\,\cdot\,)=-\mu. \]
To complete the Moser trick, we need to construct the one-parameter family of diffeomorphisms $\{\,\Phi_{t}\,\}_{t\in [0,1]}$ on some smaller tubular neighborhood in $\mathcal{U}_{2}$ by integrating the time-dependent vector field $\mathcal{X}_{t}$. 
The point is that we need to ensure the time interval of the existence of $\Phi_{t}$ includes $[0,1]$. We have $X_{t}|_{L}=0$, and hence the integral curve of $\{\,\mathcal{X}_{t}\,\}_{t\in[0,1]}$ starting from $p\in L$, denoted by $\Phi_{t}(p)$, exists for all $t\in [0,1]$ (actually $\Phi_{t}(p)=p$). 
However, if the starting point $p\in \mathcal{U}_{2}$ is far from $L$, $\Phi_{t}(p)$ could touch the boundary of $\mathcal{U}_{2}$ at some small time $t_{\ast}<1$ and could not be extended beyond $t_{\ast}<1$. 
Hence, we need to take a sufficiently small tubular neighborhood $\mathcal{U}_{3}\subset \mathcal{U}_{2}$ such that the integral curve $\Phi_{t}(p)$ of $\{\,\mathcal{X}_{t}\,\}_{t\in[0,1]}$ starting from $p\in \mathcal{U}_{3}$ exists for all $t\in [0,1]$. Then, $\Theta:=F\circ\Phi_{1}:\mathcal{U}_{3}\to M$ satisfies $\Theta^{\ast}\omega=\tilde{\omega}$ since $\Phi_{0}=\mathrm{id}$ and 
\[\Phi_{1}^{\ast}\omega_{1}-\Phi_{0}^{\ast}\omega_{0}=\int_{0}^{1}\frac{d}{dt}(\Phi_{t}^{\ast}\omega_{t})dt=\int_{0}^{1}\Phi_{t}^{\ast}\left(\mathcal{L}_{\mathcal{X}_{t}}\omega_{t}+\frac{d}{dt}\omega_{t}\right)dt=0. \]
This is the third time we should give a further restriction to the radius of the tubular neighborhood of $L$. 
If we do not care about the injectivity of $\Theta$, which is equivalent to the injectivity of $F$, this procedure stops. 

\vspace{2ex}
\noindent
\textbf{Step 4.}
If the immersion $i:L\to M$ is injective, that is, an embedding, we have a chance that the immersion $\Theta=F\circ\Phi_{1}:\mathcal{U}_{3}\to M$ constructed in Step 3 becomes an embedding when we restrict it to a sufficiently small tubular neighborhood $\mathcal{U}_{4}\subset \mathcal{U}_{3}$. 
Since the injectivity of $\Theta$ relies on that of $F$, it suffices to find a smaller tubular neighborhood $\mathcal{U}'_{1}\subset \mathcal{U}_{1}$ so that $F:\mathcal{U}'_{1}\to M$ is an embedding and define $\mathcal{U}_{4}$ by $\mathcal{U}_{3}\cap\mathcal{U}'_{1}$. 

\subsection{Resolution}\label{7t08ggna:kv}
We explain how to estimate the radius of the tubular neighborhood in each step. 
By the aid of the K\"ahler structure on the ambient space, we put a Riemannian metric $G$, the so-called Sasaki metric, on $T^{\bot}L$, and we measure norms of tensors and vector fields on $T^{\bot}L$ by this metric $G$. 
We also put an almost complex structure $\tilde{J}$ on $T^{\bot}L$ so that $\tilde{\omega}=G(\tilde{J}\,\cdot\,,\,\cdot\,)$. 

In our approach, we estimate the size of $\mathcal{U}_{2}$ in Step 2 directly without estimating the size of $\mathcal{U}_{1}$ in Step 1. 
Since $M$ is complete in our setting, the exponential map $F:=\exp$ is defined on the whole $T^{\bot}L$. 
Thus, we can consider a 2-form $F^{\ast}\omega$ on $T^{\bot}L$, which could be degenerate at some point on $T^{\bot}L$. 
First, we will obtain an estimate of the form: 
\begin{equation}\label{6tbdvpbasv}
\omega_{t}(X,\tilde{J}X)\geq C(|v|)|X|_{G}^2\quad\mbox{for all}\quad X\in T_{v}(T^{\bot}L), 
\end{equation}
where $[0,\infty)\ni\lambda\mapsto C(\lambda)\in\mathbb{R}$ with $C(0)=1$ is a decreasing function and $|v|$ is the norm of $v\in T^{\bot}L$. 
This implies that $\omega_{t}$ is non-degenerate at $v\in T^{\bot}L$ if $|v|$ is smaller than a value $r_{2}>0$ so that $C(r_{2})\geq 1/2$, since $\omega_{t}(X,\,\cdot\,)=0$ implies $X=0$ by $\omega_{t}(X,\tilde{J}X)\geq |X|_{G}^2/2$. 
Especially, the non-degeneracy of $\omega_{0}=F^{\ast}\omega$ implies the immersivity of $F$ since $\omega$ is non-degenerate. 
Hence, we can take $\mathcal{U}_{2}$ as $U_{r_{2}}(T^{\bot}L)$. 

To estimate the size of $\mathcal{U}_{3}$ in Step 3, we need to estimate the time-dependent vector field $\mathcal{X}_{t}$ ($t\in[0,1]$). 
For the sake of illustration, we temporarily assume that $\mathcal{X}_{t}$ is defined on a tube $U:=\mathbb{R}^{n}\times B(r)$, where $B(r)$ is an open ball in $\mathbb{R}^{n}$ with radius $r$. 
Let us consider $\mathbb{R}^{n}$ as $L$ and $B(r)$ as a fiber of $U_{r}(T^{\bot}L)$. 
Assume that $\mathcal{X}_{t}$ is Lipschitz continuous. 
First, we can prove that $\mathcal{X}_{t}$ has at most linear growth in the fiber direction (i.e., $y$-direction) as 
\begin{equation}\label{07ycghmwirw}
|\mathcal{X}_{t}(x,y)|\leq C|y|
\end{equation}
with some constant $C>1$. 
Let us try to construct the integral curve $c(t)$ of $\{\,\mathcal{X}_{t}\,\}_{t\in[0,1]}$ which starts from a point $p$ in a smaller tube $\mathbb{R}^{n}\times B(\alpha r)$ with $0<\alpha<1$. 
Then, by the well-known Picard--Lindel\"of theorem for ODE, we can obtain the unique existence of $c(t)$ with $c(0)=p$ up to $t_{\ast}=\min\{\,1,\mathop{\mathrm{dist}}(p,\partial U)/M_{\ast}\,\}$, where $M_{\ast}=\sup\{\,|\mathcal{X}_{t}(q)|\mid q\in U, t\in [0,1]\,\}$. 
In this case, since $M_{\ast}\leq Cr$ by \eqref{07ycghmwirw}, we have 
\[\frac{\mathrm{dist}(p,\partial U)}{M_{\ast}}\geq \frac{r-\alpha r}{C r}=\frac{1-\alpha}{C}, \]
and we cannot assert that $t_{\ast}= 1$, even if $\alpha$ is very small, in general since $(1-\alpha)/C<1$. 
Namely, only by the standard Picard--Lindel\"of theorem, we can not guarantee that the existence time of the integral curve includes $[0,1]$. 
However, here is a trick. 
There is the so-called \emph{Lindel\"of's lemma}, which states that if we know the Lipschitz constant, say $D$, of $\{\,\mathcal{X}_{t}\,\}_{t\in[0,1]}$, the integral curve $c(t)$ with $c(0)=p$ exists up to 
\[t_{\ast}=\min\left\{\,1,\frac{1}{D}\log\left(1+\frac{D\mathop{\mathrm{dist}}(p,\partial U)}{M_{\ast}(p)}\right)\,\right\}, \]
where $M_{\ast}(p)=\sup\{\,|\mathcal{X}_{t}(p)|\mid t\in [0,1]\,\}$. 
In this case, since $M_{\ast}(p)\leq C\alpha r$ by \eqref{07ycghmwirw}, we have 
\[\frac{1}{D}\log\left(1+\frac{D\mathop{\mathrm{dist}}(p,\partial U)}{M_{\ast}(p)}\right)\geq \frac{1}{D}\log\left(1+\frac{D(1-\alpha))}{C\alpha }\right), \]
and we can assert that the right-hand side is bigger than $1$ when $\alpha$ is sufficiently small. 
Thus, $t_{\ast}=1$ when the starting point $p$ is in such a small tube $\mathbb{R}^{n}\times B(\alpha r)$. 

The above argument implies that we need to estimate the Lipschitz constant of $\mathcal{X}_{t}$, which is equivalent to the estimate of $|\nabla \mathcal{X}_{t}|$. 
Since $\mathcal{X}_{t}$ is defined by $\omega_{t}(\mathcal{X}_{t},\,\cdot\,)=-\mu$, the estimate of $|\nabla \mathcal{X}_{t}|$ is obtained from the estimates of $|\omega_{t}|$, $|\nabla \omega_{t}|$, $|\mu|$ and $|\nabla \mu|$, which we need to estimate in this paper. 

\begin{remark}
In this paper, we sometimes use a calculator when performing concrete numerical computations. 
For instance, in expressions such as $(5/4)\cdot e^{(1+(1/4e))}$, we evaluate the expression numerically and write its decimal approximation (e.g., $(5/4)\cdot e^{(1+(1/4e))}=3.72\cdots$). 
After that, we typically estimate it from above (or below) by a nearby natural number---in this case, $3.72\cdots\leq 4$.
This practice of replacing complicated expressions with natural numbers (or sometimes fractions) is adopted to prevent many formulas from becoming visually cumbersome, which would make the paper harder to read.
Moreover, since we want to avoid using unspecified constants $C$ (whose existence is only guaranteed) in our main results, it is necessary to carry out every evaluation in terms of concrete numerical values. 
\end{remark}
\subsection{Outline of the paper}
In Section \ref{7gbvcscvwwf}, we explain the motivation for this work and state the main theorems of this paper. 
In Section \ref{21usb783pav}, we explain the strategy of the proof. 
In Section \ref{7go24vnao@ew}, we explain the almost K\"ahler structure on $T^{\bot}L$ with the Sasaki metric $G$ and prove the estimate \eqref{6tbdvpbasv}. We also give an upper bound for $|\nabla^{G}\tilde{\omega}|$, where $\nabla^{G}$ is the Levi--Civita connection of $(T^{\bot}L,G)$. 
In Section \ref{bhas840blhv}, we give an upper bound for $|\nabla^{G}(F^{\ast}\omega)|$ after establishing the estimate for the second derivative of $F$. 
In Section \ref{gyfsw5bdsvd}, we consider the scaling map $\rho_{t}$ in Step 3 above and give some estimates for $\rho_{t}$ and its first derivative. 
In Section \ref{9089awgrenbva}, we prove the estimate \eqref{07ycghmwirw} and also give an estimate for $|\nabla^{G}\mathcal{X}_{t}|$. 
In Section \ref{bv287adsggql}, we pull back $\mathcal{X}_{t}$ via the local trivialization $U\times T_{p}^{\bot}L\to (T^{\bot}L)|_{U}$, where $U\subset T_{p}L$, defined by a specific way, and we estimate the norm of it and its derivatives under the standard flat metric on $U\times T_{p}^{\bot}L$. 
In Section \ref{9vbcsanvgva3}, we prove one of the main theorems: Theorem \ref{maintheorem1} for $r_{W}^{imm}(L)$. 
In Section \ref{b63lbjscahq4}, we consider the embedding case and prove Theorem \ref{maintheorem2} for $r_{W}(L)$. 
\section{Nondegeneracy of the path of symplectic forms}\label{7go24vnao@ew}
Assume that $(M,g)$ is a connected complete $2n$-dimensional K\"ahler manifold 
with complex structure $J$ and K\"ahler form $\omega$. 
Let $L$ be a connected manifold with $\dim L=\dim M/2$ and $i:L\to M$ be a Lagrangian immersion. 
We usually denote the point $i(p)\in M$ simply by $p\in M$ for $p\in L$. 
Let $\pi:T^{\bot}L\to L$ be the normal bundle of $L$. 
In this section, we assume that there exist $C_{0},A_{0}\geq 0$ such that 
\begin{equation}\label{m5goLqTgtc}
|R_{M}|\leq C_{0}\quad\mbox{and}\quad |\mathrm{II}|\leq A_{0}. 
\end{equation}
We remark that this implies 
\begin{equation}\label{K2wfDjgUgB}
|R_{L}|\leq C_{0}+2A_{0}^2
\end{equation}
by the Gauss equation (see \eqref{ndTscyk0OG} for instance). 
\subsection{Almost K\"ahler structure on \texorpdfstring{$T^{\bot}L$}{the normal bundle of L}}
We start to define a Riemannian metric on $T^{\bot}L$, the so-called Sasaki metric. 
Let $\nabla^{\bot}$ be the normal connection on $T^{\bot}L$ induced by 
the Levi--Civita connection $\nabla$ of $(M,g)$. 
Namely, $\nabla^{\bot}$ is defined by 
\[\nabla^{\bot}_{X}Y:=(\nabla_{X}Y)^{\bot}\]
for a tangent vector $X$ on $L$ and smooth section $Y$ of $T^{\bot}L$. 
Let $\bar{\nabla}$ be the Levi--Civita connection of $(L, \bar{g})$ with $\bar{g}:=g|_{L}$. 
Then, we have 
\begin{equation}\label{PxAYvl6gpc}
\nabla^{\bot}_{X}Y=-(J(\nabla_{X}(JY)))^{\bot}=-\left\{J\bar{\nabla}_{X}(JY)+J\mathop{\mathrm{II}}(X,JY)\right\}^{\bot}=-J\bar{\nabla}_{X}(JY), 
\end{equation}
where the first equality follows from $\nabla J=0$ and the second and third ones follow from that $L$ is Lagrangian. By a standard theory of vector bundle with a connection, 
the tangent bundle of $T^{\bot}L$ can be decomposed into the horizontal and vertical distributions. 
Fix $\eta\in T^{\bot}L$ and put $p:=\pi(\eta)\in L$. 
Then, the \textit{vertical subspace} $V_{\eta}$ in $T_{\eta}(T^{\bot}L)$ is just $T_{\eta}(T^{\bot}_{p}L)$, the tangent space of $T^{\bot}_{p}L$. This is 
also the kernel of $\pi_{\ast\eta}:T_{\eta}(T^{\bot}L)\to T_{p}L$. 
We denote the natural identification of $V_{\eta}$ and $T^{\bot}_{p}L$ by 
$\iota_{\eta}:V_{\eta}\to T^{\bot}_{p}L$ and write $[Y]^{v}_{\eta}:=\iota_{\eta}^{-1}(Y)$ for $Y\in T_{p}^{\bot}L$, which we call the \textit{vertical lift} of $Y$ at $\eta$. 
On the other hand, the \textit{horizontal subspace} $H_{\eta}$ in $T_{\eta}(T^{\bot}L)$ is 
defined as the vector space spanned by all vectors of the form $\dot{\eta}_{\gamma}(0)$, 
where $\eta_{\gamma}(t)$ is the parallel transport of $\eta$ 
(so it is a curve in $T^{\bot}L$) with respect to the normal connection 
$\nabla^{\bot}$ along a curve $\gamma$ in $L$ through $p$ at $t=0$. 
When $X=\dot{\gamma}(0)\in T_{p}L$, we denote $\dot{\eta}_{\gamma}(0)$ by $[X]^{h}_{\eta}$ 
and call it the \textit{horizontal lift} of $X$ at $\eta$. 
Then, $\pi_{\ast\eta}$ restricted to $H_{\eta}$ gives the isomorphism from $H_{\eta}$ to $T_{p}L$ 
and satisfies $\pi_{\ast\eta}[X]^{h}_{\eta}=X$. 
Then, we have the horizontal and vertical decomposition: 
\[T_{\eta}(T^{\bot}L)=H_{\eta}\oplus V_{\eta}. \]
For $X\in T_{\eta}(T^{\bot}L)$, we denote by $X^{h}$ and $X^{v}$ 
the horizontal part and the vertical part of $X$, respectively. 
Then, we can define the \textit{connection map} $K:T(T^{\bot}L)\to T^{\bot}L$ by 
$K(X):=\iota_{\eta}(X^{v})\in T^{\bot}_{\pi(\eta)}L$ for $X\in T_{\eta}(T^{\bot}L)$. 
It is well known that if $X$ is the time derivative at $0$ of a curve $c$ in $T^{\bot}L$ such 
that $c$ can be written as a section of $T^{\bot}L$ on some curve $\gamma$ in $L$ then 
$K(X)=\nabla^{\bot}_{\dot{\gamma}(0)}c$. 
By this decomposition, we give a Riemannian metric $G$ on $T^{\bot}L$ as 
\[G(X,Y):=g(\pi_{\ast}X,\pi_{\ast}Y)+g(K(X),K(Y))\]
and this is called the \textit{Sasaki metric}. 
We can also define an almost complex structure $\tilde{J}$ on $T^{\bot}L$ at $\eta\in T^{\bot}L$ with $p=\pi(\eta)$ by 
\begin{equation}\label{dG42GUBA0G}
\tilde{J}X:=[J(K(X))]^{h}_{\eta}+[J(\pi_{\ast}X)]^{v}_{\eta}
\end{equation}
for $X\in T_{\eta}(T^{\bot}L)$, and it satisfies $G(\tilde{J}X,\tilde{J}Y)=G(X,Y)$. 
Then, we get the associated 2-form $\tilde{\omega}$ on $T^{\bot}L$ by 
 \begin{equation}\label{qIUTtqPF11}
 \tilde{\omega}:=G(\tilde{J}\,\cdot\,,\,\cdot\,).
 \end{equation}

Let $(x^{1},\dots,x^{n})$ be a local coordinate system on an open set $U$ in $L$. 
Then, by identifying $(x^{1},\dots,x^{n},\xi^{1},\dots,\xi^{n})$ with 
\begin{equation}\label{VahfBGGECR}
\eta=\xi^{1}J\left(\frac{\partial}{\partial x^{1}}\bigg\vert_{x}\right)+\dots+\xi^{n}J\left(\frac{\partial}{\partial x^{n}}\bigg\vert_{x}\right)\in T^{\bot}_{x}L, 
\end{equation}
we get a local coordinate system on $(T^{\bot}L)|_{U}$. On this local chart, by using \eqref{PxAYvl6gpc}, one can easily see that 
\[
\left[\frac{\partial}{\partial x^{i}}\bigg\vert_{x}\right]^{h}_{\eta}
=\frac{\partial}{\partial x^{i}}\bigg\vert_{(x,\xi)}-\sum_{j,k=1}^{n}\xi^{j}\bar{\Gamma}_{ji}^{k}(x)\frac{\partial}{\partial \xi^{k}}\bigg\vert_{(x,\xi)}, \quad
\left[J\left(\frac{\partial}{\partial x^{i}}\bigg\vert_{x}\right)\right]^{v}_{\eta}
=\frac{\partial}{\partial \xi^{i}}\bigg\vert_{(x,\xi)}, 
\]
where $\bar{\Gamma}_{ij}^{k}(x)$ is the Christoffel symbol of 
the Levi--Civita connection $\bar{\nabla}$ of $(L, \bar{g})$ at $x$. 
Conversely, we have 
\begin{equation}\label{O2Lhl6lar7}
\frac{\partial}{\partial x^{i}}\bigg\vert_{(x,\xi)}=\left[\frac{\partial}{\partial x^{i}}\bigg\vert_{x}\right]^{h}_{\eta}+\sum_{j,k=1}^{n}\xi^{j}\bar{\Gamma}_{ji}^{k}(x)\left[J\left(\frac{\partial}{\partial x^{k}}\bigg\vert_{x}\right)\right]^{v}_{\eta}. 
\end{equation}
Hence, we have 
\[\tilde{J}\left(\frac{\partial}{\partial \xi^{i}}\bigg\vert_{(x,\xi)}\right)
=\left[-\frac{\partial}{\partial x^{i}}\bigg\vert_{x}\right]^{h}_{\eta}
=-\frac{\partial}{\partial x^{i}}\bigg\vert_{(x,\xi)}+\sum_{j,k=1}^{n}\xi^{j}\bar{\Gamma}_{ji}^{k}(x)\frac{\partial}{\partial \xi^{k}}\bigg\vert_{(x,\xi)}. \]
This implies that 
\[\tilde{\omega}\left(\frac{\partial}{\partial \xi^{i}},\frac{\partial}{\partial \xi^{j}}\right)=0,\quad 
\tilde{\omega}\left(\frac{\partial}{\partial \xi^{i}},\frac{\partial}{\partial x^{j}}\right)=-\bar{g}_{ji}\]
and 
\[\tilde{\omega}\left(\frac{\partial}{\partial x^{i}},\frac{\partial}{\partial x^{j}}\right)
=-\sum_{k,\ell=1}^{n}\xi^{k}\bar{\Gamma}_{ki}^{\ell}g_{\ell j}+\sum_{k,\ell=1}^{n}\xi^{k}\bar{\Gamma}_{kj}^{\ell}g_{i \ell}=\sum_{k=1}^{n}\xi^{k}\left(\frac{\partial\bar{g}_{ik}}{\partial x^{j}}-\frac{\partial\bar{g}_{jk}}{\partial x^{i}}\right). \]
Thus, we have 
\begin{equation}\label{duYHiTmOjx}
\tilde{\omega}=\sum_{i,j=1}^{n}\bar{g}_{ij}dx^{i}\wedge d\xi^{j}+\sum_{i,j,k=1}^{n}\xi^{k}\frac{\partial\bar{g}_{ik}}{\partial x^{j}}dx^{i}\wedge dx^{j}. 
\end{equation}
When we define a 1-form $\alpha$ on $T^{\bot}L$ by 
\[\alpha:=\sum_{i,j=1}^{n}\bar{g}_{ij}(x)\xi^{j}dx^{i}, \]
then we have $\tilde{\omega}=-d\alpha$. 
Denote the identification $T^{\bot}L\ni \xi\mapsto (-J\xi)^{\flat}\in T^{\ast}L$ by $I:T^{\bot}L \to T^{\ast}L$, where $(-J\xi)^{\flat}:=-g(J\xi,\,\cdot\,)=-\omega(\xi,\,\cdot\,)$. 
Let $\alpha'$ and $\tilde{\omega}'$ ($=-d\alpha'$) be the tautological 1-form and the canonical symplectic form on $T^{\ast}L$, respectively. 
Then, one can easily see that $\alpha=I^{\ast}\alpha'$. 
This implies that $\tilde{\omega}=I^{\ast}\tilde{\omega}'$. 
In particular, $\tilde{\omega}$ is closed. 
Thus, $(T^{\bot}L,G,\tilde{J},\tilde{\omega})$ is an almost K\"ahler manifold. 
\subsection{Levi--Civita connection on \texorpdfstring{$(T^{\bot}L,G)$}{the normal bundle of L}}
We also need the explicit expression of the Levi--Civita connection $\nabla^{G}$ of the Sasaki metric $G$ on $T^{\bot}L$. 
Let $X$, $X_{1}$ and $X_{2}$ be vector fields defined on some open set $U$ in $L$ 
and $Y$, $Y_{1}$ and $Y_{2}$ be local sections of $T^{\bot}L$ over $U$. 
Then, we get the horizontal lift $[X]^{h}$ of $X$ and the vertical lift $[Y]^{v}$ of $Y$. 
The Levi--Civita connection for these lifted vector fields 
are explicitly computed in the equations (8), (9), (10) and (11) in \cite{MR286028} in 
the case where the total space is the tangent bundle of a Riemannian manifold 
with the Sasaki metric. 
Since $-J:T^{\bot}L\to TL$ gives the isometry as Riemannian manifolds with Sasaki metrics, 
we can easily see that 
\begin{equation}\label{aYmLTS3JTr}
\begin{aligned}
&(\nabla^{G}_{[Y_{1}]^{v}}[Y_{2}]^{v})_{\eta}=0\\
&(\nabla^{G}_{[X]^{h}}[Y]^{v})_{\eta}=[\nabla^{\bot}_{X}Y]^{v}_{\eta}+\frac{1}{2}[\mathop{R_{L}}(J\eta,JY)X]^{h}_{\eta}\\
&(\nabla^{G}_{[Y]^{v}}[X]^{h})_{\eta}=\frac{1}{2}[\mathop{R_{L}}(J\eta,JY)X]^{h}_{\eta}\\
&(\nabla^{G}_{[X_{1}]^{h}}[X_{2}]^{h})_{\eta}=[(\bar{\nabla}_{X_{1}}X_{2})]^{h}_{\eta}+\frac{1}{2}[J(\mathop{R_{L}}(X_{1},X_{2})(J\eta))]^{v}_{\eta}
\end{aligned} 
\end{equation}
at $\eta\in T^{\bot}L|_{U}$. 
By these formulae, we have the following. 
\begin{lemma}\label{RT2JlOvGJJ}
Let $\eta:(-\varepsilon',\varepsilon')\ni\alpha \mapsto \eta(\alpha)\in T^{\bot}L$ be a curve. 
Put $\eta_{0}:=\eta(0)$ and $Y:=\dot{\eta}(0)$. 
Let $X_{\alpha}\in T_{\eta(\alpha)}(T^{\bot}L)$ be the parallel transport of $X\in T_{\eta_{0}}(T^{\bot}L)$ 
along $\eta$ with respect to $G$. 
Put $U_{\alpha}:=\pi_{\ast}X_{\alpha}$, $W_{\alpha}:=K(X_{\alpha})$, $A:=\pi_{\ast}Y$ and $B:=K(Y)$. 
Then, we have 
\[
\begin{aligned}
\bar{\nabla}_{A}U_{\alpha}=&-\frac{1}{2}\mathop{R_{L}}(J\eta_{0},JW_{0})A-\frac{1}{2}\mathop{R_{L}}(J\eta_{0},JB)U_{0},\\
\nabla^{\bot}_{A}W_{\alpha}=&-\frac{1}{2}J(\mathop{R_{L}}(A,U_{0})(J\eta_{0})). 
\end{aligned}
\]
\end{lemma}
\begin{proof}
Since $X_{\alpha}=[U_{\alpha}]^{h}+[W_{\alpha}]^{v}$ and $Y=[A]^{h}+[B]^{v}$, by \eqref{aYmLTS3JTr}, we have 
\[
\begin{aligned}
0=&D\pi(\nabla^{G}_{Y}X_{\alpha})=\bar{\nabla}_{A}U_{\alpha}+\frac{1}{2}\mathop{R_{L}}(J\eta_{0},JW_{0})A+\frac{1}{2}\mathop{R_{L}}(J\eta_{0},JB)U_{0},\\
0=&K(\nabla^{G}_{Y}X_{\alpha})=\nabla^{\bot}_{A}W_{\alpha}+\frac{1}{2}J(\mathop{R_{L}}(A,U_{0})(J\eta_{0})), 
\end{aligned}
\]
and the proof is complete. 
\end{proof}

The almost complex structure $\tilde{J}$ on $T^{\bot}L$ is not integrable in general. 
We can compute the norm of $\nabla^{G}\tilde{J}$ as follows. 
\begin{proposition}\label{Yf0Wh2WqI1}
We have 
\[|\nabla^{G}\tilde{J}|_{G}\leq \sqrt{2}(C_{0}+2A_{0}^2)|\eta_{0}|\]
at $\eta_{0}\in T^{\bot}L$. 
\end{proposition}
\begin{proof}
Fix $X,Y\in T_{\eta_{0}}(T^{\bot}L)$ arbitrarily. 
Put $A:=\pi_{\ast}Y$ and $B:=K(Y)$. 
Let $\eta:(-\varepsilon',\varepsilon')\ni\alpha \mapsto \eta(\alpha)\in T^{\bot}L$ be a curve so that $\eta(0)=\eta_{0}$ and $\dot{\eta}(0)=Y$, and 
let $X_{\alpha}\in T_{\eta(\alpha)}(T^{\bot}L)$ be the parallel transport of $X$ along $\eta(\alpha)$ with respect to $G$. 
Then, we have 
\[
(\nabla^{G}\tilde{J})(Y,X)=\nabla^{G}_{Y}(\tilde{J}X_{\alpha}). 
\]
Put $p(\alpha)=\pi\circ \eta(\alpha)\in L$. As in the statement of Lemma \ref{RT2JlOvGJJ}, 
put $U_{\alpha}:=\pi_{\ast}X_{\alpha}\in T_{p(\alpha)}L$ and $W_{\alpha}:=K(X_{\alpha})\in T^{\bot}_{p(\alpha)}L$. 
By the definition of $\tilde{J}$ (see \eqref{dG42GUBA0G}), we have 
\[\tilde{J}X_{\alpha}=[JW_{\alpha}]^{h}_{\eta(\alpha)}+[JU_{\alpha}]^{v}_{\eta(\alpha)}. \]
Hence, we have 
\begin{equation}\label{yh11I1sI8m}
\nabla^{G}_{Y}(\tilde{J}X_{\alpha})=\nabla^{G}_{A}[JW_{\alpha}]^{h}_{\eta(\alpha)}+\nabla^{G}_{B}[JW_{\alpha}]^{h}_{\eta(\alpha)}+\nabla^{G}_{A}[JU_{\alpha}]^{v}_{\eta(\alpha)}
\end{equation}
since $\nabla^{G}_{B}[JU_{\alpha}]^{v}_{\eta(\alpha)}=0$ by \eqref{aYmLTS3JTr}. 
For the first term of the right-hand side on \eqref{yh11I1sI8m}, we have 
\[
\begin{aligned}
\nabla^{G}_{A}[JW_{\alpha}]^{h}_{\eta(\alpha)}=&[\bar{\nabla}_{A}(JW_{\alpha})]^{h}_{\eta_{0}}+\frac{1}{2}[J(\mathop{R_{L}}(A,JW_{0})(J\eta_{0}))]^{v}_{\eta_{0}}\\
=&[J(\nabla^{\bot}_{A}W_{\alpha})]^{h}_{\eta_{0}}+\frac{1}{2}[J(\mathop{R_{L}}(A,JW_{0})(J\eta_{0}))]^{v}_{\eta_{0}}\\
=&\frac{1}{2}[\mathop{R_{L}}(A,U_{0})(J\eta_{0})]^{h}_{\eta_{0}}+\frac{1}{2}[J(\mathop{R_{L}}(A,JW_{0})(J\eta_{0}))]^{v}_{\eta_{0}}, 
\end{aligned}
\]
where the first equality follows from \eqref{aYmLTS3JTr}, the second one from \eqref{PxAYvl6gpc} and the third one from Lemma \ref{RT2JlOvGJJ}. 
For the second term of the right-hand side on \eqref{yh11I1sI8m}, we have 
\[
\nabla^{G}_{B}[JW_{\alpha}]^{h}_{\eta(\alpha)}=\frac{1}{2}[J(\mathop{R_{L}}(J\eta_{0},JB)(JW_{0}))]^{h}_{\eta_{0}}
\]
from \eqref{aYmLTS3JTr}. 
For the third term of the right-hand side on \eqref{yh11I1sI8m}, we have 
\[
\begin{aligned}
\nabla^{G}_{A}[JU_{\alpha}]^{v}_{\eta(\alpha)}=&[\nabla^{\bot}_{A}(JU_{\alpha})]^{v}_{\eta_{0}}-\frac{1}{2}[\mathop{R_{L}}(J\eta_{0},U_{0})A]^{h}_{\eta_{0}}\\
=&[J(\bar{\nabla}_{A}U_{\alpha})]^{v}_{\eta_{0}}-\frac{1}{2}[\mathop{R_{L}}(J\eta_{0},U_{0})A]^{h}_{\eta_{0}}\\
=&-\frac{1}{2}[J(\mathop{R_{L}}(J\eta_{0},JW_{0})A)]^{v}_{\eta_{0}}-\frac{1}{2}[J(\mathop{R_{L}}(J\eta_{0},JB)U_{0})]^{v}_{\eta_{0}}\\
&-\frac{1}{2}[\mathop{R_{L}}(J\eta_{0},U_{0})A]^{h}_{\eta_{0}}, 
\end{aligned}
\]
where the first equality follows from \eqref{aYmLTS3JTr}, the second one from \eqref{PxAYvl6gpc} and the third one from Lemma \ref{RT2JlOvGJJ}. 
Hence, the norm of the horizontal part of $\nabla^{G}_{Y}(\tilde{J}X_{\alpha})$, denoted by $|(\nabla^{G}_{Y}(\tilde{J}X_{\alpha}))^{h}|_{G}$, is bounded from above by 
\[
\begin{aligned}
|(\nabla^{G}_{Y}(\tilde{J}X_{\alpha}))^{h}|_{G}\leq & |R_{L}||A||U_{0}||\eta_{0}|+\frac{1}{2}|R_{L}||B||W_{0}||\eta_{0}|\\
\leq & |R_{L}||\eta_{0}|\sqrt{|A|^2+|B|^2}\sqrt{|U_{0}|^2+|W|_{0}^2}\\
\leq & (C_{0}+2A_{0}^2)|\eta_{0}||Y||X|. 
\end{aligned}
\]
Similarly, 
the norm of the vertical part of $\nabla^{G}_{Y}(\tilde{J}X_{\alpha})$, denoted by $|(\nabla^{G}_{Y}(\tilde{J}X_{\alpha}))^{v}|_{G}$, is bounded from above by 
\[
|(\nabla^{G}_{Y}(\tilde{J}X_{\alpha}))^{v}|_{G}\leq (C_{0}+2A_{0}^2)|\eta_{0}||Y||X|. 
\]
Hence, we have 
\[
|\nabla^{G}_{Y}(\tilde{J}X_{\alpha})|_{G}\leq 
\sqrt{2}(C_{0}+2A_{0}^2)|\eta_{0}||X||Y| 
\]
and this completes the proof. 
\end{proof}

This immediately implies the following. 
\begin{proposition}\label{1Nq7NTU2sR}
We have 
\[|\nabla^{G}\tilde{\omega}|_{G}\leq \sqrt{2}(C_{0}+2A_{0}^2)|\eta_{0}|\]
at $\eta_{0}\in T^{\bot}L$. 
\end{proposition}
\begin{proof}
Fix $X,Y,Z\in T_{\eta_{0}}(T^{\bot}L)$ arbitrarily. 
Let $\eta:(-\varepsilon',\varepsilon')\ni\alpha \mapsto \eta(\alpha)\in T^{\bot}L$ be a curve so that $\eta(0)=\eta_{0}$ and $\dot{\eta}(0)=Y$, and 
let $X_{\alpha}$ and $Z_{\alpha}$ in $T_{\eta(\alpha)}(T^{\bot}L)$ be the parallel transports of $X$ and 
$W$ along $\eta(\alpha)$ with respect to $G$, respectively. 
Then, from $\tilde{\omega}=G(\tilde{J}\,\cdot\,,\,\cdot\,)$ (see \eqref{qIUTtqPF11}), we have 
\[
(\nabla^{G}_{Y}\tilde{\omega})(X,Z)=Y(\tilde{\omega}(X_{\alpha},Z_{\alpha}))=Y(G(\tilde{J}X_{\alpha},Z_{\alpha}))=G((\nabla^{G}_{Y}\tilde{J})(X),Z). 
\]
By Proposition \ref{Yf0Wh2WqI1}, we have $|(\nabla^{G}_{Y}\tilde{\omega})(X,Z)|\leq \sqrt{2}(C_{0}+2A_{0}^2)|\eta_{0}||X||Y||Z|$ and the proof is complete. 
\end{proof}
\subsection{Lower bound of \texorpdfstring{$F^{\ast}\omega$}{the pull-backed symplectic form}}
Recall that we have defined a smooth map $F:T^{\bot}L\to M$ by 
\[F(v)=\exp_{\pi(v)}v. \]
By pulling back the K\"ahler form $\omega$ on $M$ to $T^{\bot}L$ by $F$, 
we get a 2-form $F^{\ast}\omega$ on $T^{\bot}L$. 
In this subsection, we estimate the norm of the bilinear map $(X,Y)\mapsto (F^{\ast}\omega)(X,\tilde{J}Y)$ from below with respect to $G$. 

Fix $\eta\in T^{\bot}L$ with $|\eta|=1$. Put $p:=\pi(\eta)\in L$. 
We also fix a scaling constant $\lambda\in [0,r)$. 
We do every computation at $\lambda\eta \in T^{\bot}L$. 
Take $X\in T_{\lambda\eta}(T^{\bot}L)$ arbitrarily. 
Then, we want to estimate $(F^{\ast}\omega)(X,\tilde{J}X)$ from below by $C|X|_{G}^2$ 
with some constant $C$. 
By definition, we have 
\begin{equation}\label{UrhECp8hu4}
(F^{\ast}\omega)(X,\tilde{J}X)=\omega(F_{\ast}X,F_{\ast}(\tilde{J}X))=-g(F_{\ast}X,JF_{\ast}(\tilde{J}X)). 
\end{equation}
Thus, our purpose is equivalent to bound $g(F_{\ast}X,JF_{\ast}(\tilde{J}X))$ from above. 
At first, we prepare the following lemma. 

\begin{lemma}\label{Y7hcfipucS}
We have
\[|F_{\ast}X|^2\leq (1+\lambda^2 A_{0}^2)e^{1+\lambda^2C_{0}}|X|_{G}^2. \]
\end{lemma}
\begin{proof}
Put $U:=\pi_{\ast}X\in T_{p}L$ and $W:=K(X)\in T_{p}^{\bot}L$. 
Then, we have the decomposition $X=[U]_{\lambda\eta}^{h}+[W]_{\lambda\eta}^{v}$ and 
remark that $|X|^{2}_{G}=|U|^2+|W|^2$. 
Let $\gamma(t)$ ($t\in(-\varepsilon,\varepsilon)$) be a curve in $L$ such that $\gamma(0)=p$ and $\dot{\gamma}(0)=U$. 
Let $\eta_{\gamma}(t)$ and $W_{\gamma}(t)$ be the parallel transports of $\eta$ and $W$ in $T^{\bot}L$ along $\gamma$ with respect to $\nabla^{\bot}$, respectively. 
Define a curve $c$ in $T^{\bot}L$ (with $c(0)=\lambda\eta$) by 
\[c(t):=\lambda\eta_{\gamma}(t)+t\cdot W_{\gamma}(t)\in T^{\bot}_{\gamma(t)}L. \]
Then, we have $\pi_{\ast}\dot{c}(0)=\dot{\gamma}(0)=U$ and $K(\dot{c}(0))=\nabla_{\dot{\gamma}(0)}^{\bot}c=W_{\gamma}(0)=W$. 
This means that $X=\dot{c}(0)$. Hence, we have $F_{\ast}X=(d/dt|_{t=0})\exp_{\gamma(t)}c(t)$. 
For $s\in[0,1]$, we define $\tilde{X}(s)$ by 
\[\tilde{X}(s):=\frac{d}{dt}\bigg|_{t=0}\exp_{\gamma(t)}(s\cdot c(t)). \]
Then, $F_{\ast}X=\tilde{X}(1)$. 
Putting $\alpha:[0,1]\times(-\varepsilon,\varepsilon)\to M$ by $\alpha(s,t):=\exp_{\gamma(t)}(s\cdot c(t))$, it is clear that $\tilde{X}(s)$ is a variational vector field of a geodesic $s\mapsto \alpha(s,0)=\exp_{p}(s\cdot\lambda\eta)=:\sigma(s)$ in $M$ and hence satisfies 
the Jacobi field equation: 
\[\nabla_{s}\nabla_{s}\tilde{X}+R_{M}(\tilde{X},\dot{\sigma})\dot{\sigma}=0. \]
Moreover, we have the initial condition: 
\[\tilde{X}(0)=\dot{\gamma}(0)=U\quad\mbox{and}\quad \nabla_{s}\tilde{X}(0)=\nabla_{\dot{\gamma}(0)}c=\lambda S_{\eta}(U)+W, \]
where $S_{\eta}:T_{p}L\to T_{p}L$ is the shape operator with respect to $\eta\in T^{\bot}_{p}L$. 
Then, by the standard estimate for Jacobi fields (see Proposition \ref{l065WM0Tcg} with $D_{0}=0$ for instance), we have 
\[
\begin{aligned}
|\tilde{X}|^2 \leq&  \left(|U|^2+|W|^2+\lambda^2|S_{\eta}(U)|^2\right)e^{(1+\lambda^2C_{0})s}\\
\leq& (1+\lambda^2 A_{0}^2)e^{1+\lambda^2C_{0}}|X|_{G}^2. 
\end{aligned}
\]
Then, the proof is complete by recalling $F_{\ast}X=\tilde{X}(1)$. 
\end{proof}

Then, we have the following. 
\begin{proposition}\label{KU5oBVMKyH}
We have 
\[
(F^{\ast}\omega)(X,\tilde{J}X)\geq \left(1-K_{0}(\lambda)\right)|X|_{G}^2, 
\]
where
\begin{equation}\label{dSUX6lxl6M}
K_{0}(\lambda):=2\lambda^2C_{0}(1+\lambda^2A_{0}^2)e^{1+\lambda^2C_{0}}+\lambda A_{0}, 
\end{equation}
which is monotone increasing and satisfies $K_{0}(\lambda)\to 0$ as $\lambda\to 0$. 
\end{proposition}
\begin{proof}
Let $\tilde{X}(s)$ be the Jacobi field defined in the proof of the above lemma so that $\tilde{X}(1)=F_{\ast}X$. 
On the other hand, for $F_{\ast}(\tilde{J}X)$, by the similar argument as in the proof of the above lemma, 
one can see that $F_{\ast}(\tilde{J}X)=\tilde{Y}(1)$, where $\tilde{Y}(s)$ is the Jacobi field
along $\sigma(s)$ with initial condition: 
\[\tilde{Y}(0)=JW\quad\mbox{and}\quad \nabla_{s}\tilde{Y}(0)=\lambda S_{\eta}(JW)+JU\]
since $\tilde{J}X=[JW]_{\lambda\eta}^{h}+[JU]_{\lambda\eta}^{v}$. 
Put $f(s):=g(\tilde{X}(s),J\tilde{Y}(s))$. 
By \eqref{UrhECp8hu4}, what we want to estimate is $f(1)$. 
Then, we have 
\[f(0)=-g(U,W)=0\quad\mbox{and}\quad f'(0)=-g(U,U)-g(W,W)=-|X|_{G}^2. \]
Moreover, by a straightforward computation with the Jacobi field equation and $\nabla J=0$, we see that 
\begin{equation}\label{Qnrg3Wt3nQ}
\begin{aligned}
f''(s)=&-g(R_{M}(\tilde{X},\dot{\sigma})\dot{\sigma},J\tilde{Y})+g(J\tilde{X},R_{M}(\tilde{Y},\dot{\sigma})\dot{\sigma})+2g(\nabla\tilde{X},J\nabla\tilde{Y})\\
\leq & 2\lambda^2 C_{0}|\tilde{X}||\tilde{Y}|+2\int_{0}^{s}\left\{g(\nabla\tilde{X},J\nabla\tilde{Y})\right\}'ds-2g(\nabla\tilde{X}(0),J\nabla\tilde{Y}(0)). 
\end{aligned}
\end{equation}
For the final term in \eqref{Qnrg3Wt3nQ}, we have 
\[
\begin{aligned}
g(\nabla\tilde{X}(0),J\nabla\tilde{Y}(0))=&g(\lambda S_{\eta}(U)+W,\lambda JS_{\eta}(JW)-U)\\
=&\lambda g(\mathop{\mathrm{II}}(U,U),\eta)+\lambda g(\mathop{\mathrm{II}}(JW,JW),\eta)\\
\geq & -\lambda A_{0}|X|^2_{G}, 
\end{aligned}
\]
where the second equality follows from that $g(S_{\eta}(U),\,\cdot\,)=-g(\mathop{\mathrm{II}}(U,\,\cdot\,),\eta)$ and the third inequality follows from the assumption $|\mathrm{II}|\leq A_{0}$. 
For the first term in \eqref{Qnrg3Wt3nQ}, by using 
$2|\tilde{X}||\tilde{Y}|\leq |\tilde{X}|^2+|\tilde{Y}|^2$ and 
Lemma \ref{Y7hcfipucS}, we have 
\[
2|\tilde{X}||\tilde{Y}|\leq 2(1+\lambda^2A_{0}^2)e^{1+\lambda^2C_{0}}|X|^2_{G}. 
\]
For the middle term in \eqref{Qnrg3Wt3nQ}, we compute and estimate the integrand as 
\[
\begin{aligned}
2\left\{g(\nabla\tilde{X},J\nabla\tilde{Y})\right\}'=&
-2g(R_{M}(\tilde{X},\dot{\sigma})\dot{\sigma},J\nabla\tilde{Y})+2g(J\nabla\tilde{X},R_{M}(\tilde{Y},\dot{\sigma})\dot{\sigma})\\
&\leq \lambda^2C_{0}\left((|\tilde{X}|^2+|\nabla\tilde{X}|^2)+(|\tilde{Y}|^2+|\nabla\tilde{Y}|^2)\right)\\
&\leq 2\lambda^2C_{0}(1+\lambda^2A_{0}^2)e^{1+\lambda^2C_{0}}|X|^2_{G}, 
\end{aligned}
\]
where the last inequality follows from Proposition \ref{l065WM0Tcg} (with $D_{0}=0$). 
Then, combining the above three estimates with \eqref{Qnrg3Wt3nQ}, we have 
\[f''(s)\leq \left(4\lambda^2C_{0}(1+\lambda^2A_{0}^2)e^{1+\lambda^2C_{0}}+2\lambda A_{0}\right)|X|_{G}^2=2K_{0}(\lambda)|X|_{G}^2\]
for all $s\in [0,1]$. 
Then, by the general fact that $f(1)\leq f(0)+f'(0)+C/2$ for smooth function on $[0,1]$ 
with $f''(s)\leq C$ on $[0,1]$, we have 
\[
g(F_{\ast}X,JF_{\ast}(\tilde{J}X))= f(1)\leq 0+(-|X|^2_{G})+K_{0}(\lambda)|X|_{G}^2. 
\]
This immediately implies the desired inequality with \eqref{UrhECp8hu4}. 
\end{proof}

Then, we immediately have the following. 

\begin{proposition}\label{uahAg1eiGU}
If $r$ satisfies $K_{0}(r)< 1$ (see \eqref{dSUX6lxl6M} for the definition of $K_{0}$), 
then $F:U_{r}(T^{\bot}L)\to M$ is a local diffeomorphism and 
a one-parameter family of 2-forms $\omega_{t}:=(1-t)\tilde{\omega}+tF^{\ast}\omega$ is 
non-degenerate on $U_{r}(T^{\bot}L)$ for each $t\in[0,1]$. 
\end{proposition}
\begin{proof}
Fix $v\in U_{r}(T^{\bot}L)$ and put $\lambda:=|v|$. Then, we have $\lambda<r$. 
Thus, $K_{0}(\lambda)\leq K_{0}(r)< 1$ since $K_{0}$ is increasing. 
Take $X\in T_{v}(U_{r}(T^{\bot}L))$ and assume $F_{\ast}X=0$. 
Then, by Proposition \ref{KU5oBVMKyH}, we have 
\[0=(F^{\ast}\omega)(X,\tilde{J}X)=(1-K_{0}(\lambda))|X|_{G}^2. \]
Thus, we see $X=0$ and this implies that $F$ is a local diffeomorphism from $U_{r}(T^{\bot}L)$ to $M$ since 
the domain and target have the same dimension. 
For the non-degeneracy of $\omega_{t}$ at $v$, take 
$X\in T_{v}(U_{r}(T^{\bot}L))$ and assume that $\omega_{t}(X,Y)=0$ for all $Y\in T_{v}(U_{r}(T^{\bot}L))$. 
Then, choosing $Y$ as $Y=\tilde{J}X$, we have 
\[
\begin{aligned}
0=\omega_{t}(X,\tilde{J}X)\geq & t|X|_{G}^2+(1-t)(1-K_{0}(\lambda))|X|_{G}^2\\
=& (1-(1-t)K_{0}(\lambda))|X|_{G}^2, 
\end{aligned}
\]
where the first inequality follows from \eqref{qIUTtqPF11} and Proposition \ref{KU5oBVMKyH}. 
Thus, we see $X=0$ from $1-(1-t)K_{0}(\lambda)>0$ and proved the non-degeneracy of $\omega_{t}$ on  $U_{r}(T^{\bot}L)$. 
\end{proof}

\begin{corollary}\label{qJG1rPvIQI}
If there exists $\varepsilon\in(0,1)$ so that $K_{0}(r)\leq 1-\varepsilon$, 
then we have 
\begin{equation}\label{wGiqpGQpvx}
|\omega_{t}(X,\tilde{J}X)|\geq \varepsilon |X|_{G}^2
\end{equation}
on $U_{r}(T^{\bot}L)$ for each $t\in[0,1]$. 
Especially, $\omega_{t}$ is non-degenerate on $U_{r}(T^{\bot}L)$ for each $t\in[0,1]$. 
\end{corollary}
\section{Estimates of the second derivative of \texorpdfstring{$F$}{F}}\label{bhas840blhv}
In addition to  \eqref{m5goLqTgtc}, we assume that there exist constants $C_{1},A_{1}\geq 0$ such that 
\begin{equation}\label{CPnfuwslXk}
|\nabla R_{M}|\leq C_{1}\quad\mbox{and}\quad|\nabla\mathop{\mathrm{II}}|\leq A_{1}
\end{equation}
in this section. 
As in the previous sections, let $F:T^{\bot}L\to M$ be a smooth map 
defined by $F(v)=\exp_{\pi(v)}v$. 
Then, its first derivative $F_{\ast}$ can be 
considered as a section of a 
vector bundle $F^{\ast}(TM)\otimes T^{\ast}(T^{\bot}L)$ over $T^{\bot}L$. 
Thus, we can consider its derivative $\nabla^{g\otimes G} F_{\ast}$, 
where $\nabla^{g\otimes G}$ is the induced connection of $F^{\ast}(TM)\otimes T^{\ast}(T^{\bot}L)$ from $g$ and $G$, and what we want to estimate in this section is its norm. 

Fix $\eta_{0}\in T^{\bot}L$ with $p_{0}:=\pi(\eta_{0})\in L$. 
We do every computation at $\eta_{0} \in T^{\bot}L$. 
Take $X,Y\in T_{\eta_{0}}(T^{\bot}L)$ arbitrarily. 
Then, we want to estimate $(\nabla^{g\otimes G} F_{\ast})(Y,X)$ from above by $C|X|_{G}|Y|_{G}$ 
with some positive constant $C$. 

Before doing that, we prepare the following lemma. 

\begin{lemma}\label{jtcnFcfmuG}
Let $\tau:[0,L]\ni s\mapsto \tau(s)\in M$ be a geodesic. 
Put $\Omega:=[0,L]\times(-\varepsilon,\varepsilon)$ for some $\varepsilon>0$. 
Let $\Phi:\Omega\to M$ be 
a smooth map so that $s\mapsto \Phi(s,t)=:\tau_{t}(s)$ is a geodesic in $(M,g)$ for each $t\in(-\varepsilon,\varepsilon)$ and $\tau_{0}\equiv \tau$. 
Let $Z$ be a smooth section of the pull-back bundle $\Phi^{\ast}(TM)$ over $\Omega$ 
so that $s\mapsto Z(s,t):=Z_{t}(s)$ satisfies the Jacobi field equation: 
\begin{equation}\label{tjD8yQXjvu}
\nabla_{s}\nabla_{s}Z_{t}+\mathop{R_{M}}(Z_{t},\dot{\tau}_{t})\dot{\tau}_{t}=0
\end{equation}
on a geodesic $s\mapsto \tau_{t}(s)$ for each $t\in(-\varepsilon,\varepsilon)$. Then, $\nabla_{t}$-derivative of $Z$ at $t=0$, denoted by 
$\mathcal{Z}(s):=\nabla_{\partial/\partial t|_{t=0}}Z(s,t)$ satisfies
\[
|\nabla_{s}\nabla_{s}\mathcal{Z}+\mathop{R_{M}}(\mathcal{Z},\dot{\tau})\dot{\tau}|\leq 2|\nabla R_{M}||\dot{\tau}|^2|W||Z_{0}|+2|R_{M}||\dot{\tau}|\left(|\nabla_{s}W||Z_{0}|+|W||\nabla_{s}Z_{0}|\right)
\]
on $\tau$, where $W(s):=\partial\Phi(s,t)/\partial t|_{t=0}$. 
\end{lemma}
\begin{remark}
The following proof is essentially the same as a part of the proof of Theorem 4.2 of \cite{MR1121230}. 
We just put the proof here for the reader's convenience. 
\end{remark}
\begin{proof}
Simply taking the $\nabla_{t}$ derivative of \eqref{tjD8yQXjvu} (at $t=0$), 
we have 
\begin{equation}\label{GMo57QGVVt}
\nabla_{t}(\nabla_{s}\nabla_{s}Z_{t})+\nabla_{t}(\mathop{R_{M}}(Z_{t},\dot{\tau}_{t})\dot{\tau}_{t})=0. 
\end{equation}
For the first term on the left-hand side of \eqref{GMo57QGVVt}, 
one can easily see that
\[
\begin{aligned}
\nabla_{t}(\nabla_{s}\nabla_{s}Z_{t})=&\nabla_{s}\nabla_{t}\nabla_{s}Z_{t}+
\mathop{R_{M}}(W,\dot{\tau})\nabla_{s}Z_{0}\\
=&\nabla_{s}(\nabla_{s}\nabla_{t}Z_{t}+\mathop{R_{M}}(W,\dot{\tau})Z_{0})+
\mathop{R_{M}}(W,\dot{\tau})\nabla_{s}Z_{0}\\
=&\nabla_{s}\nabla_{s}\mathcal{Z}+(\mathop{\nabla_{\dot{\tau}}R_{M}})(W,\dot{\tau})Z_{0}+\mathop{R_{M}}(\nabla_{s}W,\dot{\tau})Z_{0}
+2\mathop{R_{M}}(W,\dot{\tau})\nabla_{s}Z_{0}, 
\end{aligned}
\]
where we used $\nabla_{s}\dot{\tau}=0$ in the third equality. 
For the second term on the left-hand side of \eqref{GMo57QGVVt}, 
one can easily see that
\[
\begin{aligned}
\nabla_{t}(\mathop{R_{M}}(Z_{t},\dot{\tau}_{t})\dot{\tau}_{t})=&
(\mathop{\nabla_{W}R_{M}})(Z_{0},\dot{\tau})\dot{\tau}+\mathop{R_{M}}(\mathcal{Z},\dot{\tau})\dot{\tau}\\
&+\mathop{R_{M}}(Z_{0},\nabla_{s}W)\dot{\tau}+\mathop{R_{M}}(Z_{0},\dot{\tau})\nabla_{s}W, 
\end{aligned}
\]
where we used $\nabla_{t}\dot{\tau}_{t}=\nabla_{s}W$ at $t=0$. 
Thus, we have 
\[
\begin{aligned}
    \nabla_{s}\nabla_{s}\mathcal{Z}+\mathop{R_{M}}(\mathcal{Z},\dot{\tau})\dot{\tau}=&
-(\mathop{\nabla_{\dot{\tau}}R_{M}})(W,\dot{\tau})Z_{0}
-\mathop{R_{M}}(\nabla_{s}W,\dot{\tau})Z_{0}
-2\mathop{R_{M}}(W,\dot{\tau})\nabla_{s}Z_{0}\\
&-(\mathop{\nabla_{W}R_{M}})(Z_{0},\dot{\tau})\dot{\tau}
-\mathop{R_{M}}(Z_{0},\nabla_{s}W)\dot{\tau}
-\mathop{R_{M}}(Z_{0},\dot{\tau})\nabla_{s}W\\
=&-(\mathop{\nabla_{\dot{\tau}}R_{M}})(W,\dot{\tau})Z_{0}
-2\mathop{R_{M}}(W,\dot{\tau})\nabla_{s}Z_{0}\\
&-(\mathop{\nabla_{W}R_{M}})(Z_{0},\dot{\tau})\dot{\tau}
-2\mathop{R_{M}}(Z_{0},\dot{\tau})\nabla_{s}W, 
\end{aligned}
\]
where the second equality follows from the first Bianchi identity of $R_{M}$. 
Since the norm of the right-hand side is bounded from above by 
\[2|\nabla R_{M}||\dot{\tau}|^2|W||Z_{0}|+2|R_{M}||\dot{\tau}||\nabla_{s}W||Z_{0}|+2|R_{M}||\dot{\tau}||W||\nabla_{s}Z_{0}|, \]
the proof is complete. 
\end{proof}

\begin{lemma}\label{wG0iM7ukR1}
Put $\lambda:=|\eta_{0}|$. 
Then, we have
\[|(\nabla^{g\otimes G} F_{\ast})(Y,X)|\leq K_{1}(\lambda)|X|_{G}|Y|_{G}, \]
where 
\begin{equation}\label{5rpmFJjLre}
\begin{aligned}
K_{1}(\lambda):=&\left(2A_{0}+2(C_{0}+2A_{0}^2)\lambda+A_{1}\lambda+\frac{1}{2}A_{0}(C_{0}+2A_{0}^2)\lambda^2\right.\\
&\left.+\left(2C_{1}\lambda^2+4C_{0}\lambda\right)(1+A_{0}^2\lambda^2)e^{1+\lambda^2C_{0}}\right)e^{1+C_{0}\lambda^2/2}
\end{aligned}
\end{equation}
which is monotone increasing and satisfies $K_{1}(\lambda)\to 2eA_{0}$ as $\lambda\to 0$. 
\end{lemma}

\begin{remark}\label{FPtBaOx1oS}
One can see that 
\[\lambda K_{1}(\lambda)\geq 2eK_{0}(\lambda)\]
for all $\lambda\geq 0$. See \eqref{dSUX6lxl6M} for the definition of $K_{0}$. We use this inequality later. 
\end{remark}

\begin{proof}
Put $A:=\pi_{\ast}Y$ and $B:=K(Y)$. 
Let $\eta:(-\varepsilon',\varepsilon')\ni\alpha \mapsto \eta(\alpha)\in T^{\bot}L$ be a curve so that $\eta(0)=\eta_{0}$ and $\dot{\eta}(0)=Y$, and 
let $X_{\alpha}$ be the parallel transport of $X$ along $\eta(\alpha)$ with respect to $G$. 
Then, we have 
\begin{equation}\label{rdvCLoJEQQ}
(\nabla^{g\otimes G} F_{\ast})(Y,X)=\nabla_{F_{\ast}Y}(F_{\ast}X_{\alpha})-F_{\ast}(\nabla^{G}_{Y}X_{\alpha})=\nabla_{F_{\ast}\dot{\eta}(0)}(F_{\ast}X_{\alpha}). 
\end{equation}
Put $p(\alpha)=\pi\circ \eta(\alpha)\in L$. As in the proof of Lemma \ref{Y7hcfipucS}, 
put $U_{\alpha}:=\pi_{\ast}X_{\alpha}\in T_{p(\alpha)}L$ and $W_{\alpha}:=K(X_{\alpha})\in T^{\bot}_{p(\alpha)}L$. 
Let $\gamma_{\alpha}(t)\in L$ be a sufficiently short curve in $L$ defined for $t\in (-\varepsilon,\varepsilon)$ such that 
$\gamma_{\alpha}(0)=p(\alpha)$ and $\dot{\gamma}_{\alpha}(0)=U_{\alpha}$. 
Let $\eta(\alpha)_{\gamma_{\alpha}}(t)$ and $(W_{\alpha})_{\gamma_{\alpha}}(t)$ be 
the parallel transports of $\eta(\alpha)$ and $W_{\alpha}$ in $T^{\bot}L$ along $\gamma_{\alpha}$ with respect to $\nabla^{\bot}$, respectively. 
For each $\alpha\in (-\varepsilon',\varepsilon')$, define a curve $c_{\alpha}$ in $T^{\bot}L$ by 
\[c_{\alpha}(t)=\eta(\alpha)_{\gamma_{\alpha}}(t)+t\cdot (W_{\alpha})_{\gamma_{\alpha}}(t)\in T^{\bot}_{\gamma_{\alpha}(t)}L, \]
which satisfies $c_{\alpha}(0)=\eta(\alpha)$. 
Then, as in the proof of Lemma \ref{Y7hcfipucS}, we have $\dot{c}_{\alpha}(0)=X_{\alpha}$. 
For $s\in [0,1]$, we define 
\[\tilde{X}(\alpha,s):=\frac{d}{dt}\bigg|_{t=0}F(s\cdot c_{\alpha}(t))=\frac{d}{dt}\bigg|_{t=0}\exp_{\gamma_{\alpha}(t)}(s\cdot c_{\alpha}(t)). \]
Then, $F_{\ast}(X_{\alpha})=\tilde{X}(\alpha,1)$. 
Put $\Phi:(-\varepsilon',\varepsilon')\times [0,1]\times(-\varepsilon,\varepsilon)\to M$ by $\Phi(\alpha,s,t):=F(s\cdot c_{\alpha}(t))=\exp_{\gamma_{\alpha}(t)}(s\cdot c_{\alpha}(t))$. 
Since $\nabla_{\alpha}(\partial_{t}\Phi)=\nabla_{t}(\partial_{\alpha}\Phi)$ in general, we have 
\[
\begin{aligned}
\nabla_{F_{\ast}\dot{\eta}(0)}(F_{\ast}X_{\alpha})=&\nabla_{\partial/\partial\alpha|_{\alpha=0}}\left(\frac{\partial}{\partial t}\bigg|_{t=0}\Phi(\alpha,1,t)\right)\\
=&\nabla_{\partial/\partial t|_{t=0}}\left(\frac{\partial}{\partial\alpha}\bigg|_{\alpha=0}\Phi(\alpha,1,t)\right). 
\end{aligned}
\]
Hence, putting 
\begin{equation}\label{arrLW1ISgT}
Z(s,t):=\frac{\partial}{\partial\alpha}\bigg|_{\alpha=0}\Phi(\alpha,s,t)=\frac{\partial}{\partial\alpha}\bigg|_{\alpha=0}\exp_{\gamma_{\alpha}(t)}(s\cdot c_{\alpha}(t))
\end{equation}
and 
\[
\mathcal{Z}(s):=\nabla_{\partial/\partial t|_{t=0}}Z(s,t), 
\]
we have 
\begin{equation}\label{jsUFf052pU}
\nabla_{F_{\ast}\dot{\eta}(0)}(F_{\ast}X_{\alpha})=\mathcal{Z}(1). 
\end{equation}

Since $s\mapsto Z(s,t)$ is a variational vector field of a geodesic in $M$, 
by putting $Z_{t}(s):=Z(s,t)$ and $\tau_{t}(s):=\Phi(0,s,t)$, we have the Jacobi field equation: 
\[\nabla_{s}\nabla_{s}Z_{t}+\mathop{R_{M}}(Z_{t},\dot{\tau}_{t})\dot{\tau}_{t}=0\]
for each fixed $t\in (-\varepsilon,\varepsilon)$. 
Define another variational vector field $\tilde{W}$ along $\tau_{0}$ by 
\[\tilde{W}(s):=\frac{\partial}{\partial t}\bigg|_{t=0}\Phi(0,s,t)=\frac{\partial}{\partial t}\bigg|_{t=0}\exp_{\gamma_{0}(t)}(s\cdot c_{0}(t)). \]
Then, $\Phi(0,s,t)$, $Z_{t}$, $\mathcal{Z}$ and $\tilde{W}$ fit the condition of Lemma \ref{jtcnFcfmuG}, and we have 
\begin{equation}\label{7lAAkrpyLb}
\begin{aligned}
&|\nabla_{s}\nabla_{s}\mathcal{Z}+\mathop{R_{M}}(\mathcal{Z},\dot{\tau}_{0})\dot{\tau}_{0}|\\
\leq & 2C_{1}\lambda^2|\tilde{W}||Z_{0}|
+2C_{0}\lambda\left(|\nabla_{s}\tilde{W}||Z_{0}|+|\tilde{W}||\nabla_{s}Z_{0}|\right), 
\end{aligned}
\end{equation}
where we used $\tau_{0}(s)=\Phi(0,s,0)=\exp_{p_{0}}(s\cdot\eta_{0})$ and hence $|\dot{\tau}_{0}|=|\eta_{0}|=\lambda$. 

First, we will estimate the right-hand side of \eqref{7lAAkrpyLb}. 
One can easily see that the initial condition of $Z_{t}$ is 
\[
Z_{t}(0)=\frac{\partial}{\partial\alpha}\bigg|_{\alpha=0}\gamma_{\alpha}(t)
\quad\mbox{and}\quad \nabla_{s}Z_{t}(0)=\nabla_{\partial/\partial\alpha|_{\alpha=0}}c_{\alpha}(t). 
\]
In particular, since $\gamma_{\alpha}(0)=p(\alpha)$ and 
$c_{\alpha}(0)=\eta(\alpha)$ when $t=0$, we have 
\[Z_{0}(0)=A
\quad\mbox{and}\quad \nabla_{s}Z_{0}(0)=S_{\eta_{0}}(A)+B, \]
where $S_{\eta_{0}}:T_{p_{0}}L\to T_{p_{0}}L$ is the shape operator 
with respect to $\eta_{0}\in T^{\bot}_{p_{0}}L$. 
This implies that 
\[
|Z_{0}(0)|^2+|\nabla_{s}Z_{0}(0)|^2
\leq |A|^2+A_{0}^2\lambda^2|A|^2+|B|^2\leq (1+A_{0}^2\lambda^2)|Y|_{G}^2, 
\]
where we recall that the constant $A_{0}$ is an upper bound of the norm of the second fundamental form $\mathop{\mathrm{II}}$. 
Hence, by Proposition \ref{l065WM0Tcg} (with $D_{0}=0$), we have 
\begin{equation}\label{nID6tWxLiq}
|Z_{0}(s)|^2+|\nabla_{s}Z_{0}(s)|^2
\leq (1+A_{0}^2\lambda^2)e^{1+\lambda^2C_{0}}|Y|_{G}^{2}
\end{equation}
for all $s\in[0,1]$. 
On the other hand, for $\tilde{W}$, we have 
\[
\tilde{W}(0)=U_{0}\quad\mbox{and}\quad
\nabla_{s}\tilde{W}(0)=S_{\eta_{0}}(U_{0})+W_{0}, 
\]
and by the same argument as $Z_{0}(s)$ we have 
\begin{equation}\label{3bIChIe3po}
|\tilde{W}(s)|^2+|\nabla_{s}\tilde{W}(s)|^2
\leq (1+A_{0}^2\lambda^2)e^{1+\lambda^2C_{0}}|X|_{G}^{2}
\end{equation}
for all $s\in[0,1]$. 
Then, by \eqref{7lAAkrpyLb}, \eqref{nID6tWxLiq} and \eqref{3bIChIe3po}, we have 
\begin{equation}\label{vufKPK8UaG}
\begin{aligned}
&|\nabla_{s}\nabla_{s}\mathcal{Z}+\mathop{R_{M}}(\mathcal{Z},\dot{\tau}_{0})\dot{\tau}_{0}|\\
\leq &\left(2C_{1}\lambda^2+4C_{0}\lambda\right)(1+A_{0}^2\lambda^2)e^{1+\lambda^2C_{0}}|X|_{G}|Y|_{G}=:D_{0}. 
\end{aligned}
\end{equation}

Next, we need to estimate the norm of the initial condition of $\mathcal{Z}(s)$. 
By differentiating \eqref{arrLW1ISgT} by $\nabla_{t}$ at $t=0$ and using $\nabla_{t}\partial_{\alpha}\Phi=\nabla_{\alpha}\partial_{t}\Phi$, we have 
\[
\mathcal{Z}(0)=\nabla_{A}U_{\alpha}=\bar{\nabla}_{A}U_{\alpha}+\mathop{\mathrm{II}}(A,U_{0}), 
\]
where we recall that $\bar{\nabla}$ is the Levi--Civita connection on $(L,g|_{L})$. 
Here, recalling that $X_{\alpha}$ is parallel on $\eta$ with $\dot{\eta}(0)=Y$ 
and applying Lemma \ref{RT2JlOvGJJ}, we have 
\[
\mathcal{Z}(0)=-\frac{1}{2}\mathop{R_{L}}(J\eta_{0},JW_{0})A-\frac{1}{2}\mathop{R_{L}}(J\eta_{0},JB)U_{0}+\mathop{\mathrm{II}}(A,U_{0}), 
\]
where $\mathop{R_{L}}$ is the Riemannian curvature tensor of $(L,g|_{L})$. 
Then, with \eqref{K2wfDjgUgB}, we have 
\begin{equation}\label{jXkkr2KU7f}
\begin{aligned}
|\mathcal{Z}(0)|\leq & \frac{1}{2}(C_{0}+2A_{0}^2)\lambda|X|_{G}|Y|_{G}+A_{0}|X|_{G}|Y|_{G}\\
\leq & \left(A_{0}+\frac{1}{2}(C_{0}+2A_{0}^2)\lambda\right)|X|_{G}|Y|_{G}. 
\end{aligned}
\end{equation}

For $\nabla_{s}\mathcal{Z}(0)$, by differentiating \eqref{arrLW1ISgT} by $\nabla_{s}\nabla_{t}$ at $(s,t)=(0,0)$ 
and using 
\[
\begin{aligned}
\nabla_{s}\nabla_{t}\partial_{\alpha}\Phi
=&\nabla_{s}\nabla_{\alpha}\partial_{t}\Phi\\
=&\nabla_{\alpha}\nabla_{s}\partial_{t}\Phi+\mathop{R_{M}}(\partial_{s}\Phi,\partial_{\alpha}\Phi)\partial_{t}\Phi\\
=&\nabla_{\alpha}\nabla_{t}\partial_{s}\Phi+\mathop{R_{M}}(\eta_{0},A)U_{0}, 
\end{aligned}
\]
one can easily see that 
\[
\nabla_{s}\mathcal{Z}(0)=\nabla_{A}(S_{\eta(\alpha)}(U_{\alpha})+W_{\alpha})+\mathop{R_{M}}(\eta_{0},A)U_{0}, 
\]
and this can be further computed as 
\[
\begin{aligned}
\nabla_{s}\mathcal{Z}(0)=&\bar{\nabla}_{A}(S_{\eta(\alpha)}(U_{\alpha}))+\mathop{\mathrm{II}}(S_{\eta_{0}}(U_{0}),A)+\nabla^{\bot}_{A}W_{\alpha}+S_{W_{0}}(A)+\mathop{R_{M}}(\eta_{0},A)U_{0}\\
=&(\nabla_{A}S)_{\eta_{0}}(U_{0})+S_{B}(U_{0})+S_{\eta_{0}}(\bar{\nabla}_{A}U_{\alpha})\\
&+\mathop{\mathrm{II}}(S_{\eta_{0}}(U_{0}),A)+\nabla^{\bot}_{A}W_{\alpha}+S_{W_{0}}(A)+\mathop{R_{M}}(\eta_{0},A)U_{0}. 
\end{aligned}
\]
Recalling again that $X_{\alpha}$ is parallel on $\eta$ with $\dot{\eta}(0)=Y$ 
and applying Lemma \ref{RT2JlOvGJJ}, we have 
\[
\begin{aligned}
\nabla_{s}\mathcal{Z}(0)=&(\nabla_{A}S)_{\eta_{0}}(U_{0})+S_{B}(U_{0})\\
&+S_{\eta_{0}}\left(-\frac{1}{2}\mathop{R_{L}}(J\eta_{0},JW_{0})A-\frac{1}{2}\mathop{R_{L}}(J\eta_{0},JB)U_{0}\right)\\
&+\mathop{\mathrm{II}}(S_{\eta_{0}}(U_{0}),A)-\frac{1}{2}J(\mathop{R_{L}}(A,U_{0})(J\eta_{0}))+S_{W_{0}}(A)+\mathop{R_{M}}(\eta_{0},A)U_{0}. 
\end{aligned}
\]
Then, with \eqref{K2wfDjgUgB}, we have 
\begin{equation}\label{gqbYeVHFjc}
\begin{aligned}
|\nabla_{s}\mathcal{Z}(0)|\leq& A_{1}\lambda|X|_{G}|Y|_{G}+\frac{1}{2}A_{0}(C_{0}+2A_{0}^2)\lambda^2|X|_{G}|Y|_{G}+A_{0}^2\lambda|X|_{G}|Y|_{G}\\
&+\frac{1}{2}(C_{0}+2A_{0}^2)\lambda|X|_{G}|Y|_{G}+A_{0}|X|_{G}|Y|_{G}+C_{0}\lambda|X|_{G}|Y|_{G}\\
\leq & \left(A_{0}+\frac{3}{2}(C_{0}+2A_{0}^2)\lambda+A_{1}\lambda+\frac{1}{2}A_{0}(C_{0}+2A_{0}^2)\lambda^2\right)|X|_{G}|Y|_{G}. 
\end{aligned}
\end{equation}
Combining, \eqref{vufKPK8UaG}, \eqref{jXkkr2KU7f} and \eqref{gqbYeVHFjc}, we have 
\begin{equation}\label{GI3xrAvJmm}
\begin{aligned}
& |\mathcal{Z}(0)|^2+|\nabla_{s}\mathcal{Z}(0)|^2+D_{0}^2\\
\leq & (|\mathcal{Z}(0)|+|\nabla_{s}\mathcal{Z}(0)|+D_{0})^2\\
\leq & \left(2A_{0}+2(C_{0}+2A_{0}^2)\lambda+A_{1}\lambda+\frac{1}{2}A_{0}(C_{0}+2A_{0}^2)\lambda^2\right.\\
&\left.+\left(2C_{1}\lambda^2+4C_{0}\lambda\right)(1+A_{0}^2\lambda^2)e^{1+\lambda^2C_{0}}\right)^2|X|_{G}^2|Y|_{G}^2\\
=:& \left(E_{0}(\lambda)\right)^2|X|_{G}^2|Y|_{G}^2. 
\end{aligned}
\end{equation}
Then, by Proposition \ref{l065WM0Tcg} (with $D_{0}$ defined in \eqref{vufKPK8UaG} and $\varepsilon=0$) with \eqref{GI3xrAvJmm}, we have 
\[
|\mathcal{Z}(s)|^2+|\nabla_{s}\mathcal{Z}(s)|^2\leq\left(E_{0}(\lambda)\right)^2e^{2+C_{0}\lambda^2}|X|_{G}^2|Y|_{G}^2. 
\]
In particular, we have 
\[
|\mathcal{Z}(s)|\leq E_{0}(\lambda)e^{1+C_{0}\lambda^2/2}|X|_{G}|Y|_{G}. 
\]
Then, by evaluating the left-hand side at $s=1$ with \eqref{rdvCLoJEQQ} and \eqref{jsUFf052pU}, 
the proof is complete. 
\end{proof}

This immediately implies the following. 

\begin{proposition}\label{qr8e5Glehm}
We have 
\[|\nabla^{G}(F^{\ast}\omega)|_{G}\leq 2\sqrt{1+\lambda^2 A_{0}^2}e^{1/2+\lambda^2C_{0}/2}K_{1}(\lambda)\]
at $v\in U_{r}(T^{\bot}L)$ with $\lambda:=|v|$ (see \eqref{5rpmFJjLre} for the definition of $K_{1}(\lambda)$). 
\end{proposition}
\begin{proof}
For $v\in U_{r}(T^{\bot}L)$, take $Y,Z,W\in T_{v}(T^{\bot}L)$ arbitrarily. 
Put $p:=\pi(v)\in L$. 
Let $\gamma:(-\varepsilon,\varepsilon)\to L$ be a geodesic in $(L,g|_{L})$ 
so that $\gamma(0)=p$ and $\gamma'(0)=\pi_{\ast}Y\in T_{p}L$ and 
let $\gamma^{h}$ be the horizontal lift of $\gamma$ through $v$, that is, 
$\gamma^{h}(s)\in T_{\gamma(s)}^{\bot}L$ is the parallel transport of $v$ along $\gamma(s)$ 
with respect to the normal connection $\nabla^{\bot}$. 
On the other hand, since $K(Y)$ is in $T_{p}^{\bot}L$, we can also construct 
the parallel transport of $K(Y)$ along $\gamma$ 
with respect to the normal connection $\nabla^{\bot}$ and 
denote the resulting vector by $\gamma^{v}(s)\in T_{\gamma(s)}^{\bot}L$. 
Then, we get a curve $c$ in $T^{\bot}L$ with $c(0)=v$ by 
\[c(s):=\gamma^{h}(s)+s\gamma^{v}(s)\in T_{\gamma(s)}^{\bot}L. \]
Then, we see that 
\[
\pi_{\ast}\dot{c}(0)=\dot{\gamma}(0)=\pi_{\ast}Y\quad\mbox{and}\quad K(\dot{c}(0))=\nabla^{\bot}_{\dot{\gamma}(0)}c=\gamma^{v}(0)=K(Y). 
\]
This means that $\dot{c}(0)=Y$. 
Let $Z_{s}$ and $W_{s}$ in $T_{c(s)}(T^{\bot}L)$ be the parallel transports 
of $Z$ and $W$ along $c$ with respect to $\nabla^{G}$, respectively. 
Then, we have 
\[
\begin{aligned}
(\nabla^{G}_{Y}(F^{\ast}\omega))(Z,W)=&\frac{d}{ds}\bigg|_{s=0}\omega(F_{\ast}Z_{s},F_{\ast}W_{s})\\
=&\omega((\nabla^{g\otimes G}F_{\ast})(Y,Z),F_{\ast}W)+\omega(F_{\ast}Z,(\nabla^{g\otimes G}F_{\ast})(Y,W)), 
\end{aligned}
\]
where the second equality follows from that $\omega$ is parallel with respect to $\nabla$. 
Then, by Lemma \ref{Y7hcfipucS} and Lemma \ref{wG0iM7ukR1}, we have 
\[
|(\nabla^{G}_{Y}(F^{\ast}\omega))(Z,W)|\leq  2\sqrt{1+\lambda^2 A_{0}^2}e^{1/2+\lambda^2C_{0}/2}K_{1}(\lambda)|Y|_{G}|Z|_{G}|W|_{G}. 
\]
This completes the proof. 
\end{proof}
\section{Estimates for the scaling map \texorpdfstring{$\rho_{t}$}{}}\label{gyfsw5bdsvd}
We define the scaling map $\rho_{t}:U_{r}(T^{\bot}L) \to U_{r}(T^{\bot}L)$ by $\rho_{t}(\xi):=t\xi$ 
for $t\in [0,1]$. Then, for $t\in (0,1]$, $\rho_{t}$ defines a vector field on $U_{r}(T^{\bot}L)$ denoted by $\dot{\rho}_{t}$ by 
\[\dot{\rho}_{t}(v):=\frac{d}{ds}\bigg\vert_{s=t}\rho_{s}(\rho_{t}^{-1}(v))\quad\in T_{v}(U_{r}(T^{\bot}L)). \]
In this section, we give estimates for $\dot{\rho}_{t}$, $\rho_{t\ast}$ and those derivatives. 

\begin{proposition}\label{si5olAun74}
For $v\in U_{r}(T^{\bot}L)$ and $X\in T_{v}(U_{r}(T^{\bot}L))$, we have 
\[
\begin{aligned}
&|\rho_{t\ast}X|_{G}\leq |X|_{G}, \\
&|\dot{\rho}_{t}(v)|_{G}=|v|/t, \\
&\tilde{\omega}(\dot{\rho}_{t}(tv),\rho_{t\ast}X)\leq |v||X|_{G}. 
\end{aligned}
\]
\end{proposition}
\begin{proof}
Let $p:=\pi(v)\in L$ and let $(x^{1},\dots,x^{n},\xi^{1},\dots,\xi^{n})$ be a 
local coordinate system on $(T^{\bot}L)|_{U}$ with $p\in U$ defined by \eqref{VahfBGGECR}. 
We assume that the coordinates of $v$ is $(x,\xi)$. 
Express $X$ as 
\[X=\sum_{i=1}^{n}A^{i}\frac{\partial}{\partial x^{i}}\bigg\vert_{(x,\xi)}+\sum_{i=1}^{n}B^{i}\frac{\partial}{\partial \xi^{i}}\bigg\vert_{(x,\xi)}. \]
Then, it is easy to see that 
\[\rho_{t\ast}X=\sum_{i=1}^{n}A^{i}\frac{\partial}{\partial x^{i}}\bigg\vert_{(x,t\xi)}+t\sum_{i=1}^{n}B^{i}\frac{\partial}{\partial \xi^{i}}\bigg\vert_{(x,t\xi)}
\quad\mbox{and}\quad
\dot{\rho}_{t}(v)=\sum_{i=1}^{n}\frac{\xi^{i}}{t}\frac{\partial}{\partial \xi^{i}}\bigg\vert_{(x,\xi)}. \]
Then, by \eqref{O2Lhl6lar7}, we have 
\begin{equation}\label{norHumb5qL}
\pi_{\ast}(\rho_{t\ast}X)=\pi_{\ast}(X),\,\, K(\rho_{t\ast}X)=tK(X),\,\,\pi_{\ast}(\dot{\rho}_{t}(v))=0,\,\,K(\dot{\rho}_{t}(v))=v/t. 
\end{equation}
This implies that 
\[|\rho_{t\ast}X|_{G}^2=|\pi_{\ast}X|^2+t^2|K(X)|^2\leq |X|_{G}^2\quad\mbox{and}\quad |\dot{\rho}_{t}(v)|_{G}^2=|v|^2/t^2\]
since $t\in[0,1]$. 
With the form of $\tilde{\omega}$ by \eqref{duYHiTmOjx}, we have 
\[
\tilde{\omega}(\dot{\rho}_{t}(tv),\rho_{t\ast}X)=-\sum_{i,j=1}^{n}\bar{g}_{ij}(x)\xi^{j}A^{i}=g(Jv,\pi_{\ast}X)\leq |v| |X|_{G}, 
\]
and the proof is complete. 
\end{proof}

\begin{proposition}\label{uIl6Hns0dM}
Fix $t_{0}\in(0,1]$. 
For $\eta_{0}\in U_{r}(T^{\bot}L)$ and $X,Y\in T_{\eta_{0}}(U_{r}(T^{\bot}L))$, we have 
\[
\begin{aligned}
&|(\nabla^{G}\rho_{t_{0}\ast})(Y,X)|_{G}\leq (C_{0}+2A_{0}^2)|\eta_{0}||X|_{G}|Y|_{G}, \\
&|(\nabla^{G}_{\rho_{t_{0}\ast}Y}\dot{\rho}_{t_{0}})(t_{0}\eta_{0})|_{G}\leq |Y|_{G}. 
\end{aligned}
\]
\end{proposition}
\begin{proof}
Let $A$, $B$, $\eta$, $X_{\alpha}$, $p(\alpha)$, $U_{\alpha}$, $W_{\alpha}$, $\gamma_{\alpha}$, 
$\eta(\alpha)_{\gamma_{\alpha}}(t)$, $(W_{\alpha})_{\gamma_{\alpha}}(t)$ and $c_{\alpha}$ be those as 
in the proof of Lemma \ref{wG0iM7ukR1}. 
Then, we have 
\[
(\nabla^{G} \rho_{t_{0}\ast})(Y,X)=\nabla^{G}_{\rho_{t_{0}\ast}Y}(\rho_{t_{0}\ast}X_{\alpha})-\rho_{t_{0}\ast}(\nabla^{G}_{Y}X_{\alpha})=\nabla^{G}_{\rho_{t_{0}\ast}\dot{\eta}(0)}(\rho_{t_{0}\ast}X_{\alpha})
\]
since $\nabla^{G}_{Y}X_{\alpha}=0$. 
From $X_{\alpha}=[U_{\alpha}]^{h}+[W_{\alpha}]^{v}$ and \eqref{norHumb5qL}, we have 
\[\pi_{\ast}(\rho_{t_{0}\ast}X_{\alpha})=U_{\alpha},\,\, K(\rho_{t_{0}\ast}X_{\alpha})=t_{0}W_{\alpha}. \]
Similarly, from $\dot{\eta}(0)=Y=[A]^{h}+[B]^{v}$ and \eqref{norHumb5qL}, we have 
\[\pi_{\ast}(\rho_{t_{0}\ast}\dot{\eta}(0))=A,\,\, K(\rho_{t_{0}\ast}\dot{\eta}(0))=t_{0}B. \]
Then, by \eqref{aYmLTS3JTr}, we have 
\[
\begin{aligned}
\nabla^{G}_{\rho_{t_{0}\ast}\dot{\eta}(0)}(\rho_{t_{0}\ast}X_{\alpha})=&
[\nabla^{\bot}_{A}(t_{0}W_{\alpha})]^{v}_{t_{0}\eta_{0}}+\frac{1}{2}[\mathop{R_{L}}(J(t_{0}\eta_{0}),J(t_{0}W_{0}))A]^{h}_{t_{0}\eta_{0}}\\
&+\frac{1}{2}[\mathop{R_{L}}(J(t_{0}\eta_{0}),J(t_{0}B))U_{0}]^{h}_{t_{0}\eta_{0}}\\
&+[(\bar{\nabla}_{A}U_{\alpha})]^{h}_{t_{0}\eta_{0}}+\frac{1}{2}[J(\mathop{R_{L}}(A,U_{0})(J(t_{0}\eta_{0})))]^{v}_{t_{0}\eta_{0}}. 
\end{aligned}
\]
Applying Lemma \ref{RT2JlOvGJJ}, we have 
\[
\begin{aligned}
\bar{\nabla}_{A}U_{\alpha}=&-\frac{1}{2}\mathop{R_{L}}(J\eta_{0},JW_{0})A-\frac{1}{2}\mathop{R_{L}}(J\eta_{0},JB)U_{0},\\
\nabla^{\bot}_{A}W_{\alpha}=&-\frac{1}{2}J(\mathop{R_{L}}(A,U_{0})(J\eta_{0})). 
\end{aligned}
\]
Combining these, we have 
\[
\nabla^{G}_{\rho_{t_{0}\ast}\dot{\eta}(0)}(\rho_{t_{0}\ast}X_{\alpha})=
\frac{t_{0}^2-1}{2}[\mathop{R_{L}}(J\eta_{0},JW_{0})A]^{h}_{t_{0}\eta_{0}}
+\frac{t_{0}^2-1}{2}[\mathop{R_{L}}(J\eta_{0},JB)U_{0}]^{h}_{t_{0}\eta_{0}}, 
\]
and hence 
\[|\nabla^{G}_{\rho_{t_{0}\ast}\dot{\eta}(0)}(\rho_{t_{0}\ast}X_{\alpha})|_{G}
\leq(C_{0}+2A_{0}^2)|\eta_{0}||X|_{G}|Y|_{G}.\]
This implies the first inequality. 
Next, we will prove the second inequality. 
By \eqref{norHumb5qL}, we have $\dot{\rho}_{t_{0}}(t_{0}\eta(s))=[\eta(s)]_{t_{0}\eta(s)}^{v}$. 
Then, by \eqref{aYmLTS3JTr}, we have 
\[
(\nabla^{G}_{\rho_{t_{0}\ast}Y}\dot{\rho}_{t_{0}})(t_{0}\eta_{0})=[\nabla^{\bot}_{A}\eta(s)]^{v}_{t_{0}\eta_{0}}+\frac{1}{2}[\mathop{R_{L}}(J(t_{0}\eta_{0}),J\eta_{0}))A]^{h}_{t_{0}\eta_{0}}
=[B]^{v}_{t_{0}\eta_{0}}. 
\]
This immediately implies the second inequality. 
\end{proof}
\section{Estimates for the time-dependent vector field \texorpdfstring{$\mathcal{X}_{t}$}{}}\label{9089awgrenbva}
We continue to assume \eqref{m5goLqTgtc} and \eqref{CPnfuwslXk}. 
In this section, we assume that $r>0$ satisfies 
\begin{equation}\label{SxWXt7MWlg}
rK_{1}(r)\leq e. 
\end{equation}
See \eqref{5rpmFJjLre} for the definition of $K_{1}$. 
Then, by Remark \ref{FPtBaOx1oS}, we have 
\[
K_{0}(r)\leq \frac{1}{2}. 
\]
See \eqref{dSUX6lxl6M} for the definition of $K_{0}$. 

\begin{remark}\label{tgqnwueO6G}
Under the assumption \eqref{SxWXt7MWlg}, for $\lambda\in [0,r)$, we have 
\[(C_{0}+2A_{0}^2)\lambda^2\leq 1\quad\mbox{and}\quad (1+\lambda^2 A_{0}^2)e^{1+\lambda^2C_{0}}\leq 4. \]
Actually, by $K_{0}(\lambda)\leq 1/2$, we have $2\lambda^2C_{0}(1+\lambda^2 A_{0}^2)e^{1+\lambda^2C_{0}}\leq1/2$
and $\lambda A_{0}\leq 1/2$. The former inequality implies $\lambda^2 C_{0}\leq 1/(4e)$. 
Then, we have $(C_{0}+2A_{0}^2)\lambda^2\leq 1/(4e)+1/2=0.5919\dots\leq 1$ 
and $(1+\lambda^2 A_{0}^2)e^{1+\lambda^2C_{0}}\leq (5/4)\cdot e^{(1+(1/4e))}=3.725\dots\leq 4$. 
\end{remark}

Then, by Corollary \ref{qJG1rPvIQI} (with $\varepsilon=1/2$), we see that 
$\omega_{t}=(1-t)\tilde{\omega}+tF^{\ast}\omega$ is 
non-degenerate on $U_{r}(T^{\bot}L)$. Define a 1-form $\mu$ on $U_{r}(T^{\bot}L)$ by 
\[\mu:=\int_{0}^{1}\rho_{t}^{\ast}\left\{(F^{\ast}\omega)(\dot{\rho}_{t},\,\cdot\,)-\tilde{\omega}(\dot{\rho}_{t},\,\cdot\,)\right\}dt. \]
Then, we have 
\begin{equation}\label{VD7LQHHpox}
F^{\ast}\omega-\tilde{\omega}=\rho_{1}^{\ast}(F^{\ast}\omega-\tilde{\omega})=\int_{0}^{1}\frac{d}{dt}(\rho_{t}^{\ast}(F^{\ast}\omega-\tilde{\omega}))dt=d\mu, 
\end{equation}
where the second equality follows from that $F^{\ast}\omega$ coincides with $\tilde\omega$ on 
the zero section and the third equality follows from Cartan formula and the closedness of $\omega$ and $\tilde{\omega}$. 
We define a time-dependent vector field $\mathcal{X}_{t}$ ($t\in[0,1]$) on $U_{r}(T^{\bot}L)$ by 
\begin{equation}\label{VU46Mlxo3R}
\omega_{t}(\mathcal{X}_{t},\,\cdot\,)=-\mu, 
\end{equation}
where we used the fact that $\omega_{t}$ is 
non-degenerate on $U_{r}(T^{\bot}L)$. 
The purpose of this section is to give (time-independent) explicit bounds for 
$|\mathcal{X}_{t}|_{G}$ and $|\nabla^{G} \mathcal{X}_{t}|_{G}$, 
where $\nabla^{G}$ is the Levi--Civita connection 
of the Sasaki metric $G$ on $U_{r}(T^{\bot}L)$. 
\subsection{Bound for \texorpdfstring{$|\mathcal{X}_{t}|_{G}$}{the norm of the vector field}}
First, we deduce a bound for $|\mathcal{X}_{t}|_{G}$. 

\begin{proposition}\label{PUCutd5w8A}
For each $v\in U_{r}(T^{\bot}L)$ and $t\in[0,1]$, we have 
\[|\mathcal{X}_{t}(v)|_{G}\leq 10|v|. \]
\end{proposition}
\begin{proof}
By \eqref{wGiqpGQpvx} (with $\varepsilon=1/2$) and the definition of $\mathcal{X}_{t}$, we have 
\[|\mathcal{X}_{t}|_{G}^2\leq 2\omega_{t}(\mathcal{X}_{t},\tilde{J}\mathcal{X}_{t})=-2\mu(\tilde{J}\mathcal{X}_{t})\leq 2|\mu|_{G}|\mathcal{X}_{t}|_{G}. \]
Thus, we have $|\mathcal{X}_{t}|_{G}\leq 2|\mu|_{G}$ and it suffices to bound $|\mu|_{G}$. 

For $v\in U_{r}(T^{\bot}L)$, we can write it as $v=\lambda\eta$ 
by some $\eta\in T^{\bot}L$ with $|\eta|=1$ and $\lambda\in [0,r)$. 
Put $p:=\pi(\eta)\in L$. 
We do every computation at ($v=$) $\lambda\eta \in T^{\bot}L$. 
Take $X\in T_{\lambda\eta}(T^{\bot}L)$ arbitrarily. 
Then, we want to estimate $\mu(X)$ from above by $C|X|_{G}$ 
with some time-independent positive constant $C$. 
By the definition of $\mu$, we have 
\begin{equation}\label{6kIbnBSDx4}
\begin{aligned}
\mu(X)=&\int_{0}^{1}\left\{\omega(F_{\ast}\dot{\rho}_{t}(tv),F_{\ast}\rho_{t\ast}X)-\tilde{\omega}(\dot{\rho}_{t}(tv),\rho_{t\ast}X)\right\}dt\\
\leq& |F_{\ast}\dot{\rho}_{t}(tv)||F_{\ast}\rho_{t\ast}X|+|\tilde{\omega}(\dot{\rho}_{t}(tv),\rho_{t\ast}X)|. 
\end{aligned}
\end{equation}
For the second term, by Proposition \ref{si5olAun74}, we have 
\begin{equation}\label{89p23h5ug11}
\tilde{\omega}(\dot{\rho}_{t}(tv),\rho_{t\ast}X)\leq \lambda |X|_{G}
\end{equation}
since $|v|=\lambda$. 
For the first term of \eqref{6kIbnBSDx4}, 
by Proposition \ref{si5olAun74} and Lemma \ref{Y7hcfipucS}, we have 
\begin{equation}\label{NlewaPb11t}
|F_{\ast}\dot{\rho}_{t}(tv)||F_{\ast}\rho_{t\ast}X|\leq \lambda(1+\lambda^2 A_{0}^2)e^{1+\lambda^2C_{0}}|X|_{G}. 
\end{equation}
By Remark \ref{tgqnwueO6G}, we know that $(1+\lambda^2A_{0}^2)e^{1+\lambda^2C_{0}}\leq 4$ under 
the assumption \eqref{SxWXt7MWlg}. By using this and inserting \eqref{89p23h5ug11} and 
\eqref{NlewaPb11t} into \eqref{6kIbnBSDx4}, we have 
\[
\mu(X)\leq 5\lambda |X|_{G}
\]
and the proof is complete by recalling $|\mathcal{X}_{t}|_{G}\leq 2|\mu|_{G}$. 
\end{proof}
\subsection{Bound for \texorpdfstring{$|\nabla^{G}\mathcal{X}_{t}|_{G}$}{the norm of the first derivative of the vecotr field}}
Next, we deduce a bound for $|\nabla^{G} \mathcal{X}_{t}|_{G}$, 
where $\nabla^{G}$ is the Levi--Civita connection with respect to $G$ on $U_{r}(T^{\bot}L)$. 
Since $\mathcal{X}_{t}$ is defined by $\omega_{t}$ and $\mu$, we need to 
estimate $|\nabla^{G}\omega_{t}|_{G}$ and $|\nabla^{G}\mu|_{G}$ to obtain an 
estimate for $|\nabla^{G}\mathcal{X}_{t}|_{G}$. 

\begin{proposition}\label{DiqFj8gaGH}
We have 
\[|\nabla^{G}\mu|_{G}\leq 24\]
at every $v\in U_{r}(T^{\bot}L)$. 
\end{proposition}
\begin{proof}
Fix $v\in U_{r}(T^{\bot}L)$. We write it as $v=\lambda\eta$ by some $\eta\in T^{\bot}L$ with $|\eta|=1$ and $\lambda\in [0,r)$. 
Put $p:=\pi(\eta)\in L$. 
We do every computation at $v$ ($=\lambda\eta$). 
Take $Y,Z\in T_{\lambda\eta}(T^{\bot}L)$ arbitrarily. 
Let $c(s)$ be a curve in $T^{\bot}L$ with $c(0)=\lambda\eta$ and $\dot{c}(0)=Y$ defined by the same way as in the proof of Proposition \ref{qr8e5Glehm}. 
Let $Z_{s}\in T_{c(s)}(T^{\bot}L)$ be the parallel transport 
of $Z$ along $c$ with respect to $\nabla^{G}$. 
Then, we have 
\begin{equation}\label{amLwbNvE20}
(\nabla^{G}_{Y}\mu)(Z)=\frac{d}{ds}\bigg|_{s=0}\mu(Z_{s}), 
\end{equation}
and our purpose is equivalent to estimating the right-hand side. 
By the definition of $\mu$, we have 
\begin{equation}\label{idIAvUWE1v}
\mu(Z_{s})=\int_{0}^{1}\left\{(F^{\ast}\omega)(\dot{\rho}_{t}(tc(s)),\rho_{t\ast}Z_{s})-\tilde{\omega}(\dot{\rho}_{t}(tc(s)),\rho_{t\ast}Z_{s})\right\}dt. 
\end{equation}
First, we consider the $s$-derivative of the second term of the right-hand side of \eqref{idIAvUWE1v}: 
\[
\begin{aligned}
\frac{d}{ds}\bigg|_{s=0}\tilde{\omega}(\dot{\rho}_{t}(tc(s)),\rho_{t\ast}Z_{s})
=&(\nabla^{G}_{Y}\tilde{\omega})(\dot{\rho}_{t}(tv),\rho_{t\ast}Z)+\tilde{\omega}(\nabla^{G}_{\rho_{t\ast}Y}\dot{\rho}_{t}(t(c(s))),\rho_{t\ast}Z)\\
&+\tilde{\omega}(\dot{\rho}_{t}(tv),(\nabla^{G}\rho_{t\ast})(Y,Z)). 
\end{aligned}
\]
Then, by Proposition \ref{1Nq7NTU2sR}, Proposition \ref{si5olAun74} and Proposition \ref{uIl6Hns0dM} 
with $t\in[0,1]$, we have 
\begin{equation}\label{JmNYEgnq6s}
\begin{aligned}
&\bigg\vert\frac{d}{ds}\bigg|_{s=0}\tilde{\omega}(\dot{\rho}_{t}(tc(s)),\rho_{t\ast}Z_{s})\bigg\vert\\
\leq & \sqrt{2}(C_{0}+2A_{0}^2)\lambda^2|Y|_{G}|Z|_{G}+|Y|_{G}|Z|_{G}+(C_{0}+2A_{0}^2)\lambda^2|Y|_{G}|Z|_{G}\\
\leq & \left(1+(\sqrt{2}+1)(C_{0}+2A_{0}^2)\lambda^2\right)|Y|_{G}|Z|_{G}\\
\leq & 4|Y|_{G}|Z|_{G}, 
\end{aligned}
\end{equation}
where the second inequality follows from Remark \ref{tgqnwueO6G}. 

Next, we consider the $s$-derivative of the first term of the right-hand side of \eqref{idIAvUWE1v}: 
\[
\begin{aligned}
&\frac{d}{ds}\bigg|_{s=0}(F^{\ast}\omega)(\dot{\rho}_{t}(tc(s)),\rho_{t\ast}Z_{s})\\
=&(\nabla^{G}_{\rho_{t\ast}Y}(F^{\ast}\omega))(\dot{\rho}_{t}(tv),\rho_{t\ast}Z)+(F^{\ast}\omega)(\nabla^{G}_{\rho_{t\ast}Y}\dot{\rho}_{t}(tc(s)),\rho_{t\ast}Z)\\
&+(F^{\ast}\omega)(\dot{\rho}_{t}(tv),(\nabla^{G}\rho_{t\ast})(Y,Z)). 
\end{aligned}
\]
Then, by Lemma \ref{Y7hcfipucS}, Proposition \ref{qr8e5Glehm}, Proposition \ref{si5olAun74}, Proposition \ref{uIl6Hns0dM} and Remark \ref{tgqnwueO6G}, with $t\in [0,1]$, we have 
\begin{equation}\label{cEAv4QI52G}
\begin{aligned}
&\left\vert\frac{d}{ds}\bigg|_{s=0}(F^{\ast}\omega)(\dot{\rho}_{t}(tc(s)),\rho_{t\ast}Z_{s})\right\vert\\
\leq & 4K_{1}(\lambda)\lambda|Y|_{G}|Z|_{G}+4|Y|_{G}|Z|_{G}+4|Y|_{G}|Z|_{G}\\
\leq & 20|Y|_{G}|Z|_{G}, 
\end{aligned}
\end{equation}
where the second inequality follows from $\lambda\in [0,r)$ and $r$ satisfies \eqref{SxWXt7MWlg}. 
Then, applying \eqref{JmNYEgnq6s} and \eqref{cEAv4QI52G} to the $s$-derivative of \eqref{idIAvUWE1v} with \eqref{amLwbNvE20}, we have 
\[
|(\nabla^{G}_{Y}\mu)(Z)|\leq 24|Y|_{G}|Z|_{G}
\]
and this completes the proof. 
\end{proof}

\begin{proposition}\label{5XcWejwEVY}
We have 
\[|\nabla^{G}\omega_{t}|_{G}\leq \sqrt{2}(C_{0}+2A_{0}^2)\lambda + 4K_{1}(\lambda)\]
at every $v\in U_{r}(T^{\bot}L)$ and $t\in[0,1]$, where $\lambda:=|v|$. 
\end{proposition}
\begin{proof}
Since $\omega_{t}=(1-t)\tilde{\omega}+t(F^{\ast}\omega)$, we have 
\[
\nabla^{G}\omega_{t}=(1-t)\nabla^{G}\tilde{\omega}+t\nabla^{G}_{Y}(F^{\ast}\omega).  
\]
Then, by Proposition \ref{1Nq7NTU2sR} and Proposition \ref{qr8e5Glehm} 
with Remark \ref{tgqnwueO6G}, 
we have 
\[|\nabla^{G}\omega_{t}|_{G}\leq(1-t)\cdot \sqrt{2}(C_{0}+2A_{0}^2)\lambda +t\cdot 4K_{1}(\lambda)\]
and the proof is complete. 
\end{proof}

\begin{proposition}\label{rCbfDQNh0H}
We have 
\[|\nabla^{G}\mathcal{X}_{t}(v)|_{G}\leq 294\]
at every $v\in U_{r}(T^{\bot}L)$ and $t\in[0,1]$. 
\end{proposition}
\begin{proof}
For $v\in U_{r}(T^{\bot}L)$, take $Y,Z\in T_{v}(T^{\bot}L)$ arbitrarily. 
Let $c:(-\varepsilon,\varepsilon)\to U_{r}(T^{\bot}L)$ be a curve 
so that $c(0)=v$ and $\dot{c}(0)=Y$. 
Let $Z(s)$ be the parallel transport of $Z$ along $c$ with respect to $G$. 
Then, by deriving the defining equation \eqref{VU46Mlxo3R} with inserting $Z(s)$: $\omega_{t}(\mathcal{X}_{t}(c(s)),Z(s))=-\mu(Z(s))$ by $Y$, 
we have 
\begin{equation}\label{oid6DmFIYM}
(\nabla^{G}_{Y}\omega_{t})(\mathcal{X}_{t},Z)+\omega_{t}(\nabla^{G}_{Y}\mathcal{X}_{t},Z)=-(\nabla^{G}_{Y}\mu)(Z). 
\end{equation}
Letting $Z=\tilde{J(}\nabla^{G}_{Y}\mathcal{X}_{t})$ in \eqref{oid6DmFIYM} and using \eqref{wGiqpGQpvx} (with $\varepsilon=1/2$), we have 
\[
|\nabla^{G}_{Y}\mathcal{X}_{t}|_{G}^2\leq 2\omega_{t}(\nabla^{G}_{Y}\mathcal{X}_{t},Z)\leq 2|\nabla^{G}\omega_{t}|_{G}|Y|_{G}|\mathcal{X}_{t}|_{G}|\nabla^{G}_{Y}\mathcal{X}_{t}|_{G}+2|\nabla^{G}\mu|_{G}|Y|_{G}|\nabla^{G}_{Y}\mathcal{X}_{t}|_{G}. 
\]
Then, by Proposition \ref{PUCutd5w8A}, Proposition \ref{DiqFj8gaGH}, Proposition \ref{5XcWejwEVY} and Remark \ref{tgqnwueO6G}, we have 
\[
\begin{aligned}
|\nabla^{G}_{Y}\mathcal{X}_{t}|_{G}^2\leq & 2\left(10\lambda\left(\sqrt{2}(C_{0}+2A_{0}^2)\lambda + 4K_{1}(\lambda)\right)+24\right)|\nabla^{G}_{Y}\mathcal{X}_{t}|_{G}|Y|_{G}\\
\leq & 2(10(\sqrt{2}+4e)+24)|\nabla^{G}_{Y}\mathcal{X}_{t}|_{G}|Y|_{G}\\
\leq & 294|\nabla^{G}_{Y}\mathcal{X}_{t}|_{G}|Y|_{G}. 
\end{aligned}
\]
This implies that 
\[|\nabla^{G}\mathcal{X}_{t}(v)|_{G}\leq 294, \]
and the proof is complete. 
\end{proof}
\section{Estimates on a specific coordinate chart}\label{bv287adsggql}
In this section, for $r>0$, we put $\mathcal{M}_{r}:=U_{r}(T^{\bot}L)$ for notational simplicity. 
Fix $r_{0}>0$ and assume that there is a smooth vector field $\mathcal{X}$ on $\mathcal{M}_{r_{0}}$. 
Although the results of this chapter will be applied later in the case where $\mathcal{X}$ is defined by \eqref{VU46Mlxo3R}, we do not assume this here. 
Instead, we derive results for a general vector field $\mathcal{X}$ on $\mathcal{M}_{r_{0}}$.
For $v\in \mathcal{M}_{r_{0}}$, we put 
\[\mathcal{X}^{\mathcal{H}}(v):=\pi_{\ast}(\mathcal{X}(v))\in T_{\pi(v)}L\quad\mbox{and}\quad \mathcal{X}^{\mathcal{V}}(v):=K(\mathcal{X}(v))\in T_{\pi(v)}^{\bot}L. \]
We assume that there exist $\Lambda_{1},\Lambda_{2}\geq 0$ so that  
\begin{equation}\label{xBvMuiV2T0}
|\mathcal{X}(v)|_{G}\leq \Lambda_{1}|v|\quad\mbox{and}\quad |\nabla^{G}\mathcal{X}(v)|_{G}\leq \Lambda_{2}
\end{equation}
for all $v\in \mathcal{M}_{r_{0}}$. 
Although the constants $\Lambda_{1}$ and $\Lambda_{2}$ will be fixed later as those given in Propositions \ref{PUCutd5w8A} and \ref{rCbfDQNh0H}, respectively, we treat them as arbitrary constants throughout this chapter and derive various estimates.

In this section, we further assume that there exist $\bar{C}_{0}$, $\bar{C}_{1}$ and $\bar{C}_{2}\geq 0$ so that 
\[|R_{L}|\leq \bar{C}_{0}, \quad |\bar{\nabla} R_{L}|\leq \bar{C}_{1}\quad \mbox{and} \quad|\bar{\nabla}^2R_{L}|\leq \bar{C}_{2}.\] 
We take $\bar{C}_{1}=\bar{C}_{2}=0$ when $\bar{C}_{0}=0$. 
We define a monotone increasing polynomials $D_{0}(x)$ by 
\[
D_{0}(x):=\bar{C}_{2}x^2+6\bar{C}_{1}^2x^4+20\bar{C}_{0}\bar{C}_{1}x^3+17\bar{C}_{0}^2x^2+3\bar{C}_{1}x, 
\]
and take $r\in (0,r_{0}]$ so that the following conditions hold:
\begin{equation}\label{mND4S5rsdK}
D_{0}(r)\leq \bar{C}_{0}. 
\end{equation}
We remark that such $r$ always exists since $D_{0}(x)\to 0$ 
as $x\to 0$ when $\bar{C}_{0}>0$ and $D_{0}(x)=0$ for all $x$ 
when $\bar{C}_{0}=0$. 

For $p\in L$ and $R>0$, we denote $\{\,w\in T_{p}L\mid |w|<R\,\}$ and $\{\,q\in L\mid d_{L}(p,q)<R\,\}$ by $B_{p}(R)$ and $B_{L}(p,R)$, respectively. 
We denote by $E_{p}$, for $p\in L$, the exponential map $\exp_{p}^{L}:T_{p}L\to L$ at $p$ for short. 
Namely, $E_{p}:T_{p}L\to L$ is a smooth map defined by 
\[E_{p}(X):=\exp_{p}^{L}(X). \]
For $X,Y\in T_{p}L$, we put 
\[
\tilde{Y}(X):= (DE_{p})_{X}(Y)=\frac{d}{dt}\bigg\vert_{t=0}\exp_{p}^{L}(X+tY)\quad\in T_{E_{p}(X)}L. 
\]
This means that when we put a tilde on a tangent vector $Y\in T_{p}L$ then $\tilde{Y}$ becomes a section of a pull-back bundle $E_{p}^{\ast}(TL)$ on $T_{p}L$ by a map $E_{p}:T_{p}L\to L$.  
We define a smooth map $Q_{p}: T_{p}L\times T_{p}L\to T^{\bot}L$ by 
\[Q_{p}(X,Y):=J\tilde{Y}(X) \quad\in T^{\bot}_{E_{p}(X)}L. \]

\begin{lemma}\label{d3baVXrf81}
The derivative of the exponential map $E_{p}$ at each point $X\in B_{p}(r)$ is a linear isomorphism, the derivative of the smooth map $Q_{p}$ at each point in $B_{p}(r)\times T_{p}L$ is also a linear isomorphism and we have 
\[\frac{1}{2}|Y|\leq |Q_{p}(X,Y)|\leq 2|Y|\]
for each $(X,Y)\in B_{p}(r)\times T_{p}L$. 
\end{lemma}
\begin{proof}
We remark that the sectional curvature of $(L,g|_{L})$ is bounded from above by $\bar{C}_{0}$. 
Fix $(X,Y)\in B_{p}(r)\times T_{p}L$. 
For the upper bound of $|Q_{p}(X,Y)|$, it follows from the estimate of the first derivative of the exponential map (see the first inequality of Corollary \ref{M7HIIp5UU4} prepared for this setting) that 
\[|Q_{p}(X,Y)|=|\tilde{Y}(X)|\leq 2|Y|. \]
For the lower bound of $|Q_{p}(X,Y)|$, 
we can apply the Rauch comparison theorem (see \cite[Corollary 5.6.1]{MR2829653} for instance) and get 
\[|Q_{p}(X,Y)|^2=|(DE_{p})_{X}(I_{X}(Y))|^2\geq |Y^{\top}|^2+|Y^{\bot}|^2\frac{\sin^2\left(\sqrt{\bar{C}_{0}}|X|\right)}{\bar{C}_{0}|X|^2}, \]
where $Y=Y^{\top}+Y^{\bot}\in \mathbb{R}
X\oplus(\mathbb{R}X)^{\bot}$. 

Since $17\bar{C}_{0}^2x^2$ is included in $D_{0}(x)$, 
the assumption $D_{0}(r)\leq \bar{C}_{0}$ and $|X|<r$ implies $\bar{C}_{0}|X|^2\leq 1/17$. 
Remark that $\sin(\theta)/\theta\geq 1/2$ for $\theta\in [0,1/\sqrt{17}]$, we have $|Q_{p}(X,Y)|^2\geq |Y|^2/4$. 
This implies that $|(DE_{p})_{X}(Y)|=|Q_{p}(X,Y)|>0$ if $Y\neq 0$. 
Thus, $(DE_{p})_{X}$ is a linear isomorphism. 
Take $(U,V)\in T_{p}L\times T_{p}L$ and assume $(DQ_{p})_{(X,Y)}(U,V)=0$. 
Sending the left-hand side by the derivative of the projection map $\pi:T^{\bot}L\to L$, 
we have $\pi_{\ast}((DQ_{p})_{(X,Y)}(U,V))=(DE_{p})_{X}(U)$. 
Since $E_{p}$ is non-singular on $B_{p}(r)$, we have $U=0$. 
Then, since $Q_{p}(X,Y)$ is linear for $Y$, we have $(DQ_{p})_{(X,Y)}(0,V)=Q_{p}(X,V)$. 
From $|V|\leq 2|Q_{p}(X,V)|$, we have $V=0$ and this implies that $(DQ_{p})_{(X,Y)}$ is also a linear isomorphism. 
\end{proof}

Then, by Lemma \ref{d3baVXrf81}, we see that 
\[
Q_{p}(B_{p}(r)\times B_{p}(r/2))\subset \mathcal{M}_{r}. 
\]
Fix $(X,Y)\in B_{p}(r)\times B_{p}(r/2)$ and put $v:=Q_{p}(X,Y)=J\tilde{Y}(X)\in \mathcal{M}_{r}$. 
Since $(DQ_{p})_{(X,Y)}:T_{p}L \times T_{p}L\to T_{v}\mathcal{M}_{r}$ is a linear isomorphism, 
there exists a unique vector, we denote it by $(\mathcal{X}_{1}(X,Y), \mathcal{X}_{2}(X,Y))$, in $T_{p}L\times T_{p}L$
so that 
\begin{equation}\label{H8TuK0Entk}
\mathcal{X}(v)=(DQ_{p})_{(X,Y)}(\mathcal{X}_{1}(X,Y), \mathcal{X}_{2}(X,Y)). 
\end{equation}

\begin{lemma}\label{Ej0JNawMSQ}
We abbreviate ``$(X,Y)$'' from $\mathcal{X}_{1}(X,Y)$ and $\mathcal{X}_{2}(X,Y)$. Then, we have 
\[\mathcal{X}^{\mathcal{H}}(v)=\tilde{\mathcal{X}}_{1}(X)\quad\mbox{and}\quad \mathcal{X}^{\mathcal{V}}(v)=J\bar{\nabla}_{\mathcal{X}_{1}}\tilde{Y}(X)+J\tilde{\mathcal{X}}_{2}(X), \]
where $\bar{\nabla}_{\mathcal{X}_{1}}\tilde{Y}(X)$ is an abbreviation for the $\bar{\nabla}_{t}$-derivative of $\tilde{Y}(X+t\mathcal{X}_{1})$ along a curve $t\mapsto X+t\mathcal{X}_{1}$ at $t=0$. 
\end{lemma}
\begin{proof}
By \eqref{H8TuK0Entk}, we have 
\[
\begin{aligned}
\mathcal{X}(v)=&\frac{d}{dt}\bigg|_{t=0}Q_{p}((X,Y)+t(\mathcal{X}_{1},\mathcal{X}_{2}))\\
=&\frac{d}{dt}\bigg|_{t=0}\left(J\tilde{Y}(X+t\mathcal{X}_{1})+tJ\tilde{\mathcal{X}}_{2}(X+t\mathcal{X}_{1})\right).
\end{aligned}
\]
Hence, for the horizontal part of $\mathcal{X}$, we have 
\[
\mathcal{X}^{\mathcal{H}}(v)=\pi_{\ast}(\mathcal{X}(v))
=\frac{d}{dt}\bigg|_{t=0}E_{p}(X+t\mathcal{X}_{1})=\tilde{\mathcal{X}}_{1}(X). 
\]
For the vertical part of $\mathcal{X}$, we have 
\[
\begin{aligned}
\mathcal{X}^{\mathcal{V}}(v)=&K(\mathcal{X}(v))
=\nabla^{\bot}_{t}\left(J\tilde{Y}(X+t\mathcal{X}_{1})+tJ\tilde{\mathcal{X}}_{2}(X+t\mathcal{X}_{1})\right)\Big|_{t=0}\\
=&J\bar{\nabla}_{\mathcal{X}_{1}}\tilde{Y}(X)+J\tilde{\mathcal{X}}_{2}(X), 
\end{aligned}
\]
and the proof is complete. 
\end{proof}

\begin{proposition}\label{yY3iO1pf0f}
For each $(X,Y)\in B_{p}(r)\times B_{p}(r/2)$, we have 
\[
|\mathcal{X}_{1}(X,Y)|\leq 4\Lambda_{1}|Y|\quad\mbox{and}\quad
|\mathcal{X}_{2}(X,Y)|\leq 14\Lambda_{1}|Y|. 
\]
\end{proposition}
\begin{proof}
In the proof, we abbreviate ``$(X,Y)$'' from $\mathcal{X}_{1}(X,Y)$ and $\mathcal{X}_{2}(X,Y)$. 
Put $v:=Q_{p}(X,Y)$. 
Then, by Lemma \ref{d3baVXrf81}, we have $|v|\leq 2|Y|$. 
Since $\tilde{\mathcal{X}}_{1}(X)=-JQ_{p}(X,\mathcal{X}_{1})$, 
we can use Lemma \ref{d3baVXrf81} again and have 
\begin{equation}\label{jcjSF1gY2i}
|\mathcal{X}_{1}|\leq 2|Q_{p}(X,\mathcal{X}_{1})|=2|\tilde{\mathcal{X}}_{1}(X)|=2|\mathcal{X}^{\mathcal{H}}(v)|\leq 2\Lambda_{1}|v|\leq 4\Lambda_{1}|Y|, 
\end{equation}
where the second equality follows from Lemma \ref{Ej0JNawMSQ} and the second inequality follows from \eqref{xBvMuiV2T0}. 
For $\mathcal{X}_{2}$, we have 
\[
|\mathcal{X}_{2}|\leq 2|Q_{p}(X,\mathcal{X}_{2})|=2|J\tilde{\mathcal{X}}_{2}(X)|\leq2|\mathcal{X}^{\mathcal{V}}(v)|+2|\bar{\nabla}_{\mathcal{X}_{1}}\tilde{Y}(X)|, 
\]
where the second inequality follows from Lemma \ref{Ej0JNawMSQ}. 
We have $|\mathcal{X}^{\mathcal{V}}(v)|\leq 2\Lambda_{1}|Y|$ as above. For the estimate of $|\bar{\nabla}_{\mathcal{X}_{1}}\tilde{Y}(X)|$, we use Corollary \ref{M7HIIp5UU4}. Then, we have 
\[
\begin{aligned}
|\bar{\nabla}_{\mathcal{X}_{1}}\tilde{Y}(X)|\leq 38\bar{C}_{0}|X||Y||\mathcal{X}_{1}| \leq 152\bar{C}_{0}\Lambda_{1}|X||Y|^2\leq 76\Lambda_{1}\bar{C}_{0}r^2|Y|, 
\end{aligned}
\]
where the second inequality follows from \eqref{jcjSF1gY2i} and the third inequality follows from $|X|\leq r$ and $|Y|\leq r/2$. 
Since $r$ satisfies \eqref{mND4S5rsdK} which implies $\bar{C}_{0}r^2\leq 1/17$, we have 
\[
76\Lambda_{1}\bar{C}_{0}r^2|Y|\leq 5 \Lambda_{1}|Y|. 
\]
By adding $2\Lambda_{1}|Y|$ to this and multiplying it by $2$, the proof is complete. 
\end{proof}

Since $\mathcal{X}_{i}$ ($i=1,2$) is a smooth map from an open set in a vector space $T_{p}L\times T_{p}L$ to a vector space $T_{p}L$, we can differentiate those in the usual sense, that is, we define $D\mathcal{X}_{i}$ at $(X,Y)\in B_{p}(r/2)\times B_{p}(r)$ by 
\[
(D\mathcal{X}_{i})_{(X,Y)}(U,V):=\frac{d}{dt}\bigg|_{t=0}\mathcal{X}_{i}(X+tU,Y+tV)
\]
for $i=1,2$ and $U,V\in T_{p}L$. 

\begin{lemma}\label{q5gfqLtP5q}
For each $(X,Y)\in B_{p}(r)\times B_{p}(r/2)$ and $U,V\in T_{p}L$ with $|U|=|V|=1$, we have 
\[
\begin{aligned}
|(D\mathcal{X}_{1})_{(X,Y)}(U,V)|\leq & 12\Lambda_{2}+12\Lambda_{1}\\
|(D\mathcal{X}_{2})_{(X,Y)}(U,V)|\leq & 40\Lambda_{2}+82\Lambda_{1}. 
\end{aligned}
\]
\end{lemma}
\begin{proof}
Put $X_{t}:=X+tU$, $Y_{t}:=Y+tV$ and $\mathcal{X}_{i}(t):=\mathcal{X}_{i}(X_{t},Y_{t})$ ($i=1,2$). 
Then, for $\eta(t):=Q_{p}(X_{t},Y_{t})$, by Lemma \ref{Ej0JNawMSQ}, 
we have
\[
\mathcal{X}^{\mathcal{H}}(\eta(t))=\widetilde{\mathcal{X}_{1}(t)}(X_{t})\quad\mbox{and}\quad \mathcal{X}^{\mathcal{V}}(\eta(t))=J\bar{\nabla}_{\mathcal{X}_{1}(t)}\widetilde{Y_{t}}(X_{t})+J\widetilde{\mathcal{X}_{2}(t})(X_{t}). 
\]
Also for $\dot{\eta}(0)$, by the same computation in the proof of Lemma \ref{Ej0JNawMSQ} with replacing $\mathcal{X}_{1}$ with $U$ and $\mathcal{X}_{2}$ with $V$, we have 
\[\dot{\eta}^{\mathcal{H}}(0)=\tilde{U}(X)\quad\mbox{and}\quad \dot{\eta}^{\mathcal{V}}(0)=J\bar{\nabla}_{U}\tilde{Y}(X)+J\tilde{V}(X), \]
where $\dot{\eta}^{\mathcal{H}}(0):=\pi_{\ast}(\dot{\eta}(0))$ and $\dot{\eta}^{\mathcal{V}}(0):=K(\dot{\eta}(0))$. 
Then, by Corollary \ref{M7HIIp5UU4}, we have 
\[
|\dot{\eta}^{\mathcal{H}}(0)|\leq 2 \quad\mbox{and}\quad 
|\dot{\eta}^{\mathcal{V}}(0)|\leq 38\bar{C}_{0}|X||Y|+2. 
\]
Hence, we have 
\begin{equation}\label{pSucdwEQL3}
|\dot{\eta}(0)|_{G}\leq 38\bar{C}_{0}|X||Y|+4\leq 19\bar{C}_{0}r^2+4\leq 6,  
\end{equation}
where the second inequality follows from $|X|\leq r$ and $|Y|\leq r/2$ and the third one follows from \eqref{mND4S5rsdK} which implies $\bar{C}_{0}r^2\leq 1/17$. 

Next, we will compute $\nabla^{G}_{\dot{\eta}(0)}\mathcal{X}$. 
Then, by \eqref{aYmLTS3JTr} with the above equations, we have 
\begin{equation}\label{fd6oc5NgTU}
\begin{aligned}
(\nabla^{G}_{\dot{\eta}(0)}\mathcal{X})^{\mathcal{H}}=&\bar{\nabla}_{\frac{\partial}{\partial t}|_{t=0}}(\widetilde{\mathcal{X}_{1}(t)}(X_{t}))\\
&+\frac{1}{2}R_{L}(\tilde{Y}(X),\bar{\nabla}_{\mathcal{X}_{1}}\tilde{Y}(X)+\tilde{\mathcal{X}}_{2}(X))\tilde{U}(X)\\
&+\frac{1}{2}R_{L}(\tilde{Y}(X),\bar{\nabla}_{U}\tilde{Y}(X)+\tilde{V}(X))\tilde{\mathcal{X}_{1}}(X), 
\end{aligned}
\end{equation}
where $(\nabla^{G}_{\dot{\eta}(0)}\mathcal{X})^{\mathcal{H}}:=\pi_{\ast}(\nabla^{G}_{\dot{\eta}(0)}\mathcal{X})$ and we sometimes used $J^2=-\mathrm{id}$. 
Also for the vertical part $(\nabla^{G}_{\dot{\eta}(0)}\mathcal{X})^{\mathcal{V}}:=K(\nabla^{G}_{\dot{\eta}(0)}\mathcal{X})$, 
we have
\begin{equation}\label{KpEIIBM2Jh}
\begin{aligned}
(\nabla^{G}_{\dot{\eta}(0)}\mathcal{X})^{\mathcal{V}}=&\nabla^{\bot}_{\frac{\partial}{\partial t}|_{t=0}}(J\bar{\nabla}_{\mathcal{X}_{1}(t)}\widetilde{Y_{t}}(X_{t})+J\widetilde{\mathcal{X}_{2}(t})(X_{t}))\\
&-\frac{1}{2}J(R_{L}(\tilde{U}(X),\tilde{\mathcal{X}_{1}}(X))\tilde{Y}(X)). 
\end{aligned}
\end{equation}
For the first term on the right-hand side of \eqref{fd6oc5NgTU}, 
we have 
\[
\bar{\nabla}_{\frac{\partial}{\partial t}|_{t=0}}(\widetilde{\mathcal{X}_{1}(t)}(X_{t}))=\bar{\nabla}_{U}\tilde{\mathcal{X}}_{1}(X)+\widetilde{\mathcal{X}'_{1}}(X), 
\]
where we put $\mathcal{X}_{1}:=\mathcal{X}_{1}(0)$ and $\mathcal{X}'_{1}:=\mathcal{X}'_{1}(0)=(D\mathcal{X}_{1})_{(X,Y)}(U,V)$ for short. 
Hence, with Lemma \ref{d3baVXrf81}, we have 
\[
\begin{aligned}
&\frac{1}{2}|\mathcal{X}'_{1}|\leq |\widetilde{\mathcal{X}'_{1}}(X)|\\
\leq & |(\nabla^{G}_{\dot{\eta}(0)}\mathcal{X})^{\mathcal{H}}|+|\bar{\nabla}_{U}\tilde{\mathcal{X}}_{1}(X)|+\frac{1}{2}|R_{L}(\tilde{Y}(X),\bar{\nabla}_{\mathcal{X}_{1}}\tilde{Y}(X)+\tilde{\mathcal{X}}_{2}(X))\tilde{U}(X)|\\
&+\frac{1}{2}|R_{L}(\tilde{Y}(X),\bar{\nabla}_{U}\tilde{Y}(X)+\tilde{V}(X))\tilde{\mathcal{X}_{1}}(X)|
\end{aligned}
\]
\[
\begin{aligned}
\leq & 6\Lambda_{2}+38\bar{C}_{0}|X|(4\Lambda_{1}|Y|)+(1/2)\bar{C}_{0}(2|Y|)(38\bar{C}_{0}|X||Y|(4\Lambda_{1}|Y|)+28\Lambda_{1}|Y|)\cdot2\\
&+(1/2)\bar{C}_{0}(2|Y|)(38\bar{C}_{0}|X||Y|+2)(8\Lambda_{1}|Y|)\\
\leq & 6\Lambda_{2}+(76/17)\Lambda_{1}+(38/289)\Lambda_{1} +(14/17)\Lambda_{1}+(38/289)\Lambda_{1}+(4/17)\Lambda_{1}\\
\leq & 6\Lambda_{2}+6\Lambda_{1}, 
\end{aligned}
\]
where we used \eqref{pSucdwEQL3}, Proposition \ref{yY3iO1pf0f} and Corollary \ref{M7HIIp5UU4} many times and the fourth inequality follows from $|Y|\leq r/2$, $|X|\leq r$ and $r$ satisfies \eqref{mND4S5rsdK} which implies $\bar{C}_{0}r^2\leq 1/17$. 
Then, we have 
\begin{equation}\label{dETqKq51wM}
|\mathcal{X}'_{1}|\leq 12\Lambda_{2}+12\Lambda_{1}
\end{equation}
and this is the first inequality of the statement. 

For the first term on the right-hand side of \eqref{KpEIIBM2Jh}, 
we have 
\[
\begin{aligned}
&\nabla^{\bot}_{\frac{\partial}{\partial t}|_{t=0}}(J\bar{\nabla}_{\mathcal{X}_{1}(t)}\widetilde{Y_{t}}(X_{t})+J\widetilde{\mathcal{X}_{2}(t})(X_{t}))\\
=&J\left(\bar{\nabla}_{U}\bar{\nabla}_{\mathcal{X}_{1}}\tilde{Y}(X)+\bar{\nabla}_{\mathcal{X}'_{1}}\tilde{Y}(X)+\bar{\nabla}_{\mathcal{X}_{1}}\tilde{V}(X)+\bar{\nabla}_{U}\tilde{X}_{2}(X)+\widetilde{\mathcal{X}'_{2}}(X)\right), 
\end{aligned}
\]
where we put $\mathcal{X}_{2}:=\mathcal{X}_{2}(0)$ and $\mathcal{X}'_{2}:=\mathcal{X}'_{2}(0)=(D\mathcal{X}_{2})_{(X,Y)}(U,V)$ for short. 
Hence, with Lemma \ref{d3baVXrf81}, we have 
\[
\begin{aligned}
&\frac{1}{2}|\mathcal{X}'_{2}|\leq |\widetilde{\mathcal{X}'_{2}}(X)|\\
\leq & |(\nabla^{G}_{\dot{\eta}(0)}\mathcal{V})^{\mathcal{H}}|+|\bar{\nabla}_{U}\bar{\nabla}_{\mathcal{X}_{1}}\tilde{Y}(X)|+|\bar{\nabla}_{\mathcal{X}'_{1}}\tilde{Y}(X)|+|\bar{\nabla}_{\mathcal{X}_{1}}\tilde{V}(X)|+|\bar{\nabla}_{U}\tilde{X}_{2}(X)|\\
& + \frac{1}{2}|R_{L}(\tilde{U}(X),\tilde{\mathcal{X}_{1}}(X))\tilde{Y}(X)|\\
\leq & 6\Lambda_{2}+109\bar{C}_{0}(4\Lambda_{1}|Y|)|Y|+38\bar{C}_{0}|X|(12\Lambda_{2}+12\Lambda_{1})|Y|+38\bar{C}_{0}|X|(4\Lambda_{1}|Y|)\\
&+38\bar{C}_{0}|X|(14\Lambda_{1}|Y|)+(1/2)\bar{C}_{0}\cdot 2(8\Lambda_{1}|Y|)(2|Y|)\\
\leq & 6\Lambda_{2}+(109/17)\Lambda_{1}+(228/17)\Lambda_{2}+(228/17)\Lambda_{1}+(76/17)\Lambda_{1}\\
&+(266/17)\Lambda_{1}+(4/17)\Lambda_{1}\\
\leq & 20\Lambda_{2}+41\Lambda_{1}, 
\end{aligned}
\]
where we used \eqref{pSucdwEQL3}, Proposition \ref{yY3iO1pf0f}, Corollary \ref{M7HIIp5UU4} and \eqref{dETqKq51wM} many times and the fourth inequality follows from $|Y|\leq r/2$, $|X|\leq r$ and $r$ satisfies \eqref{mND4S5rsdK} which implies $\bar{C}_{0}r^2\leq 1/17$. 
This is the second inequality of the statement. 
\end{proof}
\section{Existence of the generating flow}\label{9vbcsanvgva3}
Assume that $r>0$ satisfies \eqref{SxWXt7MWlg} and \eqref{mND4S5rsdK}. 
Let $\mathcal{X}_{t}$ $(t\in[0,1])$ be the one-parameter family of vector fields on $\mathcal{M}_{r}$ 
defined by \eqref{VU46Mlxo3R}. 
Then, by Proposition \ref{PUCutd5w8A} and Proposition \ref{rCbfDQNh0H}, we have 
\[
|\mathcal{X}_{t}(v)|_{G}\leq 10|v|\quad\mbox{and}\quad |\nabla^{G}\mathcal{X}(v)|_{G}\leq 294
\]
for all $v\in \mathcal{M}_{r}$ and $t\in[0,1]$. 
This fits the assumption \eqref{xBvMuiV2T0} with $\Lambda_{1}=10$ and $\Lambda_{2}=294$. 
Put 
\begin{equation}\label{oIfdfu1h7Q}
\alpha:=\frac{629\sqrt{2}}{629\sqrt{2}+7(e^{12580\sqrt{2}}-1)}=10^{-7724.35\cdots}>10^{-7725}.
\end{equation}

Fix $p\in L$ and consider the smooth map $Q_{p}:B_{p}(r/2)\times B_{p}(r/2) \to \mathcal{M}_{r}$. 
Let $(\mathcal{X}_{i})_{t}:B_{p}(r/2)\times B_{p}(r/2)\to T_{p}L$ ($i=1,2$) be the one-parameter families of smooth maps defined by \eqref{H8TuK0Entk} for each $t \in [0,1]$. 

\begin{lemma}\label{bxdbNGGfex}
For every $(X_{0},Y_{0})\in B_{p}(\alpha r/2)\times B_{p}(\alpha r/2)$, 
there exists a unique curve $c:[0,1]\to B_{p}(r/2)\times B_{p}(r/2)$ so that 
\[
\dot{c}(t)=((\mathcal{X}_{1})_{t}(c(t)),(\mathcal{X}_{2})_{t}(c(t))) \quad\mbox{and}\quad c(0)=(X_{0},Y_{0}). 
\]
\end{lemma}
\begin{proof}
By Proposition \ref{yY3iO1pf0f} and Lemma \ref{q5gfqLtP5q}, we have 
\[|(\mathcal{X}_{i})_{t}(X,Y)|\leq 14 \cdot 10|Y|\quad\mbox{and}\quad |D(\mathcal{X}_{i})_{t}(X,Y)|\leq 40\cdot 294+82\cdot 10\]
for all $(X,Y)\in B_{p}(\alpha r/2)\times B_{p}(\alpha r/2)$ and $t\in[0,1]$. 
Then, the statement follows from Proposition \ref{YYQ4BqmpQa} with $C=14\cdot 10$ ($=20\cdot 7$) and 
$L=40\cdot 294+82\cdot 10$ ($=12580=20\cdot 629$). 
\end{proof}

Let $\Phi:\mathcal{E}\to \mathcal{M}_{r}$ be the time-dependent flow of 
$\{\,\mathcal{X}_{t}\,\}_{t\in[0,1]}$. 
We use the same notation as Theorem 9.48 of \cite{MR2954043} and see also that Theorem, for instance, for the existence of the flow. 
We explain some properties of $\Phi:\mathcal{E}\to \mathcal{M}_{r}$ here. 
Actually, $\mathcal{E}$ is an open subset of $[0,1]\times [0,1]\times \mathcal{M}_{r}$ and 
$\Phi:\mathcal{E}\to \mathcal{M}_{r}$ is a smooth map with the following properties: 
\begin{itemize}
\item[(a)] For each $t_{0}\in [0,1]$ and $v\in\mathcal{M}_{r}$, the set $\mathcal{E}^{(t_{0},v)}:=\{\,t\in [0,1]\mid (t,t_{0},v)\in\mathcal{E}\,\}$ is an interval containing $t_{0}$, and the smooth curve $\Phi^{(t_{0},v)}:\mathcal{E}^{(t_{0},v)}\to\mathcal{M}_{r}$ defined by $\Phi^{(t_{0},v)}(t):=\Phi(t,t_{0},v)$ becomes the unique maximal integral curve 
of $\{\,\mathcal{X}_{t}\,\}_{t\in[0,1]}$ with initial condition $\Phi^{(t_{0},v)}(t_{0})=v$. 
\item[(b)] If $t_{1}\in \mathcal{E}^{(t_{0},v)}$ and $w=\Phi^{(t_{0},v)}(t_{1})$, then $\mathcal{E}^{(t_{1},w)}=\mathcal{E}^{(t_{0},v)}$ and $\Phi^{(t_{1},w)}=\Phi^{(t_{0},v)}$. 
\item[(c)] For each $(t_{1},t_{0})\in [0,1]\times [0,1]$, the set $(\mathcal{M}_{r})_{t_{1},t_{0}}:=\{\, p\in \mathcal{M}_{r}\mid (t_{1},t_{0},p)\in \mathcal{E}\,\}$ is open in $\mathcal{M}_{r}$, 
and the map $\Phi_{(t_{1},t_{0})}:(\mathcal{M}_{r})_{t_{1},t_{0}}\to \mathcal{M}_{r}$ defined by $\Phi_{(t_{1},t_{0})}(p):=\Phi(t_{1},t_{0},p)$ is a diffeomorphism from $(\mathcal{M}_{r})_{t_{1},t_{0}}$ onto $(\mathcal{M}_{r})_{t_{0},t_{1}}$ with inverse $\Phi_{(t_{0},t_{1})}$. 
\item[(d)] If $p\in (\mathcal{M}_{r})_{t_{1},t_{0}}$ and $\Phi_{(t_{1},t_{0})}(p)\in (\mathcal{M}_{r})_{t_{2},t_{1}}$, then $p\in (\mathcal{M}_{r})_{t_{2},t_{0}}$ and 
\[\Phi_{(t_{2},t_{1})}\circ \Phi_{(t_{1},t_{0})}=\Phi_{(t_{2},t_{0})}. \]
\end{itemize}

\begin{theorem}\label{SYsRPnN5hT}
For each $v\in \mathcal{M}_{\alpha r/2}$, 
the integral curve $\tilde{c}_{v}$ of $\{\,\mathcal{X}_{t}\,\}_{t\in[0,1]}$ 
with initial condition $\tilde{c}_{v}(0)=v$ exists for all $t\in[0,1]$, 
and a map $\Phi_{t}:\mathcal{M}_{\alpha r/2}\to \mathcal{M}_{r}$ defined by $\Phi_{t}(v):=\tilde{c}_{v}(t)$ gives 
a diffeomorphism from $\mathcal{M}_{\alpha r/2}$ to its image. 
Moreover, the restriction of $\Phi_{t}$ to the zero section in $\mathcal{M}_{\alpha r/2}$ is the identity map. 
\end{theorem}
\begin{proof}
Fix $v\in \mathcal{M}_{\alpha r/2}$. 
Put $p:=\pi(v)\in L$. Then, $(0,-Jv)\in B_{p}(\alpha r/2)\times B_{p}(\alpha r/2)$ and $Q_{p}(0,-Jv)=J(-Jv)=v$. 
Then, by Lemma \ref{bxdbNGGfex}, there exists a unique curve $c:[0,1]\to B_{p}(r/2)\times B_{p}(r/2)$ so that 
\[
\dot{c}(t)=((\mathcal{X}_{1})_{t}(c(t)),(\mathcal{X}_{2})_{t}(c(t))) \quad\mbox{and}\quad c(0)=(0,-Jv). 
\]
Then, by the definition of $(\mathcal{X}_{i})_{t}$ ($i=1,2$), the curve $\tilde{c}_{v}:=Q_{p}\circ c:[0,1]\to \mathcal{M}_{r}$ becomes an integral curve of $\{\,\mathcal{X}_{t}\,\}_{t\in[0,1]}$ with initial condition $\tilde{c}_{v}(0)=v$. 
This implies that $[0,1]=\mathcal{E}^{(0,v)}$ for all $v\in \mathcal{M}_{\alpha r/2}$ and 
$\Phi^{(0,v)}(t)=\tilde{c}_{v}(t)$ for all $t\in [0,1]$. 
Then, $(t,0,v)\in\mathcal{E}$ for all $v\in \mathcal{M}_{\alpha r/2}$ and $t\in[0,1]$. 
Hence, this implies $\mathcal{M}_{\alpha r/2}\subset (\mathcal{M}_{r})_{(t,0)}$ for all $t\in[0,1]$. 
Since $\Phi_{(t,0)}:(\mathcal{M}_{r})_{(t,0)}\to (\mathcal{M}_{r})_{(0,t)}$ is a diffeomorphism 
by the property (c), listed above, of the flow $\Phi:\mathcal{E}\to\mathcal{M}_{r}$, 
its restriction to $\mathcal{M}_{\alpha r/2}$ is also a diffeomorphism to its image. 
For $v\in \mathcal{M}_{\alpha r/2}$, we have $\Phi_{(t,0)}(v)=\Phi(t,0,v)=\Phi^{(0,v)}(t)=\tilde{c}_{v}(t)$. 
This implies that $\Phi_{t}:\mathcal{M}_{\alpha r/2}\to \mathcal{M}_{r}$ defined by $\Phi_{t}(v):=\Phi_{(t,0)}(v)=\tilde{c}_{v}(t)$ is a diffeomorphism from $\mathcal{M}_{\alpha r/2}$ to its image. 
Moreover, we see that $\tilde{c}_{0}(t)=0$ for all $t\in[0,1]$ since $\mathcal{X}_{t}(0)=0$ for all $t\in [0,1]$. 
This implies that $\Phi_{t}(0)=0$ for all $t\in [0,1]$. 
\end{proof}
\subsection{Easy expression for \texorpdfstring{$r$}{r}}
The assumption for $r$ in Theorem \ref{SYsRPnN5hT} is to satisfy \eqref{SxWXt7MWlg} and \eqref{mND4S5rsdK}. 
This is complicated. So, in this section, we give a more understandable expression for $r$. 
Assume 
\[
\sup_{M}|R_{M}|\leq C_{0}, \quad \sup_{M}|\nabla R_{M}|\leq C_{1}, \quad \sup_{M}|\nabla^2 R_{M}|\leq C_{2}, 
\]
and 
\[\sup_{L}|\mathrm{II}|\leq A_{0},\quad \sup_{L}|\nabla\mathrm{II}|\leq A_{1},\quad \sup_{L}|\nabla^2\mathrm{II}|\leq A_{2}. \]
We take $C_{1}=C_{2}=0$ if $C_{0}=0$, and $A_{1}=A_{2}=0$ if $A_{0}=0$. 
Then, by Proposition \ref{QuSRGcyB1m}, we can put $\bar{C}_{i}$ ($i=0,1,2$) as 
\begin{equation}\label{w1V4nnbcEv}
\begin{aligned}
\bar{C}_{0}:= &C_{0}+2A_{0}^2, \\
\bar{C}_{1}:= &C_{1}+4C_{0}A_{0}+4A_{1}A_{0}, \\
\bar{C}_{2}:= & C_{2}+9C_{1}A_{0}+4C_{0}A_{1}+12C_{0}A_{0}^2+4A_{2}A_{0}+4A_{1}^2. 
\end{aligned}
\end{equation}
For $\varepsilon\in(0,1)$, assume $r>0$ satisfies 
\begin{equation}\label{rI4YhrDHae}
r\leq \min\left\{\,\frac{1}{\sqrt{C_{0}}},\frac{C_{0}}{C_{1}},\frac{\sqrt{C_{0}}}{\sqrt{C_{2}}},\frac{1}{A_{0}},\frac{A_{0}}{A_{1}},\frac{\sqrt{A_{0}}}{\sqrt{A_{2}}} \,\right\}\times \varepsilon. 
\end{equation}
In the case where $C_{0}=0$ (this implies $C_{1}=C_{2}=0$), we put $1/\sqrt{C_{0}}=C_{0}/C_{1}=\sqrt{C_{0}}/\sqrt{C_{2}}=\infty$. 
In the case where $A_{0}=0$ (this implies $A_{1}=A_{2}=0$), we put $1/A_{0}=A_{0}/A_{1}=\sqrt{A_{0}}/\sqrt{A_{2}}=\infty$. 
We remark that, under the assumption \eqref{rI4YhrDHae}, 
we have $C_{0}r^2\leq \varepsilon^2$, $C_{1}r^3\leq (C_{1}r/C_{0})(C_{0}r^2)\leq \varepsilon^3$, $C_{2}r^4\leq (C_{2}r^{2}/C_{0})(C_{0}r^2)\leq \varepsilon^4$, $A_{0}r\leq\varepsilon$, $A_{1}r^2\leq (A_{1}r/A_{0})(A_{0}r)\leq \varepsilon^2$ and $A_{2}r^3\leq (A_{2}r^{2}/A_{0})(A_{0}r)\leq \varepsilon^3$. 
Then, we have 
\[
\begin{aligned}
r K_{1}(r)=&r\left(2A_{0}+2(C_{0}+2A_{0}^2)r+A_{1}r+\frac{1}{2}A_{0}(C_{0}+2A_{0}^2)r^2\right.\\
&\left.+\left(2C_{1}r^2+4C_{0}r\right)(1+A_{0}^2r^2)e^{1+r^2C_{0}}\right)e^{1+C_{0}r^2/2}\\
\leq & \bigg(2\varepsilon+2(\varepsilon^2+2\varepsilon^2)+\varepsilon^2+\frac{1}{2}(\varepsilon^3+2\varepsilon^3)\\
&+(2\varepsilon^3+4\varepsilon^2)(1+\varepsilon^2)e^{1+\varepsilon^2}\bigg)e^{1+\varepsilon^2/2}\\
\leq & (10.5+12e^2)e^2\varepsilon
\end{aligned}
\]
by $\varepsilon^{k}\leq\varepsilon$ ($k=2,3,4$) and $e^{1+\varepsilon^2},e^{1+\varepsilon^2/2}\leq e^2$. 
Hence, if 
\begin{equation}\label{7bqtQaNIAN}
\varepsilon\leq \frac{1}{(10.5+12e^2)e}=\frac{1}{269.56\cdots}, 
\end{equation}
$r$ satisfies $rK_{1}(r)\leq e$. 
For the assumption coming from $D_{0}(r)\leq \bar{C}_{0}$, which is equivalent to $D_{0}(r)/\bar{C}_{0}\leq 1$, we use the following technique. 
When we want to bound a fraction $X/\bar{C_{0}}=X/(C_{0}+2A_{0}^2)$, 
we apply the lower bound $C_{0}+2A_{0}^2\geq \max\{\,C_{0},2A_{0}^2,2\sqrt{2}\sqrt{C}_{0}A_{0}\,\}$ for the denominator to bound the fraction by the product of elements in the set in \eqref{rI4YhrDHae} where the minimum is taken in. 
For example, we have 
\[
\frac{C_{2}}{C_{0}+2A_{0}^2}\leq \frac{C_{2}}{C_{0}},\quad \frac{A_{1}A_{0}}{C_{0}+2A_{0}^2}\leq \frac{1}{2}\frac{A_{1}}{A_{0}},\quad \frac{C_{1}A_{0}}{C_{0}+2A_{0}^2}\leq \frac{C_{1}}{C_{0}}A_{0}. 
\]
Then, with the aid of \eqref{w1V4nnbcEv}, we have
\[
\begin{aligned}
\bar{C}_{0}r^2= & C_{0}r^2+2A_{0}^2r^2\leq 3\varepsilon^2\leq 3\varepsilon, \\
\frac{\bar{C}_{1}}{\bar{C}_{0}}r\leq & \frac{C_{1}}{C_{0}}r+4A_{0}r+2\frac{A_{1}}{A_{0}}r\leq 7\varepsilon, \\
\frac{\bar{C}_{2}}{\bar{C}_{0}}r^2\leq & \frac{C_{2}}{C_{0}}r^2+9\frac{C_{1}}{C_{0}}A_{0}r^2+\sqrt{2}\sqrt{C_{0}}\frac{A_{1}}{A_{0}}r^2+6C_{0}r^2+2\frac{A_{2}}{A_{0}}r^2+2\frac{A_{1}^2}{A_{0}^2}r^2\\
\leq & (20+\sqrt{2})\varepsilon^2\leq (20+\sqrt{2})\varepsilon. 
\end{aligned}
\]
Hence, we have 
\[
\begin{aligned}
\frac{D_{0}(r)}{\bar{C}_{0}}= & \frac{\bar{C}_{2}}{\bar{C}_{0}}r^2+6\frac{\bar{C}_{1}^2}{\bar{C}_{0}^2}\bar{C}_{0}r^4+20\bar{C}_{0}\frac{\bar{C}_{1}}{\bar{C}_{0}}r^3+17\bar{C}_{0}r^2+3\frac{\bar{C}_{1}}{\bar{C}_{0}}r\\
\leq & (20+\sqrt{2})\varepsilon+6(7\varepsilon)^2(3\varepsilon)+20(3\varepsilon)(7\varepsilon)+17(3\varepsilon)+3(7\varepsilon)\\
\leq & (1394+\sqrt{2})\varepsilon. 
\end{aligned}
\]
Hence, if 
\begin{equation}\label{4t8R317glQ}
\varepsilon\leq \frac{1}{1394+\sqrt{2}}=\frac{1}{1395.41\cdots}, 
\end{equation}
$r$ satisfies $D_{0}(r)\leq \bar{C}_{0}$. 
Since \eqref{4t8R317glQ} is stronger than \eqref{7bqtQaNIAN}, 
a sufficient condition for $r$ to satisfy \eqref{SxWXt7MWlg} and \eqref{mND4S5rsdK} is 
\begin{equation}\label{XxN27SUsfX}
r\leq\frac{1}{1396} \min\left\{\,\frac{1}{\sqrt{C_{0}}},\frac{C_{0}}{C_{1}},\frac{\sqrt{C_{0}}}{\sqrt{C_{2}}},\frac{1}{A_{0}},\frac{A_{0}}{A_{1}},\frac{\sqrt{A_{0}}}{\sqrt{A_{2}}} \,\right\}. 
\end{equation}

In summary, applying this to Theorem \ref{SYsRPnN5hT} and noting that $10^{-4}\leq (1/1396)\times (1/2)$, we obtain the following. 

\begin{theorem}\label{S81cHSIQmU}
If $r_{0}>0$ satisfies 
\[
r_{0}\leq 10^{-7729} \min\left\{\,\frac{1}{\sqrt{C_{0}}},\frac{C_{0}}{C_{1}},\frac{\sqrt{C_{0}}}{\sqrt{C_{2}}},\frac{1}{A_{0}},\frac{A_{0}}{A_{1}},\frac{\sqrt{A_{0}}}{\sqrt{A_{2}}} \,\right\}, 
\]
there exists a local diffeomorphism $\Theta:\mathcal{M}_{r_{0}}\to M$ so that $\Theta^{\ast}\omega=\tilde{\omega}$ and 
$\Theta(p)=p$ for all $p\in L$. 
\end{theorem}
\begin{proof}
Let $r:=2r_{0}/\alpha$. 
Then, we can check that $r$ satisfies \eqref{XxN27SUsfX} by \eqref{oIfdfu1h7Q}. 
Let $\Phi_{t}:\mathcal{M}_{r_{0}}=\mathcal{M}_{\alpha r/2}\to \mathcal{M}_{r}$ ($t\in[0,1]$) be a diffeomorphism to its image given in Theorem \ref{SYsRPnN5hT}. 
Now, $r$ satisfies \eqref{SxWXt7MWlg} which implies $K_{0}(r)\leq 1/2$ as explained in Remark \ref{FPtBaOx1oS}. 
Thus, we can use Proposition \ref{uahAg1eiGU} and say that $F:\mathcal{M}_{r}\to M$ defined by $F(v)=\exp_{\pi(v)}v$ is 
a local diffeomorphism. 
Recall that $\omega_{t}=(1-t)\tilde{\omega}+tF^{\ast}\omega$. 
Then, we have 
\[\Phi_{1}^{\ast}\omega_{1}-\Phi_{0}^{\ast}\omega_{0}=\int_{0}^{1}\frac{d}{dt}(\Phi_{t}^{\ast}\omega_{t})dt=\int_{0}^{1}\Phi_{t}^{\ast}\left(\mathcal{L}_{\mathcal{X}_{t}}\omega_{t}+\frac{d}{dt}\omega_{t}\right)dt=0, \]
where the last equality follows from 
$\mathcal{L}_{\mathcal{X}_{t}}\omega_{t}=d(\omega_{t}(\mathcal{X}_{t},\,\cdot\,))=-d\mu$ by \eqref{VU46Mlxo3R} and 
$d\omega_{t}/dt=F^{\ast}\omega-\tilde{\omega}=d\mu$ by \eqref{VD7LQHHpox}. 
Since $\Phi_{0}(v)=v$ for all $v\in\mathcal{M}_{r_{0}}$, $\omega_{0}=\tilde{\omega}$ and $\omega_{1}=F^{\ast}\omega$, by putting $\Theta:=F\circ \Phi_{1}:\mathcal{M}_{r_{0}}\to M$ which is also a local 
diffeomorphism, we have 
\[\Theta^{\ast}\omega-\tilde{\omega}=0. \]
Moreover, from $\Phi_{t}(0_{p})=0_{p}$ and $F(0_{p})=p$ for the origin $0_{p}\in T^{\bot}_{p}L$, we have $\Theta(p)=p$ under the identification of $L$ in $M$ and the zero section in $\mathcal{M}_{r_{0}}$. 
Then, the proof is complete. 
\end{proof}

The following assumption for $r$ might be more understandable. 
Put 
\[
\begin{aligned}
B:=\max\bigg\{\,&\sup_{M}|R_{M}|^{\frac{1}{2}},\sup_{M}|\nabla R_{M}|^{\frac{1}{3}},\sup_{M}|\nabla^2 R_{M}|^{\frac{1}{4}},\\
&\sup_{L}|\mathrm{II}|,\sup_{L}|\nabla\mathrm{II}|^{\frac{1}{2}},\sup_{L}|\nabla^2\mathrm{II}|^{\frac{1}{3}}\,\bigg\}. 
\end{aligned}
\]
Assume $B<\infty$. 

\begin{theorem}\label{97ygy3rowsd}
If $r_{0}>0$ satisfies 
\[
r_{0}\leq 10^{-7729} \frac{1}{B}, 
\]
there exists a local diffeomorphism $\Theta:\mathcal{M}_{r_{0}}\to M$ so that $\Theta^{\ast}\omega=\tilde{\omega}$ and 
$\Theta(p)=p$ for all $p\in L$. 
\end{theorem}
\begin{proof}
We can put $C_{0}:=B^2$, $C_{1}:=B^3$, $C_{2}:=B^4$, $A_{0}:=B$, $A_{1}:=B^2$ and $A_{2}:=B^3$. 
Then,  we have 
\[\min\left\{\,\frac{1}{\sqrt{C_{0}}},\frac{C_{0}}{C_{1}},\frac{\sqrt{C_{0}}}{\sqrt{C_{2}}},\frac{1}{A_{0}},\frac{A_{0}}{A_{1}},\frac{\sqrt{A_{0}}}{\sqrt{A_{2}}} \,\right\}=\frac{1}{B}. \]
Then, by Theorem \ref{S81cHSIQmU}, the proof is complete. 
\end{proof}
\section{Injectivity}\label{b63lbjscahq4}
In the previous section, we have constructed a \emph{local} diffeomorphism $\Theta$. 
Hence, it could lack the injectivity. 
In this section, we provide a condition which ensures the injectivity of $\Theta$. 

To obtain the injectivity of $\Theta$, we, of course, need to assume that $i:L\to M$ is an embedding rather than immersion in this section. 
We identify $i(L)$ with $L$ and consider $L$ as a subset of $M$. 
Furthermore, we assume \eqref{m5goLqTgtc} and  
\[\rho_{0}:=\mathrm{inj}(M,g)>0. \]
\begin{definition}
For an embedded submanifold $L$ in $M$ we define the \emph{embedding constant} $\mathop{\mathrm{emb}}(L)$ of $L$ by 
\[\mathop{\mathrm{emb}}(L):=\sup\left\{\frac{d_{L}(p,q)}{d_{M}(p,q)}\,\bigg|\,p,q\in L\mbox{ with }p\neq q\right\}, \]
where $d_{M}$ is the distance function on $M$ and $d_{L}$ 
is the one on $L$ with respect to the induced Riemannian metric on $L$. 
\end{definition}

\begin{proposition}\label{Swv6xwbyXh}
If $L$ is compact, then $1\leq \mathop{\mathrm{emb}}(L)<\infty$. 
\end{proposition}
\begin{proof}
Since $d_{L}(p,q)\geq d_{M}(p,q)$ always holds, $1\leq \mathop{\mathrm{emb}}(L)$ is clear. 
Hence, it is enough to prove $\mathop{\mathrm{emb}}(L)<\infty$. 
To obtain a contradiction, assume $\mathop{\mathrm{emb}}(L)=\infty$. 
Then, there exist sequences $p_{i},q_{i}\in L$ with $p_{i}\neq q_{i}$ such that 
\begin{equation}\label{lsCLVcJFmE}
\frac{d_{L}(p_{i},q_{i})}{d_{M}(p_{i},q_{i})}\to \infty\quad\mbox{as}\quad i\to\infty. 
\end{equation}\
Since $d_{L}(p_{i},q_{i})\leq \mathrm{diam}(L)<\infty$, this implies that 
\[d_{M}(p_{i},q_{i})\to 0 \quad\mbox{as}\quad i\to\infty.\]
By the compactness of $L$, there exist $p,q\in L$ and subsequences $p_{k},q_{k}$ such that 
$p_{k}\to p$ and $q_{k}\to q$ as $k\to\infty$ with respect to $d_{L}$-distance on $L$. 
Since we have 
\[
\begin{aligned}
d_{M}(p,q)\leq & d_{M}(p,p_{k})+d_{M}(p_{k},q_{k})+d_{M}(q_{k},q)\\
\leq & d_{L}(p,p_{k})+d_{M}(p_{k},q_{k})+d_{L}(q_{k},q)\to 0
\end{aligned}
\]
as $k\to \infty$, we have $p=q$. Put $o:=p$ ($=q$) in $L$. 

We make a notational convention here. 
The point $o\in L$ is the fixed point, whereas $0\in T_{o}L$ denotes the zero vector. 
Moreover, $B(o,r)\subset L$ denotes the open ball of radius $r$ in $L$ with respect to the induced metric, 
while $B(0,r)\subset T_{o}L$ denotes the open ball of radius $r$ in $T_{o}L$. 

Since $L$ can be written as the graph of a vector-valued function over $T_{o}L$ around $o$, there exist an open set $U$ in $M$ containing $o$, 
an open ball $B(0,r)$ in $T_{o}L$ with $r<\rho_{0}/2$ and 
a smooth map $f:B(o,r)\to T_{o}^{\bot}L$ satisfying 
$|f(v)|<r$ for all $v\in B(o,r)$ such that 
\[B(0,r)\ni v\mapsto \exp_{o}(v+f(v)) \in L\cap U\]
is a diffeomorphism. 
We remark that $\exp_{o}$ of $v+f(v)$ is defined since 
$|v+f(v)|< 2r<\rho_{0}\leq \mathrm{inj}(M,g)$. 
Since $d_{L}(p_{k},o)\to 0$ and $d_{L}(q_{k},o)\to 0$, there exists $N\in\mathbb{N}$ such that $p_{k},q_{k}\in U$ and $2d_{L}(o,p_{k})+d_{L}(o,q_{k})<r/2$ for all $k\geq N$. 
Fix, $k$ with $k\geq N$. Let $c:[0,T]\to M$ be a shortest geodesic in $M$ 
with $|c'|\equiv 1$ from $p_{k}$ to $q_{k}$, 
where $T=d_{M}(p_{k},q_{k})$. 
Then, we have 
\[
\begin{aligned}
d_{M}(o,c(t))\leq & d_{M}(o,p_{k})+d_{M}(p_{k},c(t))
\leq d_{M}(o,p_{k})+d_{M}(p_{k},q_{k})\\
\leq & d_{L}(o,p_{k})+d_{L}(p_{k},q_{k})
\leq 2d_{L}(o,p_{k})+d_{L}(o,q_{k})
<r/2. 
\end{aligned}
\]
Thus, we can define 
\[\tilde{c}(t):=\exp_{o}^{-1}(c(t))\in T_{o}M. \]
Split $\tilde{c}(t)$ as 
\[\tilde{c}(t)=\tilde{c}^{\top}(t)+\tilde{c}^{\bot}(t)\in T_{o}L\oplus T^{\bot}_{o}L. \]
Put $v(t):=\tilde{c}^{\top}(t)$. 
Then, we have $|v(t)|=|\tilde{c}^{\top}(t)|\leq |\tilde{c}(t)|<r/2$. 
Hence, we can define a curve $\alpha:[0,T]\to L\cap U$ by
\[\alpha(t):=\exp_{o}(v(t)+f(v(t)))\quad\mbox{in}\quad L\cap U. \]
Since $p_{k}=c(0)=\exp_{o}(v(0)+\tilde{c}^{\bot}(0))$ and it is in $L\cap U$, 
we see that $\tilde{c}^{\bot}(0)=f(v(0))$ and $\alpha(0)=p_{k}$. 
Similarly, we have $\alpha(T)=q_{k}$. 
Thus, $\alpha:[0,T]\to L\cap U$ is connecting from $p_{k}$ to  $q_{k}$. 
Then, we have 
\[d_{L}(p_{k},q_{k})\leq \int_{0}^{T}\left|\alpha'(t)\right|dt\leq \left(\max_{t\in [0,T]}\left|\alpha'(t)\right|\right)d_{M}(p_{k},q_{k}). \]
Since $\alpha'(t)=(D\exp_{o})_{w(t)}(w'(t))$ with 
$w(t):=v(t)+f(v(t))$, $|v(t)|<r/2$ and $|w(t)|<3r/2<3\rho_{0}/4$, we have 
\begin{equation}\label{RyMcomctCV}
|\alpha'(t)|\leq \left(\max_{x\in \bar{B}(0,3r/2)}|(D\exp_{o})_{x}|\right)\left(1+\max_{v\in \bar{B}(0,r/2)}|(Df)_{v}| \right)|v'(t)|, 
\end{equation}
where $\bar{B}(0,R)$ is the closed ball with radius $R$. 
Since $|v'(t)|\leq |\tilde{c}'(t)|$, $\tilde{c}'(t)=(D\exp_{o}^{-1})_{c(t)}(c'(t))$ and $|c'(t)|=1$, we have 
\begin{equation}\label{fBIL1rDMtA}
|v'(t)|\leq \max_{p\in \bar{B}(o,r/2)}|(D\exp_{o}^{-1})_{p}|. 
\end{equation}
Combining these inequalities, we get 
\[d_{L}(p_{k},q_{k})\leq C(r) d_{M}(p_{k},q_{k}), \]
where $C(r)$ is the product of the coefficient of $|v'(t)|$ in \eqref{RyMcomctCV} and the right-hand side 
of \eqref{fBIL1rDMtA}. 
Especially, we remark that $C(r)$ does not depend on $k\geq N$. 
This, however, contradicts \eqref{lsCLVcJFmE}, 
and the proof is complete. 
\end{proof}

Recall that we have defined a smooth map $F:T^{\bot}L\to M$ by 
\[F(v)=\exp_{\pi(v)}v. \]

\begin{proposition}\label{GUuCw8Vx1b}
Assume $L$ is compact. Then, 
$F:U_{r}(T^{\bot}L)\to M$ is injective when 
\begin{equation}\label{P1SeE5PYgS}
r\leq \frac{1}{3\mathop{\mathrm{emb}}(L)}\min\left\{\,\rho_{0},\frac{\pi}{2\sqrt{C_{0}}},\frac{1}{\sqrt{C_{0}}}\arctan\left(\frac{\sqrt{C_{0}}}{A_{0}}\right)\right\}. 
\end{equation}
\end{proposition}
\begin{proof}
We prove this by contradiction. 
Assume that $r$ satisfies \eqref{P1SeE5PYgS} and 
there exist $v_{0}, v_{1}\in U_{r}(T^{\bot}L)$ with $v_{0}\neq v_{1}$ and $F(v_{0})=F(v_{1})$. 
Put $p_{i}:=\pi(v_{i})$ ($i=0,1$). 
If $p_{0}=p_{1}$, then $r$ should satisfy $r\geq \rho_{0}$ 
since $\exp_{p}$ is injective on $\{\,v\in T_{p}M\mid |v|<\rho_{0}\,\}$ for each $p \in M$. 
This contradicts $r\leq \rho_{0}/(3\mathop{\mathrm{emb}}(L))$, 
since $\mathop{\mathrm{emb}}(L)\geq 1$ by Proposition~\ref{Swv6xwbyXh}, 

In the following, we treat the case where $p_{0}\neq p_{1}$. 
Put 
\[O:=\exp_{p_{0}}(v_{0})=\exp_{p_{1}}(v_{1}) \in M, \]
and $T:=d_{L}(p_{0},p_{1})>0$, the distance between $p_{0}$ and $p_{1}$ measured by curves in $L$. 
Since $d_{M}(p_{0},p_{1})\leq d_{M}(p_{0},O)+d_{M}(O,p_{1})\leq |v_{0}|+|v_{1}|<2r$, 
we have 
\begin{equation}\label{sGRkL2tTW8}
T=d_{L}(p_{0},p_{1})\leq \mathop{\mathrm{emb}}(L)d_{M}(p_{0},p_{1})<2r\mathop{\mathrm{emb}}(L). 
\end{equation}
Let $c:[0,T]\to L$ be a shortest geodesic curve from $p_{0}$ to $p_{1}$ in $L$ with $|c'|\equiv 1$. 
Then, we can say that $c(t)\in B(O,3r\mathop{\mathrm{emb}}(L))$ for all $t\in [0,T]$. 
Indeed, $d_{M}(O,c(t))\leq d_{M}(O,p_{0})+d_{M}(p_{0},c(t))< r+T<3r\mathop{\mathrm{emb}}(L)$, 
where we used $\mathop{\mathrm{emb}}(L)\geq 1$ in Proposition~\ref{Swv6xwbyXh} and \eqref{sGRkL2tTW8}. 
Put $r':=3r\mathop{\mathrm{emb}}(L)$. 
Then, we have $r'\leq\rho_{0}$, 
and we can define a curve $\tilde{c}(t)\in T_{O}M$ by 
\[\tilde{c}(t):=\exp_{O}^{-1}(c(t))\]
for all $t\in [0,T]$, since $c(t)\in B(O,r')$ and $r'\leq\rho_{0}$. 
We remark that $|\tilde{c}(t)|<r'$ for all $t\in [0,T]$, 
since $c(t)\in B(O,r')$. 
By using this curve, we construct a one-parameter family of geodesics in $M$ denoted by $\gamma_{t}:[0,1]\to M$ by 
\[\gamma_{t}(s):=\exp_{O}(s\tilde{c}(t)), \]
and define a map $\alpha:[0,T]\times [0,1]\to M$ by 
\[\alpha(t,s):=\gamma_{t}(s)=\exp_{O}(s\tilde{c}(t)). \]
We list some properties of $\alpha(t,s)$. 
\begin{itemize}
\item[(i)] For each $t\in [0,T]$, the curve $[0,1]\ni s\mapsto \gamma_{t}(s)$ is the shortest geodesic in $M$ from $O$ to $c(t)$ 
with speed $|\gamma_{t}'(s)|= |\tilde{c}(t)|<r'$. 
\item[(ii)] Since $\gamma_{0}(s)$ gives the reverse parametrization of $[0,1]\ni \tilde{s}\to \exp_{p_{0}}(\tilde{s}v_{0})$, 
we have $\gamma_{0}'(1)=-v_{0}$. Similarly, we also have $\gamma_{T}'(1)=-v_{1}$. Especially, we have 
$\gamma_{0}'(1)\in T_{p_{0}}^{\bot}L$ and $\gamma_{T}'(1)\in T_{p_{1}}^{\bot}L$. 
\item[(iii)] If $s=0$, the map $[0,T]\ni t\mapsto \alpha(t,0)\in M$ is a constant map, so $\partial_{t}\alpha(t,0)=0$ for all $t\in [0,T]$. 
\item[(iv)] If $s=1$, the map $[0,T]\ni t\mapsto \alpha(t,1)\in M$ coincides with $c(t)$. 
\end{itemize}

First, we pick a special $\ell$ from $[0,T]$ in the following way. 
Consider a function $\varphi:[0,T]\to \mathbb{R}$ defined by 
$\varphi(t):=g(\gamma_{t}'(1),c'(t))$. 
Since $\varphi(0)=\varphi(T)=0$ by (ii), we can pick $\ell\in [0,T]$ such that $\varphi'(\ell)=0$. 
By a direct computation of $\varphi'(\ell)=0$, we have 
\begin{equation}\label{2tRsVwl3Qo}
\begin{aligned}
0=\varphi'(\ell)
=&g\left(\nabla_{t}\frac{\partial \alpha}{\partial s}(\ell,1),c'(\ell)\right)+g\left(\gamma_{\ell}'(1),\nabla_{t}c'(\ell)\right)\\
=&g\left(\nabla_{s}\frac{\partial \alpha}{\partial t}(\ell,1),c'(\ell)\right)+g\left(\gamma_{\ell}'(1),\mathop{\mathrm{II}}(c'(\ell),c'(\ell))\right), 
\end{aligned}
\end{equation}
where the last equality follows from $\nabla_{t}=\bar{\nabla}_{t}+\mathop{\mathrm{II}}(c',\,\cdot\,)$ for 
the induced Levi--Civita connection $\bar{\nabla}$ on $L$ and that $c$ is a geodesic in $L$. 
Here, we introduce the variational vector field $J$ along the curve $\gamma_{\ell}:[0,1]\to M$ by 
\[J(s):=\frac{\partial \alpha}{\partial t}(\ell,s)\in T_{\gamma_{\ell}(s)}M. \]
Since $J$ is a variational vector field of one parameter family of geodesics in $M$, $J$ satisfies the Jacobi field equation: 
\[\nabla_{s}\nabla_{s}J+R_{M}(J,\gamma_{\ell}')\gamma_{\ell}'=0\]
and also $J(0)=0$ by (iii) and $J(1)=c'(\ell)$ by (iv). 
Moreover, by \eqref{2tRsVwl3Qo}, we see that $\nabla_{s}J(1)$ satisfies 
\begin{equation}\label{cjXv78GUXd}
g(\nabla_{s}J(1),J(1))=-g(\gamma_{\ell}'(1),\mathop{\mathrm{II}}(c'(\ell),c'(\ell)))\leq |\gamma_{\ell}'|A_{0}
\end{equation}
from $|c'|=1$. 

On the other hand, by the standard comparison theorem for a Jacobi field 
(see Theorem 5.5.1 in \cite{MR2829653} for instance), 
we have 
\begin{equation}\label{VgJ5VjUdFP}
g(\nabla_{s}J(1),J(1))\geq \frac{\sqrt{C_{0}}|\gamma_{\ell}'|}{\tan(\sqrt{C_{0}}|\gamma_{\ell}'|)}|J(1)|^2=\frac{\sqrt{C_{0}}|\gamma_{\ell}'|}{\tan(\sqrt{C_{0}}|\gamma_{\ell}'|)}, 
\end{equation}
where the last equality follows from $J(1)=c'(\ell)$ and $|c'|\equiv 1$. 
We remark that, in this case, a comparison function for $|J(s)|$ is 
\[f(s):=\frac{|\nabla_{s}J(0)|}{\sqrt{C_{0}}|\gamma_{\ell}'|}\sin(\sqrt{C_{0}}|\gamma_{\ell}'|s)\]
and is nonnegative on $[0,\pi/(\sqrt{C_{0}}|\gamma_{\ell}'|)]$ which 
includes $s=1$ since $\sqrt{C_{0}}|\gamma_{\ell}'|<\sqrt{C_{0}}r'
=3r\mathop{\mathrm{emb}}(L)\sqrt{C_{0}}\leq \pi/2$ by (i), 
where the last inequality follows from the assumption \eqref{P1SeE5PYgS} for $r$. 
Then, combining \eqref{cjXv78GUXd} and \eqref{VgJ5VjUdFP}, 
we get 
\[A_{0}\geq \frac{\sqrt{C_{0}}}{\tan(\sqrt{C_{0}}|\gamma_{\ell}'|)}. \]
Since $\sqrt{C_{0}}|\gamma_{\ell}'|<\pi/2$, this is equivalent to 
\[|\gamma_{\ell}'|\geq \frac{1}{\sqrt{C_{0}}}\arctan\left(\frac{\sqrt{C_{0}}}{A_{0}}\right). \]
However, this contradicts the assumption \eqref{P1SeE5PYgS} for $r$ with using $|\gamma_{\ell}'|<r'$ in (i) and 
$r'=3r\mathop{\mathrm{emb}}(L)$. 
\end{proof}

Put 
\[
\begin{aligned}
B_{*}:=3\mathop{\mathrm{emb}}(L)\times\max\bigg\{\,&\frac{1}{\mathrm{inj}(M,g)},\sup_{M}|R_{M}|^{\frac{1}{2}},\sup_{M}|\nabla R_{M}|^{\frac{1}{3}},\sup_{M}|\nabla^2 R_{M}|^{\frac{1}{4}},\\
&\sup_{L}|\mathrm{II}|,\sup_{L}|\nabla\mathrm{II}|^{\frac{1}{2}},\sup_{L}|\nabla^2\mathrm{II}|^{\frac{1}{3}}\,\bigg\}. 
\end{aligned}
\]
Of course, $B_{\ast}\geq B$. 
Assume $B_{*}<\infty$. 

Combining Theorem \ref{S81cHSIQmU} and Proposition \ref{GUuCw8Vx1b}, we have the following. 

\begin{theorem}\label{3h2HuVLYkq}
If $r_{0}>0$ satisfies 
\[
r_{0}\leq 10^{-7729} \frac{1}{B_{*}}, 
\]
the local diffeomorphism $\Theta:\mathcal{M}_{r_{0}}\to M$ constructed in the proof of Theorem \ref{S81cHSIQmU} becomes an injective local diffeomorphism, that is, a diffeomorphism to its image. 
\end{theorem}
\begin{proof}
Since $\Theta$ is defined by $F\circ \Phi_{1}$ and $\Phi_{1}:\mathcal{M}_{r_{0}}\to \mathcal{M}_{r}$ is a diffeomorphism to its image with $r:=2r_{0}/\alpha$, it suffices to check that $F:\mathcal{M}_{r}\to M$ is injective. 
Let $r_{\ast}$ be the constant on the right-hand side of \eqref{P1SeE5PYgS} with letting $C_{0}:=\sup_{M}|R_{M}|$ and $A_{0}:=\sup_{L}|\mathrm{II}|$. 
Then, by Proposition \ref{GUuCw8Vx1b}, we can say that $F:\mathcal{M}_{r}\to M$ is injective if $2r_{0}/\alpha\leq r_{\ast}$. 
By \eqref{oIfdfu1h7Q}, it is enough to check that $1/B_{\ast}\leq 10^{3}r_{\ast}$. 
Since 
\[\frac{1}{x}\arctan\left(\frac{x}{y}\right)=\frac{1}{y}\left(\frac{y}{x}\arctan\left(\frac{x}{y}\right)\right)\geq \frac{\pi}{4}\frac{1}{y}\]
if $x/y\leq 1$ and 
\[\frac{1}{x}\arctan\left(\frac{x}{y}\right)\geq \frac{\pi}{4}\frac{1}{x}\]
if $x/y\geq 1$, 
we have $(1/x)\arctan(x/y)\geq 10^{-1}/\max\{\,x,y\,\}$. 
Applying this with $x=\sqrt{C_{0}}$ and $y=A_{0}$ implies that $1/B_{\ast}\leq 10^{3}r_{\ast}$. 
Thus, $F:\mathcal{M}_{r}\to M$ is injective with $r=2r_{0}/\alpha$. 
On the other hand, $10^{-7729}/B_{\ast}\leq 10^{-7729}/B$ is clear. 
Thus, $\Theta:\mathcal{M}_{r_{0}}\to M$ is a local diffeomorphism by Theorem \ref{S81cHSIQmU}. 
Combining these two facts, the proof is complete. 
\end{proof}
\appendix
\section{An explicit bound of derivatives of the exponential map}
Let $(M,g)$ be an $n$-dimensional complete Riemannian manifold with Riemannian curvature tensor $R_{M}$. 
We denote by $\nabla^{k}R_{M}$ the $k$-th derivative of $R_{M}$. 
In this appendix, we derive some estimates for derivatives of the exponential map $\exp_{p}:T_{p}M\to M$, which is used in Section \ref{bv287adsggql}. 
\begin{proposition}\label{l065WM0Tcg}
Assume that there exists a constant $C_{0}>0$ such that $|R_{M}|\leq C_{0}$. 
Let $\sigma:[0,L]\ni s \to \sigma(s)\in M$ be a geodesic with $|\dot{\sigma}|=v$ and 
$J$ be a vector field along $\sigma$. 
Assume that there exists a constant $D_{0}\geq 0$ such that 
\[|\nabla\nabla J(s)+\mathop{R_{M}}(J(s),\dot{\sigma}(s))\dot{\sigma}(s)|\leq D_{0}\]
for all $s\in [0,L]$. 
Then, for any fixed $\varepsilon\in [0,2]$, we have 
\[
|J(s)|^2+|\nabla J(s)|^2\leq (|J(0)|^2+|\nabla J(0)|^2+D_{0}^{2-\varepsilon})e^{(1+C_{0}v^2+D_{0}^{\varepsilon})L}
\]
for all $s\in[0,L]$, where we define $D_{0}^{2-\varepsilon}=0$ if $D_{0}=0$ and $\varepsilon=2$. 
\end{proposition}
\begin{proof}
Put $f(s):=|J(s)|^2+|\nabla J(s)|^2$. 
Then, we have 
\[
\begin{aligned}
f'(s)=&2g(J,\nabla J)+2g(\nabla J,\nabla\nabla J)\\
=&2g(J,\nabla J)+2g(\nabla J,\nabla\nabla J+\mathop{R_{M}}(J,\dot{\sigma})\dot{\sigma})-2g(\nabla J,\mathop{R_{M}}(J,\dot{\sigma})\dot{\sigma})\\
\leq&2|J||\nabla J|+2D_{0}|\nabla J|+2C_{0}v^2|J||\nabla J|\\
\leq& |J|^2+|\nabla J|^2+D_{0}^{2-\varepsilon}+D_{0}^{\varepsilon}|\nabla J|^2+C_{0}v^2(|J|^2+|\nabla J|^2)\\
\leq& (1+C_{0}v^2+D_{0}^{\varepsilon})f(s)+D_{0}^{2-\varepsilon}\\
\leq& (1+C_{0}v^2+D_{0}^{\varepsilon})(f(s)+D_{0}^{2-\varepsilon}). 
\end{aligned}
\]
By this ordinary differential inequality for $f+D_{0}^{2-\varepsilon}$, 
one can easily see that 
\[
f(s)+D_{0}^{2-\varepsilon}\leq (f(0)+D_{0}^{2-\varepsilon})e^{(1+C_{0}v^2+D_{0}^{\varepsilon})L}, 
\]
and the proof is complete. 
\end{proof}

Fix $p\in M$. Put $E_{p}:=\exp_{p}:T_{p}M\to M$ for short. 
For $Y\in T_{p}M$, we define a smooth section of the pull-back bundle of $TM$ by $E_{p}$ over $T_{p}M$, 
denoted by $\tilde{Y}\in\Gamma(T_{p}M,E_{p}^{\ast}(TM))$, by 
\[\tilde{Y}(X):=\frac{d}{dt}\bigg|_{t=0}E_{p}(X+tY)=(D\exp_{p})_{X}(Y)\,\,\in T_{E_{p}(X)}M, \]
for $X\in T_{p}M$. 
\begin{proposition}\label{VDev15mdth}
Assume $|R_{M}|\leq C_{0}$, $|\nabla R_{M}|\leq C_{1}$, $|\nabla^2R_{M}|\leq C_{2}$. 
Fix $X,Y_{1}\in T_{p}M$. 
Take $Y_{2},Y_{3}\in T_{p}M$. Then, we have 
\[
\begin{aligned}
|\tilde{Y}_{1}(X)|^2\leq &|Y_{1}|^2\exp(1+C_{0}|X|^2)\\
|\nabla_{Y_{2}}\tilde{Y}_{1}(X)|^2\leq &4|X|^2|Y_{1}|^2|Y_{2}|^2(C_{1}|X|+2C_{0})^2\exp(4+3C_{0}|X|^2)\\
|\nabla_{Y_{3}}\nabla_{Y_{2}}\tilde{Y}_{1}(X)|^2\leq &4|Y_{1}|^2|Y_{2}|^2|Y_{3}|^2(D_{0}(|X|)+C_{0})^2\exp(7+5C_{0}|X|^2), 
\end{aligned}
\]
where 
\begin{equation}\label{T6SOQS4wgY}
D_{0}(x):=C_{2}x^2+6C_{1}^2x^4+20C_{0}C_{1}x^3+17C_{0}^2x^2+3C_{1}x. 
\end{equation}
Here, $\nabla_{Y_{2}}\tilde{Y}_{1}(X)$ is the abbreviation for $\nabla$-derivative of 
$\tilde{Y}_{1}$ along a curve $t_{2}\mapsto \alpha(t_{2},0)$ at $t_{2}=0$, where $\alpha(t_{2},t_{3}):=X+t_{2}Y_{2}+t_{3}Y_{3}\in T_{p}M$. 
\end{proposition}
\begin{proof}
The proof is essentially the same as a part of the proof of Theorem 4.2 of \cite{MR1121230}. 
It is enough to consider the case where $|Y_{i}|=1$ for $i=1,2,3$ since $\tilde{Y}_{1}(X)$, $\nabla_{Y_{2}}\tilde{Y}_{1}(X)$ and $\nabla_{Y_{3}}\nabla_{Y_{2}}\tilde{Y}_{1}(X)$ are linear 
for each $Y_{i}$ ($i=1,2,3$). 
Put $c(s,t_{1},t_{2},t_{3}):=s(X+t_{1}Y_{1}+t_{2}Y_{2}+t_{3}Y_{3})\in T_{p}M$ for $s\in [0,1]$. 
For a fixed $(t_{1},t_{2},t_{3})$, $s\mapsto c(s,t_{1},t_{2},t_{3})$ 
is a curve connecting $0$ to $X+t_{1}Y_{1}+t_{2}Y_{2}+t_{3}Y_{3}$ and its image 
under $E_{p}$, that is, $s\mapsto E_{p}(c(s,t_{1},t_{2},t_{3}))$, 
is a geodesic in $M$ from $p$ to $E_{p}(X+t_{1}Y_{1}+t_{2}Y_{2}+t_{3}Y_{3})$. 
We sometimes identify $c$ with $E_{p}\circ c$.  
Define  
\[
J_{1}(s,t_{2},t_{3}):=\frac{\partial}{\partial t_{1}}\bigg|_{t_{1}=0}E_{p}(c(s,t_{1},t_{2},t_{3})). 
\]
We similarly define $J_{2}(s,t_{1},t_{3})$ and $J_{3}(s,t_{1},t_{2})$. 
Since $J_{1}$ (and also $J_{2},J_{3}$) is a variational vector field of geodesics in $M$, $s\mapsto J_{1}(s,t_{2},t_{3})$ satisfies 
the Jacobi field equation: 
\begin{equation}\label{ev2Pgb4yUu}
\nabla_{s}\nabla_{s}J_{1}+R_{M}(J_{1},\dot{c})\dot{c}=0
\end{equation}
for each fixed $(t_{2},t_{3})$, where $\dot{c}$ is 
the abbreviation for $(\partial c/\partial s)(s,0,t_{2},t_{3})$, and the initial conditions:
\begin{equation}\label{uUqSM6pRXT}
J_{1}(0,t_{2},t_{3})=0\quad\mbox{and}\quad \nabla_{s}J_{1}(0,t_{2},t_{3})=Y_{1}. 
\end{equation}
Then, by Proposition \ref{l065WM0Tcg} with $D_{0}=0$ and $\varepsilon=2$, we have 
\[
|J_{1}|^2+|\nabla_{s}J_{1}|^2\leq \exp(1+C_{0}|X+t_{2}Y_{2}+t_{3}Y_{3}|^2). 
\]
This also holds for $J_{2}$ and $J_{3}$. 
Put 
\begin{equation}\label{WYUjYTX1QR}
\mathcal{A}_{0}:=\{\,J_{1},J_{2},J_{3}\,\}\quad\mbox{and}\quad \mathcal{A}_{1}:=\{\,\nabla_{s}J_{1},\nabla_{s}J_{2},\nabla_{s}J_{3}\,\}. 
\end{equation}
Then, after evaluating the above inequality at $t_{2}=t_{3}=0$, we can say that 
\begin{equation}\label{7Dq5N58yR1}
|\spadesuit|^2\leq \exp(1+C_{0}|X|^2)\quad\mbox{for all }\spadesuit\in\mathcal{A}_{0}\cup\mathcal{A}_{1}.  
\end{equation}
Since $J_{1}(1,0,0)=\tilde{Y}_{1}(X)$, especially we obtain the first estimate in the statement. 

For the next computation, define  
\[Z_{21}(s,t_{3}):=\nabla_{t_{2}}J_{1}(s,0,t_{3}). \]
We similarly define $Z_{ij}(s,t_{k})$ for every pairwise distinct $i,j,k\in\{\,1,2,3\,\}$. 
Then, by taking the $\nabla_{t_{2}}$ derivative of \eqref{ev2Pgb4yUu} (at $t_{2}=0$), 
we have 
\begin{equation}\label{07nMuKu2ai}
\nabla_{t_{2}}(\nabla_{s}\nabla_{s}J_{1})+\nabla_{t_{2}}(\mathop{R_{M}}(J_{1},\dot{c})\dot{c})=0. 
\end{equation}
For the first term on the left-hand side of \eqref{07nMuKu2ai}, 
one can easily see that
\begin{equation}\label{gEeHJu42ut}
\begin{aligned}
&\nabla_{t_{2}}(\nabla_{s}\nabla_{s}J_{1})\\
=&\nabla_{s}\nabla_{t_{2}}\nabla_{s}J_{1}+
\mathop{R_{M}}(J_{2},\dot{c})\nabla_{s}J_{1}\\
=&\nabla_{s}(\nabla_{s}\nabla_{t_{2}}J_{1}+\mathop{R_{M}}(J_{2},\dot{c})J_{1})+
\mathop{R_{M}}(J_{2},\dot{c})\nabla_{s}J_{1}\\
=&\nabla_{s}\nabla_{s}Z_{21}+(\mathop{\nabla_{\dot{c}}R_{M}})(J_{2},\dot{c})J_{1}+\mathop{R_{M}}(\nabla_{s}J_{2},\dot{c})J_{1}
+2\mathop{R_{M}}(J_{2},\dot{c})\nabla_{s}J_{1}, 
\end{aligned}
\end{equation}
where we used $\nabla_{s}\dot{c}=0$ in the third equality. 
For the second term on the left-hand side of \eqref{07nMuKu2ai}, 
one can easily see that
\begin{equation}\label{DS3cHBbMDm}
\begin{aligned}
\nabla_{t_{2}}(\mathop{R_{M}}(J_{1},\dot{c})\dot{c})=&
(\mathop{\nabla_{J_{2}}R_{M}})(J_{1},\dot{c})\dot{c}+\mathop{R_{M}}(Z_{21},\dot{c})\dot{c}\\
&+\mathop{R_{M}}(J_{1},\nabla_{s}J_{2})\dot{c}+\mathop{R_{M}}(J_{1},\dot{c})\nabla_{s}J_{2}, 
\end{aligned}
\end{equation}
where we used $\nabla_{t_{2}}\dot{c}=\nabla_{s}J_{2}$ at $t_{2}=0$. 
Thus, we have 
\begin{equation}\label{8He0qTh5qc}
\begin{aligned}
&\nabla_{s}\nabla_{s}Z_{21}+\mathop{R_{M}}(Z_{21},\dot{c})\dot{c}\\
=&
-(\mathop{\nabla_{\dot{c}}R_{M}})(J_{2},\dot{c})J_{1}
-\mathop{R_{M}}(\nabla_{s}J_{2},\dot{c})J_{1}
-2\mathop{R_{M}}(J_{2},\dot{c})\nabla_{s}J_{1}\\
&-(\mathop{\nabla_{J_{2}}R_{M}})(J_{1},\dot{c})\dot{c}
-\mathop{R_{M}}(J_{1},\nabla_{s}J_{2})\dot{c}
-\mathop{R_{M}}(J_{1},\dot{c})\nabla_{s}J_{2}\\
=&-(\mathop{\nabla_{\dot{c}}R_{M}})(J_{2},\dot{c})J_{1}
-2\mathop{R_{M}}(J_{2},\dot{c})\nabla_{s}J_{1}\\
&-(\mathop{\nabla_{J_{2}}R_{M}})(J_{1},\dot{c})\dot{c}
-2\mathop{R_{M}}(J_{1},\dot{c})\nabla_{s}J_{2}, 
\end{aligned}
\end{equation}
where the second equality follows from the first Bianchi identity of $R_{M}$. 
We denote $(\nabla_{X}S)(\cdots)$ by $(\nabla S)(X,\cdots)$ 
for a tensor field $S$ and denote $T(\cdots)X$ by $T(\cdots,X)$ for a $\mathrm{End}(TM)$-valued tensor field $T$. 
Then, the right-hand side of \eqref{8He0qTh5qc} is written as 
\[
\begin{aligned}
-(\nabla R_{M})(\dot{c},J_{2},\dot{c},J_{1})-2R_{M}(J_{2},\dot{c},\nabla_{s}J_{1})\\
-(\nabla R_{M})(J_{2},J_{1},\dot{c},\dot{c})-2R_{M}(J_{1},\dot{c},\nabla_{s}J_{2}). 
\end{aligned}
\]
Here, we introduce the following abbreviations:
\[
\begin{aligned}
\ctext{2}(\nabla R_{M})(\dot{c},J_{2},\dot{c},J_{1}):=&-(\nabla R_{M})(\dot{c},J_{2},\dot{c},J_{1})-(\nabla R_{M})(J_{2},J_{1},\dot{c},\dot{c})\\
\ctext{4}R_{M}(J_{2},\dot{c},\nabla_{s}J_{1}):=&-2R_{M}(J_{2},\dot{c},\nabla_{s}J_{1})-2R_{M}(J_{1},\dot{c},\nabla_{s}J_{2}). 
\end{aligned}
\]
We consider a circled number, like \ctext{k}, as 
a kind of coefficient for each term. 
We explain the usage of this abbreviation more precisely below. 
Fix $m\in \mathbb{N}$. 
For an $m$-tuple of vector fields $(X_{1},\dots,X_{m})$ so that each $X_{i}$ is in $\{\,\dot{c}\,\}\cup\mathcal{A}_{0}\cup\mathcal{A}_{1}$ (see \eqref{WYUjYTX1QR} for the definition of $\mathcal{A}_{0}$ and $\mathcal{A}_{1}$), we define 
$\#_{\dot{c}}(X_{1},\dots,X_{m})$ as the number of $X_{i}$ equal to $\dot{c}$. 
Similarly, for $j=0,1$, we define $\#_{\mathcal{A}_{j}}(X_{1},\dots,X_{m})$ as the number of $X_{i}$ contained in $\mathcal{A}_{j}$. 
For example, 
\[\#_{\dot{c}}(\dot{c},J_{2},\dot{c},J_{1})=2,\ \#_{\mathcal{A}_{0}}(\dot{c},J_{2},\dot{c},J_{1})=2,\ \#_{\mathcal{A}_{1}}(\dot{c},J_{2},\dot{c},J_{1})=0. \]
By using these counting operators, we say that $T(X_{1},\dots,X_{m})$ and $S(Y_{1},\dots,Y_{m})$ 
are equivalent if $S=\pm T$ and $\#_{\bullet}(X_{1},\dots,X_{m})=\#_{\bullet}(Y_{1},\dots,Y_{m})$ for all $\bullet=\dot{c}, \mathcal{A}_{0}, \mathcal{A}_{1}$ and denote the situation as $T(X_{1},\dots,X_{m})\sim S(Y_{1},\dots,Y_{m})$. 
For example, 
\[
(\nabla R_{M})(\dot{c},J_{2},\dot{c},J_{1})\sim -(\nabla R_{M})(J_{2},J_{1},\dot{c},\dot{c}). 
\]
We denote the equivalence class of $T(X_{1},\dots,X_{m})$ by $[T(X_{1},\dots,X_{m})]$. 
Then, for $k\in\mathbb{N}$, the symbol $\ctext{k}T(X_{1},\dots,X_{m})$ 
means that it can be written as a sum of $k$ terms in $[T(X_{1},\dots,X_{m})]$. 
Put $p:=\#_{\dot{c}}(X_{1},\dots,X_{m})$ and $q:=\#_{\mathcal{A}_{0}}(X_{1},\dots,X_{m})+\#_{\mathcal{A}_{1}}(X_{1},\dots,X_{m})$. 
Then, by \eqref{7Dq5N58yR1}, we can say that 
\[
|S(Y_{1},\dots,Y_{m})|\leq |T||X|^{p}\left(\exp(1+C_{0}|X|^2)\right)^{q/2}
\]
if $T(X_{1},\dots,X_{m})\sim S(Y_{1},\dots,Y_{m})$. 
Since the purpose of this proof is to estimate norms, this abbreviation works sufficiently for that purpose, and we have that
\[|\ctext{k}T(X_{1},\dots,X_{m})|\leq k |T||X|^{p}\left(\exp(1+C_{0}|X|^2)\right)^{q/2}. \]

By using this abbreviation, the equation \eqref{8He0qTh5qc} can be written as 
\begin{equation}\label{vXvPSiFbsQ}
\nabla_{s}\nabla_{s}Z_{21}+\mathop{R_{M}}(Z_{21},\dot{c})\dot{c}=\ctext{2}(\nabla R_{M})(\dot{c},J_{2},\dot{c},J_{1})+\ctext{4}R_{M}(J_{2},\dot{c},\nabla_{s}J_{1}). 
\end{equation}
Thus, after evaluating the above equality at $t_{3}=0$, we have 
\begin{equation}\label{gsXfYSyKgn}
\begin{aligned}
&|\nabla_{s}\nabla_{s}Z_{21}+\mathop{R_{M}}(Z_{21},\dot{c})\dot{c}|\\
\leq & 2C_{1}|X|^2\exp(1+C_{0}|X|^2)+4C_{0}|X|\exp(1+C_{0}|X|^2)\\
= & 2|X|\exp(1+C_{0}|X|^2)(C_{1}|X|+2C_{0}), 
\end{aligned}
\end{equation}
where we use \eqref{7Dq5N58yR1} for the estimate of $J_{i}$ and $\nabla J_{i}$ ($i=1,2$) with $t_{2}=t_{3}=0$. 
For the initial condition of $Z_{21}$, we see that $Z_{21}(0,t_{3})=0$ and 
\[
\begin{aligned}
\nabla_{s}Z_{21}(0,t_{3})=&\nabla_{s}\nabla_{t_{2}}J_{1}(0,0,t_{3})\\
=&\nabla_{t_{2}}\nabla_{s}J_{1}(0,0,t_{3})+R_{M}(\dot{c}(0),J_{2}(0,0,t_{3}))J_{1}(0,0,t_{3})=0, 
\end{aligned}
\]
where the last equality follows from \eqref{uUqSM6pRXT}. 
Here is a point. 
The setting and notation in this proof are the same as in the proof of Theorem 4.2 of \cite{MR1121230}. However, in (2.20) in \cite{MR1121230}, $Z_{21}$ is to satisfy $\nabla_{s}Z_{21}(0)=Y_{2}$. At least for the author, it seems incorrect. 
This is the reason why almost the same proof as that of Theorem 4.2 is repeatedly added to this proposition. Anyway, $Z_{21}$ satisfies the initial conditions: 
\begin{equation}\label{l3Cv8stQri}
Z_{21}(0,t_{3})=0 \quad\mbox{and}\quad \nabla_{s}Z_{21}(0,t_{3})=0. 
\end{equation}
Then, by Proposition \ref{l065WM0Tcg} with \eqref{gsXfYSyKgn} and $\varepsilon=0$, we have (when $t_{3}=0$) 
\[
\begin{aligned}
& |Z_{21}|^2+|\nabla_{s}Z_{21}|^2\\
\leq &\Bigl(2|X|\exp(1+C_{0}|X|^2)(C_{1}|X|+2C_{0})\Bigr)^2\exp(2+C_{0}|X|^2)\\
= & 4|X|^2(C_{1}|X|+2C_{0})^2\exp(4+3C_{0}|X|^2). 
\end{aligned}
\]
This also holds for $Z_{ij}$ with distinct $i,j\in\{1,2,3\}$. 
Put 
\[
\mathcal{B}_{0}:=\{\,Z_{ij}\mid i\neq j\in\{\,1,2,3\,\}\,\}\quad\mbox{and}\quad \mathcal{B}_{1}:=\{\,\nabla_{s}Z_{ij}\mid i\neq j\in\{\,1,2,3\,\}\,\}. 
\]
Then, after evaluating the above inequality at $t_{3}=0$, we can say that 
\begin{equation}\label{VbtraE74B5}
|\clubsuit|^2\leq 4|X|^2(C_{1}|X|+2C_{0})^2\exp(4+3C_{0}|X|^2)\quad\mbox{for all }\clubsuit\in\mathcal{B}_{0}\cup\mathcal{B}_{1}.  
\end{equation}
Since $Z_{21}(1,0)=\nabla_{Y_{2}}\tilde{Y}_{1}(X)$, especially we obtain the second estimate in the statement. 
By $\mathcal{B}_{0}$ and $\mathcal{B}_{1}$, we also define counting operator $\#_{\mathcal{B}_{0}}$ and $\#_{\mathcal{B}_{1}}$ and define the equivalence relation $\sim$ by adding $\mathcal{B}_{0},\mathcal{B}_{1}$ after the phrase ``for all $\bullet=\dot{c}, \mathcal{A}_{0}, \mathcal{A}_{1}$'' in the definition. 

For the final computation, define  
\[Z_{321}(s):=\nabla_{t_{3}}Z_{21}(s,0)=\nabla_{t_{3}}\nabla_{t_{2}}J_{1}(s,0,0). \]
We similarly define $Z_{ijk}(s)$ for every pairwise distinct $i,j,k\in\{\,1,2,3\,\}$. 
Then, by taking the $\nabla_{t_{3}}$ derivative of \eqref{vXvPSiFbsQ} (at $t_{3}=0$) with keeping the abbreviations, 
we have 
\begin{equation}\label{IlcjQJTpHY}
\begin{aligned}
&\nabla_{t_{3}}(\nabla_{s}\nabla_{s}Z_{21})+\nabla_{t_{3}}(\mathop{R_{M}}(Z_{21},\dot{c})\dot{c})\\
=&\ctext{2}\nabla_{t_{3}}((\nabla R_{M})(\dot{c},J_{2},\dot{c},J_{1}))+\ctext{4}\nabla_{t_{3}}((R_{M}(J_{2},\dot{c},\nabla_{s}J_{1})). 
\end{aligned}
\end{equation}
By the same procedure to obtain \eqref{8He0qTh5qc} from \eqref{gEeHJu42ut} and \eqref{DS3cHBbMDm}, 
we can see that the left-hand side of \eqref{IlcjQJTpHY} is equal to 
\[
\nabla_{s}\nabla_{s}Z_{321}+\mathop{R_{M}}(Z_{321},\dot{c})\dot{c}+\ctext{2}(\nabla R_{M})(\dot{c},J_{3},\dot{c},Z_{21})+\ctext{4}R_{M}(J_{3},\dot{c},\nabla_{s}Z_{21}). 
\]
For the first term on the right-hand side of \eqref{IlcjQJTpHY}, one can see that 
\[
\begin{aligned}
\nabla_{t_{3}}((\nabla R_{M})(\dot{c},J_{2},\dot{c},J_{1}))= & (\nabla^2 R_{M})(J_{3},\dot{c},J_{2},\dot{c},J_{1})+\ctext{2}(\nabla R_{M})(\nabla_{s}J_{3},J_{2},\dot{c},J_{1})\\
& +\ctext{2}(\nabla R_{M})(\dot{c},Z_{32},\dot{c},J_{1}). 
\end{aligned}
\]
For the second term on the right-hand side of \eqref{IlcjQJTpHY}, one can see that 
\[
\begin{aligned}
&\nabla_{t_{3}}((R_{M}(J_{2},\dot{c},\nabla_{s}J_{1}))\\
= & (\nabla R_{M})(J_{3},J_{2},\dot{c},\nabla_{s}J_{1})+R_{M}(Z_{31},\dot{c},\nabla_{s}J_{1})+R_{M}(J_{2},\nabla_{s}J_{3},\nabla_{s}J_{1})\\
&+R_{M}(J_{2},\dot{c},\nabla_{s}Z_{31})+R_{M}(J_{2},\dot{c},R_{M}(J_{3},\dot{c})J_{1}), 
\end{aligned}
\]
where we used $\nabla_{t_{3}}\nabla_{s}J_{1}=\nabla_{s}Z_{31}+R_{M}(J_{3},\dot{c})J_{1}$. 
Inserting the above three computations into \eqref{IlcjQJTpHY}, we have 
\[
\begin{aligned}
&\nabla_{s}\nabla_{s}Z_{321}+\mathop{R_{M}}(Z_{321},\dot{c})\dot{c}\\
=&\ctext{2}(\nabla R_{M})(\dot{c},J_{3},\dot{c},Z_{21})+\ctext{4}R_{M}(J_{3},\dot{c},\nabla_{s}Z_{21})\\
&+\ctext{2}\Big((\nabla^2 R_{M})(J_{3},\dot{c},J_{2},\dot{c},J_{1})+\ctext{2}(\nabla R_{M})(\nabla_{s}J_{3},J_{2},\dot{c},J_{1})\\
&+\ctext{2}(\nabla R_{M})(\dot{c},Z_{32},\dot{c},J_{1})\Big)\\
&+\ctext{2}\Big((\nabla R_{M})(J_{3},J_{2},\dot{c},\nabla_{s}J_{1})+R_{M}(Z_{31},\dot{c},\nabla_{s}J_{1})+R_{M}(J_{2},\nabla_{s}J_{3},\nabla_{s}J_{1})\\
&+R_{M}(J_{2},\dot{c},\nabla_{s}Z_{31})+R_{M}(J_{2},\dot{c},R_{M}(J_{3},\dot{c})J_{1})\Big)\\
=&\ctext{2}(\nabla^2 R_{M})(J_{3},\dot{c},J_{2},\dot{c},J_{1})+\ctext{6}(\nabla R_{M})(\dot{c},J_{3},\dot{c},Z_{21})\\
&+\ctext{6}R_{M}(J_{3},\dot{c},\nabla_{s}Z_{21})+\ctext{6}(\nabla R_{M})(\nabla_{s}J_{3},J_{2},\dot{c},J_{1})\\
&+\ctext{2}R_{M}(Z_{31},\dot{c},\nabla_{s}J_{1})+\ctext{2}R_{M}(J_{2},\nabla_{s}J_{3},\nabla_{s}J_{1})\\
&+\ctext{2}R_{M}(J_{2},\dot{c},R_{M}(J_{3},\dot{c})J_{1}). 
\end{aligned}
\]
Thus, by \eqref{7Dq5N58yR1} and \eqref{VbtraE74B5}, we have 
\[
\begin{aligned}
&|\nabla_{s}\nabla_{s}Z_{321}+\mathop{R_{M}}(Z_{321},\dot{c})\dot{c}|\\
\leq & 2C_{2}|X|^2\exp\left(\frac{3}{2}+\frac{3}{2}C_{0}|X|^2\right)+12C_{1}|X|^3(C_{1}|X|+2C_{0})\exp\left(\frac{5}{2}+2C_{0}|X|^2\right)\\
& + 12C_{0}|X|^2(C_{1}|X|+2C_{0})\exp\left(\frac{5}{2}+2C_{0}|X|^2\right)+6C_{1}|X|\exp\left(\frac{3}{2}+\frac{3}{2}C_{0}|X|^2\right)\\
&+4C_{0}|X|^2(C_{1}|X|+2C_{0})\exp\left(\frac{5}{2}+2C_{0}|X|^2\right)+ 2C_{0}\exp\left(\frac{3}{2}+\frac{3}{2}C_{0}|X|^2\right)\\
&+2C_{0}^2|X|^2\exp\left(\frac{3}{2}+\frac{3}{2}C_{0}|X|^2\right)\\
\leq & 2\Big(C_{2}|X|^2+6C_{1}^2|X|^4+20C_{0}C_{1}|X|^3+17C_{0}^2|X|^2\\
&+3C_{1}|X|+C_{0}\Big)\exp\left(\frac{5}{2}+2C_{0}|X|^2\right)
\end{aligned}
\]
For the initial condition of $Z_{321}$, we see that $Z_{321}(0)=0$ and 
\[
\begin{aligned}
\nabla_{s}Z_{321}(0)=&\nabla_{s}\nabla_{t_{3}}Z_{21}(0,0)\\
=&\nabla_{t_{3}}\nabla_{s}Z_{21}(0,0)+R_{M}(\dot{c}(0),J_{3}(0,0,0))Z_{21}(0,0)=0, 
\end{aligned}
\]
where the last equality follows from \eqref{l3Cv8stQri}. 
Then, by Proposition \ref{l065WM0Tcg} with $\varepsilon=0$, we have 
\[
\begin{aligned}
& |Z_{321}|^2+|\nabla_{s}Z_{321}|^2\\
\leq &4\Big(C_{2}|X|^2+6C_{1}^2|X|^4+20C_{0}C_{1}|X|^3+17C_{0}^2|X|^2\\
&+3C_{1}|X|+C_{0}\Big)^2\exp(7+5C_{0}|X|^2). 
\end{aligned}
\]
Since $Z_{321}(1)=\nabla_{Y_{3}}\nabla_{Y_{2}}\tilde{Y}_{1}(X)$, especially we obtain the third estimate in the statement. 
\end{proof}

\begin{corollary}\label{M7HIIp5UU4}
Under the same setting as Proposition \ref{VDev15mdth}, 
if $D_{0}(|X|)\leq C_{0}$ \rm{(}see \eqref{T6SOQS4wgY} for the definition of $D_{0}$\rm{)}, we have 
\[
\begin{aligned}
|\tilde{Y}_{1}(X)|\leq &2|Y_{1}|\\
|\nabla_{Y_{2}}\tilde{Y}_{1}(X)|\leq & 38C_{0}|X||Y_{1}||Y_{2}|\\
|\nabla_{Y_{3}}\nabla_{Y_{2}}\tilde{Y}_{1}(X)|\leq &109C_{0}|Y_{1}||Y_{2}||Y_{3}|. 
\end{aligned}
\]
\end{corollary}
\begin{proof}
Since $17C_{0}^2x^2$ and $3C_{1}x$ are included in $D_{0}(x)$, the assumption $D_{0}(|X|)\leq C_{0}$ implies $C_{0}|X|^2\leq 1/17$ and $C_{1}|X|\leq C_{0}/3$. 
This implies that $\exp(1+C_{0}|X|^2)\leq e^{18/17}=2.88\cdots\leq 2^2$. 
Similarly, we have 
\[
\begin{aligned}
&4(C_{1}|X|+2C_{0})^2\exp(4+3C_{0}|X|^2)\\
\leq & 4\cdot((7/3)C_{0})^2e^{71/17}=(1418.50\cdots)C_{0}^2\leq (38C_{0})^2
\end{aligned}
\] and 
\[
\begin{aligned}
& 4(D_{0}(|X|)+C_{0})^2\exp(7+5C_{0}|X|^2)\\
\leq & 4\cdot (2C_{0})^2e^{124/17}=(11772.94\cdots)C_{0}^2\leq (109C_{0})^2. 
\end{aligned}
\]
The proof is complete. 
\end{proof}

Let $L$ be an $m$-dimensional immersed submanifold in $(M,g)$ with second fundamental form $\mathrm{II}$. 
We denote by $\nabla^{k}\mathrm{II}$ the $k$-th derivative of $\mathrm{II}$. 
Let $\bar{\nabla}$ and $R_{L}$ be the Levi--Civita connection and the Riemannian curvature tensor of $(L,g|_{L})$, respectively. 

\begin{proposition}\label{QuSRGcyB1m}
If $|R_{M}|\leq C_{0}$, $|\nabla R_{M}|\leq C_{1}$, $|\nabla^2R_{M}|\leq C_{2}$, 
$|\mathrm{II}|\leq A_{0}$, $|\nabla\mathrm{II}|\leq A_{1}$ and $|\nabla^2\mathrm{II}|\leq A_{2}$, 
we have 
\[
\begin{aligned}
|R_{L}|\leq &C_{0}+2A_{0}^2, \\
|\bar{\nabla}R_{L}|\leq & C_{1}+4C_{0}A_{0}+4A_{1}A_{0}, \\
|\bar{\nabla}^2R_{L}|\leq & C_{2}+9C_{1}A_{0}+4C_{0}A_{1}+12C_{0}A_{0}^2+4A_{2}A_{0}+4A_{1}^2. 
\end{aligned}
\]
\end{proposition}
\begin{proof}
The first inequality follows from the Gauss equation: 
\begin{equation}\label{ndTscyk0OG}
\begin{aligned}
R_{M}(X,Y,Z,W)= & R_{L}(X,Y,Z,W)+\langle \mathop{\mathrm{II}}(X,Z),\mathop{\mathrm{II}}(Y,W)\rangle\\
& -\langle\mathop{\mathrm{II}}(Y,Z),\mathop{\mathrm{II}}(X,W)\rangle. 
\end{aligned}
\end{equation}
Take $U\in T_{p}L$ and parallelly extend $X,Y,Z,W$ along a short curve $c$ in $L$ with $\dot{c}(0)=U$ 
with respect to the Levi--Civita connection $\bar{\nabla}$ of $(L,g|_{L})$. 
We denote $(\nabla_{U}T)(\cdots)$ by $(\nabla T)(U,\cdots)$ for some tensor field $T$. 
Taking $\bar{\nabla}_{U}$-derivative of both hand sides of \eqref{ndTscyk0OG}, we have 
\begin{equation}\label{sycFTxIqV3}
\begin{aligned}
&(\bar{\nabla}R_{L})(U,X,Y,Z,W)\\
= & (\nabla R_{M})(U,X,Y,Z,W)+\sum^{4}(\pm)R_{M}(\mathop{\mathrm{II}}(U,X),Y,Z,W)\\
&+\sum^4(\pm)\langle (\mathop{\nabla\mathrm{II}})(U,X,Z),\mathop{\mathrm{II}}(Y,W)\rangle. 
\end{aligned}
\end{equation}
In the above equality, we used some abbreviations. First, $(\pm)$ means that 
we do not care about the sign of each term because we only need to bound its norm from above. 
Second, $\sum^{4}T$ indicates that there are actually 4 terms (including $T$ itself) which 
are similar to $T$ and can be bounded from above the same bound as $T$. 
Then, taking the norm of both hand sides of \eqref{sycFTxIqV3} implies the second inequality in the statement. 
Finally, taking $\bar{\nabla}_{V}$-derivative (for some $V\in T_{p}L$) of both hand sides of \eqref{sycFTxIqV3} with the above abbreviations, we have 
\begin{equation}\label{uI6qbEBcSt}
\begin{aligned}
&(\bar{\nabla}^2R_{L})(V,U,X,Y,Z,W)\\
=&(\nabla^2 R_{M})(V,U,X,Y,Z,W)+\sum^{5}(\pm)(\nabla R_{M})(\mathop{\mathrm{II}}(V,U),X,Y,Z,W)\\
&+\sum^{4}(\pm)\Big((\nabla R_{M})(V,\mathop{\mathrm{II}}(U,X),Y,Z,W)+R_{M}((\mathop{\nabla\mathrm{II}})(V,U,X),Y,Z,W)\\
&+\sum^3(\pm)R_{M}(\mathop{\mathrm{II}}(U,X),\mathop{\mathrm{II}}(V,Y),Z,W)\Big)\\
&+\sum^4(\pm)\Big(\langle (\mathop{\nabla^2\mathrm{II}})(V,U,X,Z),\mathop{\mathrm{II}}(Y,W)\rangle+\langle (\mathop{\nabla\mathrm{II}})(U,X,Z),(\mathop{\nabla\mathrm{II}})(V,Y,W)\rangle\Big). 
\end{aligned}
\end{equation}
Then, taking the norm of both hand sides of \eqref{uI6qbEBcSt} implies the third inequality in the statement. 
\end{proof}
\section{Lindel\"of's lemma for ODE}
In this appendix, we give a proof of an important proposition used in the proof of Lemma \ref{bxdbNGGfex}. 
The idea of the proof comes from Lindel\"of's lemma for ODE, briefly introduced in Section \ref{21usb783pav}, which gives some upper bound of the existing time interval of ODE, and we can say that the existing time interval can be bigger than $[0,1]$ if the norm of the vector field on $\bar{B}(r)\times \bar{B}(r)$ is relatively small with respect to $y$-direction. 

\begin{proposition}\label{YYQ4BqmpQa}
Let $\bar{B}(r)\subset \mathbb{R}^{n}$ be a closed ball with radius $r$ centered at the origin. 
Let $X_{t}$ and $Y_{t}$ be one-parameter families of smooth maps from $\bar{B}(r)\times \bar{B}(r)\to \mathbb{R}^{n}$ for $t\in [0,1]$. 
Assume that there exists a constant $C>0$ such that $|X_{t}(x,y)|\leq C|y|$ and $|Y_{t}(x,y)|\leq C|y|$ for all $(x,y)\in \bar{B}(r)\times \bar{B}(r)$ and $t\in[0,1]$. 
Assume that there exists $L>0$ such that $|X_{t}(z_{1})-X_{t}(z_{2})|\leq L|z_{1}-z_{2}|$ and $|Y_{t}(z_{1})-Y_{t}(z_{2})|\leq L|z_{1}-z_{2}|$ for all $z_{i}=(x_{i},y_{i})\in \bar{B}(r)\times \bar{B}(r)$ ($i=1,2$) and $t\in[0,1]$, where 
$|z_{1}-z_{2}|:=\sqrt{|x_{1}-x_{2}|^2+|y_{1}-y_{2}|^2}$. 
Put 
\[
\alpha:=\frac{\sqrt{2}L}{\sqrt{2}L+C(e^{\sqrt{2}L}-1)}\in(0,1). 
\]
Then, for each $z=(x,y)\in \bar{B}(\alpha r)\times \bar{B}(\alpha r)$, there exists a unique smooth curve $c_{z}:[0,1]\to \bar{B}(r)\times \bar{B}(r)$ such that  
\begin{equation}\label{XnCfQQ053f}
\dot{c}(t)=(X_{t}(c(t)),Y_{t}(c(t))) \quad\mbox{and}\quad c(0)=z=(x,y). 
\end{equation}
\end{proposition}
\begin{proof}
We construct $c$ as the limit of an inductively defined curve $c_{i}$. 
Define $c_{0}:[0,1]\to \bar{B}(r)\times \bar{B}(r)$ by $c_{0}(t):=(\alpha_{0}(t),\beta_{0}(t))$ with $\alpha_{0}(t)=x$ and $\beta_{0}(t)=y$. 
Define $c_{1}:[0,1]\to \bar{B}(r)\times \bar{B}(r)$ with $c_{1}(t):=(\alpha_{1}(t),\beta_{1}(t))$ by 
\[
\alpha_{1}(t):=x+\int_{0}^{t}X_{s}(c_{0}(s))ds\quad\mbox{and}\quad\beta_{1}(t):=y+\int_{0}^{t}Y_{s}(c_{0}(s))ds. 
\]
Then, these satisfy 
\[|\alpha_{1}(t)-\alpha_{0}(t)|,|\beta_{1}(t)-\beta_{0}(t)|\leq C\max_{s\in[0,1]}|\beta_{0}(s)|t\leq C\alpha rt\]
for all $t\in [0,1]$. 
Fix $i\geq 1$. 
Assume that we have $c_{k}(t):=(\alpha_{k}(t),\beta_{k}(t)) \in \bar{B}(r)\times \bar{B}(r)$ for $t\in[0,1]$ and $k=0,1,\dots,i$ and these satisfy 
\[
|\alpha_{k+1}(t)-\alpha_{k}(t)|,|\beta_{k+1}(t)-\beta_{k}(t)|\leq \frac{C\alpha r}{\sqrt{2}L} \frac{(\sqrt{2}Lt)^{k+1}}{(k+1)!}
\]
for all $t\in [0,1]$ and $k=0,1,\dots,i-1$. 
We remark that we have already checked that this assumption holds for $i=1$. 
Then, under this assumption, $\alpha_{i}(t)$ satisfies 
\[
|\alpha_{i}(t)-\alpha_{0}(t)|\leq \sum_{k=0}^{i-1}|\alpha_{k+1}(t)-\alpha_{k}(t)|\leq \frac{C\alpha r}{\sqrt{2}L}\sum_{k=0}^{i-1} \frac{(\sqrt{2}L)^{k+1}}{(k+1)!}\leq \frac{C\alpha r}{\sqrt{2}L}(e^{\sqrt{2}L}-1). 
\]
Since $|\alpha_{0}(t)|=|x|\leq \alpha r$, we have 
\[|\alpha_{i}(t)|\leq \alpha r+\frac{C\alpha r}{\sqrt{2}L}(e^{\sqrt{2}L}-1)=\alpha r \frac{\sqrt{2}L+C(e^{\sqrt{2}L}-1)}{\sqrt{2}L}=r. \]
Similarly, since $|\beta_{0}(t)|=|y|\leq \alpha r$, we have 
\[|\beta_{i}(t)|\leq \alpha r+\frac{C\alpha r}{\sqrt{2}L}(e^{\sqrt{2}L}-1)=r. \]
Hence, we can inductively define $c_{i+1}(t):=(\alpha_{i+1}(t),\beta_{i+1}(t))$ by 
\[
\alpha_{i+1}(t):=x+\int_{0}^{t}X_{s}(c_{i}(s))ds\quad\mbox{and}\quad\beta_{i+1}(t):=y+\int_{0}^{t}Y_{s}(c_{i}(s))ds. 
\]
Moreover, since $|c_{i}(s)-c_{i-1}(s)|=\sqrt{|\alpha_{i}(s)-\alpha_{i-1}(s)|^2+|\beta_{i}(s)-\beta_{i-1}(s)|^2}$, we have 
\[
\begin{aligned}
|\alpha_{i+1}(t)-\alpha_{i}(t)|= & \int_{0}^{t}|X_{s}(c_{i}(s))-X_{s}(c_{i-1}(s))|ds\\
\leq & \sqrt{2}L\int_{0}^{t}\frac{C\alpha r}{\sqrt{2}L} \frac{(\sqrt{2}Ls)^{i}}{i!}ds=\frac{C\alpha r}{\sqrt{2}L} \frac{(\sqrt{2}Lt)^{i+1}}{(i+1)!}. 
\end{aligned}
\]
Similarly, $\beta_{i}(t)$ satisfies the same inequality. 
Hence, we can inductively define $c_{n}(t)=(\alpha_{n}(t),\beta_{n}(t))$ for all $n \in \mathbb{N}\cup \{\,0\,\}$. 
Then, for any $m$ and $n>m$, we have 
\[
|\alpha_{n}(t)-\alpha_{m}(t)|\leq \frac{C\alpha r}{\sqrt{2}L}\sum_{k=m+1}^{n}\frac{(\sqrt{2}L)^{k}}{k!}\to 0 \quad\mbox{as}\quad m\to\infty. 
\]
Hence, the sequence $\{\,\alpha_{n}(t)\,\}_{n=0}^{\infty}$ uniformly converges to a continuous map $\alpha:[0,1]\to \bar{B}(r)$. 
Similarly, $\{\,\beta_{n}(t)\,\}_{n=0}^{\infty}$ uniformly converges to a continuous map $\beta:[0,1]\to \bar{B}(r)$. Then, we get a continuous curve $c_{z}:[0,1]\to \bar{B}(r)\times\bar{B}(r)$ with $c_{z}(t):=(\alpha(t),\beta(t))$ which satisfies 
\[
\alpha(t)=x+\int_{0}^{t}X_{s}(c(s))ds\quad\mbox{and}\quad\beta(t)=y+\int_{0}^{t}Y_{s}(c(s))ds. 
\]
In particular, $c_{z}$ is smooth and satisfies \eqref{XnCfQQ053f}. 

We can also prove the uniqueness. 
If a smooth curve $\tilde{c}:[0,1]\to \bar{B}(r)\times\bar{B}(r)$ with $\tilde{c}(t)=(\tilde{\alpha}(t),\tilde{\beta}(t))$ also satisfies \eqref{XnCfQQ053f}, then we have 
\[
|\alpha(t)-\tilde{\alpha}(t)|\leq \int_{0}^{t}|X_{s}(c(s))-X_{s}(\tilde{c}(s))|ds\leq L\max_{s\in [0,1]}|c(s)-\tilde{c}(s)|t\leq 2\sqrt{2}rLt. 
\]
We also have $|\beta(t)-\tilde{\beta}(t)|\leq 2\sqrt{2}rLt$. 
Using this procedure again, we have 
\[
|\alpha(t)-\tilde{\alpha}(t)|\leq \int_{0}^{t}|X_{s}(c(s))-X_{s}(\tilde{c}(s))|ds\leq \int_{0}^{t} 2(\sqrt{2}L)^2rs ds\leq  2r\frac{(\sqrt{2}Lt)^2}{2}. 
\]
We also have $|\beta(t)-\tilde{\beta}(t)|\leq 2r(\sqrt{2}Lt)^2/2$. 
Iterating this, we have 
\[
|c_{z}(t)-\tilde{c}(t)|\leq 2\sqrt{2}r\frac{(\sqrt{2}Lt)^{m}}{m!}\leq 2\sqrt{2}r\frac{(\sqrt{2}L)^{m}}{m!}\to 0 \quad\mbox{as}\quad m\to\infty.  
\]
Hence, $\tilde{c}(t)=c_{z}(t)$ for all $t\in [0,1]$. 
\end{proof}
\section*{Conflict of Interest}
The author has no competing interests to declare that are relevant to the content of this article. 
\section*{Data Availability}
We do not analyse or generate any datasets, because our work proceeds within a theoretical and mathematical approach. 
\printbibliography

@book {MR2829653,
    AUTHOR = {Jost, J\"{u}rgen},
     TITLE = {Riemannian geometry and geometric analysis},
    SERIES = {Universitext},
   EDITION = {Sixth},
 PUBLISHER = {Springer, Heidelberg},
      YEAR = {2011},
     PAGES = {xiv+611},
      ISBN = {978-3-642-21297-0},
   MRCLASS = {53Cxx (35R01 53-02 57R58 58E05 58E20 58J05)},
  MRNUMBER = {2829653},
MRREVIEWER = {Fr\'{e}d\'{e}ric Robert},
       DOI = {10.1007/978-3-642-21298-7},
       URL = {https://doi.org/10.1007/978-3-642-21298-7},
}

@article {MR286028,
    AUTHOR = {Kowalski, Old\v{r}ich},
     TITLE = {Curvature of the induced {R}iemannian metric on the tangent
              bundle of a {R}iemannian manifold},
   JOURNAL = {J. Reine Angew. Math.},
  FJOURNAL = {Journal f\"{u}r die Reine und Angewandte Mathematik. [Crelle's
              Journal]},
    VOLUME = {250},
      YEAR = {1971},
     PAGES = {124--129},
      ISSN = {0075-4102},
   MRCLASS = {53.70},
  MRNUMBER = {286028},
MRREVIEWER = {J. R. Vanstone},
       DOI = {10.1515/crll.1971.250.124},
       URL = {https://doi.org/10.1515/crll.1971.250.124},
}

@article {MR1121230,
    AUTHOR = {Eichhorn, J\"{u}rgen},
     TITLE = {The boundedness of connection coefficients and their
              derivatives},
   JOURNAL = {Math. Nachr.},
  FJOURNAL = {Mathematische Nachrichten},
    VOLUME = {152},
      YEAR = {1991},
     PAGES = {145--158},
      ISSN = {0025-584X},
   MRCLASS = {53C21},
  MRNUMBER = {1121230},
MRREVIEWER = {Reinhard Illge},
       DOI = {10.1002/mana.19911520113},
       URL = {https://doi.org/10.1002/mana.19911520113},
}

@book {MR2954043,
    AUTHOR = {Lee, John M.},
     TITLE = {Introduction to smooth manifolds},
    SERIES = {Graduate Texts in Mathematics},
    VOLUME = {218},
   EDITION = {Second},
 PUBLISHER = {Springer, New York},
      YEAR = {2013},
     PAGES = {xvi+708},
      ISBN = {978-1-4419-9981-8},
   MRCLASS = {58-01 (53-01 57-01)},
  MRNUMBER = {2954043},
}

@article {MR286137,
    AUTHOR = {Weinstein, Alan},
     TITLE = {Symplectic manifolds and their {L}agrangian submanifolds},
   JOURNAL = {Advances in Math.},
  FJOURNAL = {Advances in Mathematics},
    VOLUME = {6},
      YEAR = {1971},
     PAGES = {329--346 (1971)},
      ISSN = {0001-8708},
   MRCLASS = {57.50},
  MRNUMBER = {286137},
MRREVIEWER = {D. G. Ebin},
       DOI = {10.1016/0001-8708(71)90020-X},
       URL = {https://doi.org/10.1016/0001-8708(71)90020-X},
}

@article {MR182927,
    AUTHOR = {Moser, J\"{u}rgen},
     TITLE = {On the volume elements on a manifold},
   JOURNAL = {Trans. Amer. Math. Soc.},
  FJOURNAL = {Transactions of the American Mathematical Society},
    VOLUME = {120},
      YEAR = {1965},
     PAGES = {286--294},
      ISSN = {0002-9947},
   MRCLASS = {53.45},
  MRNUMBER = {182927},
MRREVIEWER = {R. S. Palais},
       DOI = {10.2307/1994022},
       URL = {https://doi.org/10.2307/1994022},
}

@book {MR598470,
    AUTHOR = {Weinstein, Alan},
     TITLE = {Lectures on symplectic manifolds},
    SERIES = {CBMS Regional Conference Series in Mathematics},
    VOLUME = {29},
      NOTE = {Corrected reprint},
 PUBLISHER = {American Mathematical Society, Providence, RI},
      YEAR = {1979},
     PAGES = {ii+48},
      ISBN = {0-8218-1679-9},
   MRCLASS = {58F05 (47G05 53C15 70Hxx)},
  MRNUMBER = {598470},
}

@book {MR516965,
    AUTHOR = {Guillemin, Victor and Sternberg, Shlomo},
     TITLE = {Geometric asymptotics},
    SERIES = {Mathematical Surveys, No. 14},
 PUBLISHER = {American Mathematical Society, Providence, RI},
      YEAR = {1977},
     PAGES = {xviii+474 pp. (one plate)},
   MRCLASS = {58G15 (35C99 41A60 47G05 58F05 78.58)},
  MRNUMBER = {516965},
MRREVIEWER = {Yu. V. Egorov},
}

@book {MR1853077,
    AUTHOR = {Cannas da Silva, Ana},
     TITLE = {Lectures on symplectic geometry},
    SERIES = {Lecture Notes in Mathematics},
    VOLUME = {1764},
 PUBLISHER = {Springer-Verlag, Berlin},
      YEAR = {2001},
     PAGES = {xii+217},
      ISBN = {3-540-42195-5},
   MRCLASS = {53Dxx (53-01)},
  MRNUMBER = {1853077},
MRREVIEWER = {Brendan J. Foreman},
       DOI = {10.1007/978-3-540-45330-7},
       URL = {https://doi.org/10.1007/978-3-540-45330-7},
}

@incollection {MR2369441,
    AUTHOR = {Cieliebak, Kai and Hofer, Helmut and Latschev, Janko and
              Schlenk, Felix},
     TITLE = {Quantitative symplectic geometry},
 BOOKTITLE = {Dynamics, ergodic theory, and geometry},
    SERIES = {Math. Sci. Res. Inst. Publ.},
    VOLUME = {54},
     PAGES = {1--44},
 PUBLISHER = {Cambridge Univ. Press, Cambridge},
      YEAR = {2007},
      ISBN = {978-0-521-87541-7},
   MRCLASS = {53D35 (37J05 57R17)},
  MRNUMBER = {2369441},
MRREVIEWER = {Michael\ J.\ Usher},
       DOI = {10.1017/CBO9780511755187.002},
       URL = {https://doi.org/10.1017/CBO9780511755187.002},
}

@article {MR3548485,
    AUTHOR = {Groman, Yoel and Solomon, Jake P.},
     TITLE = {{$J$}-holomorphic curves with boundary in bounded geometry},
   JOURNAL = {J. Symplectic Geom.},
  FJOURNAL = {The Journal of Symplectic Geometry},
    VOLUME = {14},
      YEAR = {2016},
    NUMBER = {3},
     PAGES = {767--809},
      ISSN = {1527-5256,1540-2347},
   MRCLASS = {53D05 (32Q65 53D12)},
  MRNUMBER = {3548485},
MRREVIEWER = {Weiyi\ Zhang},
       DOI = {10.4310/JSG.2016.v14.n3.a5},
       URL = {https://doi.org/10.4310/JSG.2016.v14.n3.a5},
}

@article {MR3773793,
    AUTHOR = {Cieliebak, K. and Mohnke, K.},
     TITLE = {Punctured holomorphic curves and {L}agrangian embeddings},
   JOURNAL = {Invent. Math.},
  FJOURNAL = {Inventiones Mathematicae},
    VOLUME = {212},
      YEAR = {2018},
    NUMBER = {1},
     PAGES = {213--295},
      ISSN = {0020-9910,1432-1297},
   MRCLASS = {53D12 (53D05 53D35)},
  MRNUMBER = {3773793},
MRREVIEWER = {Luc\ Vrancken},
       DOI = {10.1007/s00222-017-0767-8},
       URL = {https://doi.org/10.1007/s00222-017-0767-8},
}

@article {MR1957663,
    AUTHOR = {Thomas, R. P. and Yau, S.-T.},
     TITLE = {Special {L}agrangians, stable bundles and mean curvature flow},
   JOURNAL = {Comm. Anal. Geom.},
  FJOURNAL = {Communications in Analysis and Geometry},
    VOLUME = {10},
      YEAR = {2002},
    NUMBER = {5},
     PAGES = {1075--1113},
      ISSN = {1019-8385,1944-9992},
   MRCLASS = {53C38 (32Q25 53C44)},
  MRNUMBER = {1957663},
MRREVIEWER = {Yong-Geun\ Oh},
       DOI = {10.4310/CAG.2002.v10.n5.a8},
       URL = {https://doi.org/10.4310/CAG.2002.v10.n5.a8},
}

@article {MR1362874,
    AUTHOR = {Taubes, Clifford H.},
     TITLE = {{${\rm SW}\Rightarrow{\rm Gr}$}: from the {S}eiberg-{W}itten
              equations to pseudo-holomorphic curves},
   JOURNAL = {J. Amer. Math. Soc.},
  FJOURNAL = {Journal of the American Mathematical Society},
    VOLUME = {9},
      YEAR = {1996},
    NUMBER = {3},
     PAGES = {845--918},
      ISSN = {0894-0347,1088-6834},
   MRCLASS = {57R57 (53C15 58D10 58D27 58G30)},
  MRNUMBER = {1362874},
MRREVIEWER = {Dietmar\ A.\ Salamon},
       DOI = {10.1090/S0894-0347-96-00211-1},
       URL = {https://doi.org/10.1090/S0894-0347-96-00211-1},
}

@article {MR2823871,
    AUTHOR = {Joyce, Dominic and Lee, Yng-Ing and Schoen, Richard},
     TITLE = {On the existence of {H}amiltonian stationary {L}agrangian
              submanifolds in symplectic manifolds},
   JOURNAL = {Amer. J. Math.},
  FJOURNAL = {American Journal of Mathematics},
    VOLUME = {133},
      YEAR = {2011},
    NUMBER = {4},
     PAGES = {1067--1092},
      ISSN = {0002-9327,1080-6377},
   MRCLASS = {53D12 (53D35)},
  MRNUMBER = {2823871},
MRREVIEWER = {Spiro\ Karigiannis},
       DOI = {10.1353/ajm.2011.0030},
       URL = {https://doi.org/10.1353/ajm.2011.0030},
}
\end{document}